\documentclass[12pt]{amsart}
\date{4 December 2022}
\usepackage{amscd,amssymb}
\usepackage{framed}
\usepackage{xcolor}
\usepackage{mathrsfs}
\usepackage{mathdots}
\usepackage{hyperref}
\input xy
\xyoption{all}
\let\oldmarginpar\marginpar
\renewcommand\marginpar[1]{\oldmarginpar{\tiny\bf\begin{flushleft} #1
\end{flushleft}}}

\numberwithin{table}{section}

\newcommand{\xra}{\xrightarrow}
\theoremstyle{plain}  % default
\newtheorem{theorem}{Theorem}[section]

\newtheorem*{theorem*}{Theorem}
\newtheorem{corollary}[theorem]{Corollary}
\newtheorem{lemma}[theorem]{Lemma}
\newtheorem{proposition}[theorem]{Proposition}

\theoremstyle{definition}
\newtheorem{definition}[theorem]{Definition}

\theoremstyle{remark}

\newtheorem{remark}[theorem]{Remark}
\newtheorem*{remark*}{Remark}

\newtheorem*{claim*}{Claim}

\DeclareFontFamily{OT1}{pzc}{}
\DeclareFontShape{OT1}{pzc}{m}{it}{<-> s * [1.10] pzcmi7t}{}
\DeclareMathAlphabet{\mathpzc}{OT1}{pzc}{m}{it}

\DeclareFontFamily{OT1}{rsfs}{}
 \DeclareFontShape{OT1}{rsfs}{n}{it}{<->rsfs10}{}
 \DeclareMathAlphabet{\curly}{OT1}{rsfs}{n}{it}

\newcommand{\cM}{\mathcal{M}}

\newcommand{\CC}{\mathbb{C}}

\newcommand{\Ad}{\operatorname{Ad}}

\newcommand{\Aut}{\operatorname{Aut}}
\newcommand{\Int}{\operatorname{Int}}
\newcommand{\Out}{\operatorname{Out}}

\newcommand{\Id}{\operatorname{Id}}

\newcommand{\PPP}{{\curly P}}

\newcommand{\SSS}{{\curly S}}

\newcommand{\UUU}{{\curly U}}
\newcommand{\VVV}{{\curly V}}

\newcommand{\qu}{/\kern-.7ex/}
\newcommand{\exh}{\to\kern-1.8ex\to}

\newcommand{\VP}{{\curly V}\kern-0.9ex\PPP}

\newcommand{\suchthat}{\;|\;}

\newcommand{\Kto}{\ar@{*+=[o][F]{\scriptscriptstyle{K}}->}[r]}

\newcommand{{\zerocochain}}{g}
\newcommand{{\onecochain}}{f}
\newcommand{\action}{\varphi}

\newcommand{\ug}{\underline G}
\newcommand{\uh}{\underline H}
\newcommand{\ugh}{\underline {G/H}}
\newcommand{\uz}{\underline Z}
\newcommand{\ugz}{\underline {G/Z}}
\newcommand{\cu}[1]{C^{#1}(\UUU,\ug)}

\newcommand{\cugz}[1]{C^{#1}(\UUU,\ugz)}
\newcommand{\cuz}[1]{C^{#1}(\UUU,\uz)}

\newcommand{\zuc}{Z^1_{\Gamma,\theta,c}(\UUU,\ug)}
\newcommand{\huc}{H_{\Gamma,\theta,c}^{1}(\UUU,\ug)}
\newcommand{\hxc}{H_{\Gamma,\theta,c}^{1}(X,\ug)}
\newcommand{\thxc}{\widetilde H_{\Gamma,\theta,c}^{1}(X,\ug)}
\newcommand{\su}{S(\UUU,G,\Gamma,\theta,c)}
\newcommand{\tsu}{\widetilde{S}(\UUU,G,\Gamma,\theta,c)}
\newcommand{\sv}{S(\VVV,G,\Gamma,\theta,c)}

\newcommand{\sx}{S(X,G,\Gamma,\theta,c)}
\newcommand{\tsx}{\widetilde{S}(X,G,\Gamma,\theta,c)}
\newcommand{\cgamma }[1]{#1^{\gamma}}
\newcommand{\clambda}[1]{#1^{\lambda}}

\newcommand{\rgamma}{\rtimes Z(\Gamma)}
\DeclareMathOperator{\bg}{B\Gamma}
\DeclareMathOperator{\eg}{E\Gamma}

\newcommand{\zu}{Z^1_{\Gamma,\theta}(\UUU,\ug)}

\newcommand{\zugz}{Z^1_{\Gamma,\theta}(\UUU,\ugz)}
\newcommand{\zuz}{Z^1_{\Gamma,\theta}(\UUU,\uz)}

\newcommand{\hu}{H_{\Gamma,\theta}^{1}(\UUU,\ug)}

\newcommand{\huz}{H_{\Gamma,\theta}^{1}(\UUU,\uz)}

\newcommand{\thuc}{\widetilde H_{\Gamma,\theta,c}^{1}(\UUU,\ug)}

\newcommand{\hhu}{H_{\Gamma,\theta}^{0}(\UUU,\ug)}
\newcommand{\hx}{H_{\Gamma,\theta}^{1}(X,\ug)}
\newcommand{\hxh}{H_{\Gamma,\theta}^{1}(X,\uh)}

\newcommand{\thx}{\widetilde H_{\Gamma,\theta}^{1}(X,\ug)}

\newcommand{\thxgh}{\widetilde H_{\Gamma,\theta}^{1}(X,\ugh)}
\newcommand{\hxgz}{H_{\Gamma,\theta}^{1}(X,\ugz)}
\newcommand{\thxgz}{\widetilde H_{\Gamma,\theta}^{1}(X,\ugz)}
\newcommand{\hxz}{H_{\Gamma,\theta}^{1}(X,\uz)}
\newcommand{\hhx}{H_{\Gamma,\theta}^{0}(X,\ug)}
\newcommand{\hhxh}{H_{\Gamma,\theta}^{0}(X,\uh)}
\newcommand{\hhxgh}{H_{\Gamma,\theta}^{0}(X,\ugh)}
\newcommand{\hhxgz}{H_{\Gamma,\theta}^{0}(X,\ugz)}
\newcommand{\hhxz}{H_{\Gamma,\theta}^{0}(X,\uz)}

\newcommand{\cc}[1]{\mathcal C_{#1}}
\newcommand{\ff}{\mathcal F}

\newcommand{\tg}{\theta_{\gamma}}
\newcommand{\ot}{\overline{\theta}}

\DeclareMathOperator{\Fun}{Fun}
\newcommand{\fgg}[1]{{\action}_{\gamma,#1}}

\newcommand{\Z}{\mathbb{Z}}

\newcommand{\andd}{\quad\text{and}\quad}

\newcommand{\class}[1]{\left[#1\right]}

\newcommand{\forevery}{\;\text{for}\;\text{every}\;}

\newcommand{\cMg}{\mathcal{M}'_{\Gamma'}}
\newcommand{\nog}{N_{\Gamma}(\Gamma')}
\newcommand{\hatg}{\hat{G}}
\newcommand{\hatgg}{\hat{G}'}

\newcommand{\iso}{\mathcal{I}so_G}
\newcommand{\g}{\mathcal{G}}

\DeclareMathOperator{\ob}{ob}

\title[Non-connected Lie groups, twisted equivariant bundles and coverings]
{Non-connected Lie groups, twisted equivariant bundles and coverings}

\author{G. Barajas, O. Garc\'{\i}a-Prada, P.~B. Gothen, I. Mundet i Riera}



\thanks{The first author received the support of a fellowship from ”la Caixa” Foundation (ID 100010434 and code  LCF/BQ/DI19/11730022). The second author was partially supported by the Spanish MICINN under
the ICMAT Severo Ochoa grant CEX2019-000904-S, and grants
PID2019-109339GB-C31  and EIN2020-112392.
The third author was partially supported by CMUP (UIDB/MAT/00144/2020)
and the project EXPL/MAT-PUR/1162/2021
  funded by FCT (Portugal) with national funds.
The fourth author was partially supported by the grant
PID2019-104047GB-I00 from the Spanish MICINN}

\begin{document}

\begin{abstract}
  Let  $\Gamma$ be a finite group acting on a Lie group $G$. We consider a class of  group extensions   $1 \to G \to \hat{G} \to \Gamma \to 1$  defined by this action and a $2$-cocycle of $\Gamma$ with values in the centre of $G$.
  We establish and study a
  correspondence between $\hat{G}$-bundles on a manifold  and twisted
  $\Gamma$-equivariant
  bundles with  structure group $G$  on a suitable Galois  $\Gamma$-covering
  of the manifold. We also describe this correspondence in terms of non-abelian cohomology. Our results  apply, in particular, to the case of a compact or
  reductive complex Lie group $\hat{G}$, since  such a group is always isomorphic to an extension as above, where $G$ is the connected component of the identity and $\Gamma$ is the group of connected components of $\hat{G}$.
\end{abstract}

\maketitle

% \newpage

% \tableofcontents

%%%%%%%%%%%%%%%%%%%%%%%
\section{Introduction}
%%%%%%%%%%%%%%%%%%%%%%%

In this paper,  we are concerned with the geometry of bundles with  non-connected
structure group. For the sake of
concreteness,  we present our main results here in  the  smooth
category, but  the analogous results  are also valid in the topological category (under some mild assumptions which we will make precise), as well as  in the holomorphic category.

Let $\hat{G}$ be  a non-connected Lie group. This group fits in an extension of groups
\begin{equation}
 \label{extension}
 1 \to G \to \hat{G} \xra{q} \Gamma \to 1
\end{equation}
where $G$ is the connected component containing the identity, and $\Gamma$ is the group of connected components of $\hat{G}$. We will assume that $\Gamma$ is finite,  although some of our results apply more generally.

Let $X$ be a manifold and $E$ be a principal  $\hat{G}$-bundle over $X$. The quotient of $E$ by $G$ defines a principal $\Gamma$-bundle $Y$ over $X$, that is, a Galois covering $Y\to X$ with Galois group $\Gamma$. It is very natural to try to reduce the study of the geometry of  $\hat{G}$-bundles over $X$ to that of $G$-bundles over $Y$, leading to a situation where the structure group is connected. We will pursue this here in the case in which the
characteristic homomorphism
$\bar{\theta}\colon\Gamma \to \Out(G)$ of the extension
\eqref{extension}, where
$\Out(G)=\Aut(G)/\Int(G)$,  has a lift to a homomorphism $\theta\colon \Gamma \to \Aut(G)$. Here $\Aut(G)$ is the group of automorphisms of $G$ and $\Int(G)$ is the subgroup  of $\Aut(G)$ consisting of  inner automorphisms.
In this situation the group $\hat{G}$ is isomorphic to a group, which we denote by  $G\times_{(\theta,c)} \Gamma$, whose underlying set is the Cartesian product  $G\times  \Gamma$, and the group operation is  given by

$$
(g_1,\gamma_1)\cdot(g_2,\gamma_2)
  =(g_1{g_2}^{\gamma_1} c(\gamma_1,\gamma_2),\gamma_1\gamma_2),
$$
where  $c$ is a $2$-cocycle of $\Gamma$ with values in the centre $Z(G)$ of $G$. Here the action of $\Gamma$ on $Z(G)$ is defined by $\theta$, but clearly only depends  on  $\bar\theta$.

By  a result of de Siebenthal  \cite{siebenthal:1956}, the lift $\theta$ always exists in the case in which  $\hat{G}$ is a compact or a complex reductive
Lie group --- this is the situation that  has primarily  motivated our work, as we will explain below.
Under the assumption of the existence of the lift $\theta$, we are able to identify the  $G$-bundles on $Y$ corresponding to  $\hat{G}$-bundles on $X$ with covering $Y$. These are $G$-bundles on $Y$ equipped with what we call a $(\theta,c)$-twisted
$\Gamma$-equivariant structure (see Definition
\ref{def:gamma-twisted-principal-bundle} for details). Notice that the twisting here refers to both,  the fact that $\Gamma$ acts on $G$ via $\theta$, and the presence of the $2$-cocycle $c$ of $\Gamma$ with values in the centre of $G$.

After establishing  in Section \ref{lie-theory} some facts on the structure of  non-connected Lie groups and their actions, in Section \ref{sec:twist}, we define and study twisted equivariant structures on principal bundles and  associated bundles. With this in hand, in Section \ref{coverings-twist}, we establish a categorical correspondence between  principal $\hat{G}$-bundles on $X$ and  $(\theta,c)$-twisted
$\Gamma$-equivariant $G$-bundles on a suitable covering $Y\to X$ (see Proposition \ref{prop:G-Ghat-equivalence} and Theorem \ref{category-monodromy}).

In Section \ref{non-abelian-cohomology}, we analyse the problem of identifying which $\hat{G}$-bundles over $X$ determine  certain  $\Gamma$-coverings of $X$ from the perspective of non-abelian sheaf cohomology. An answer to this problem is given by a  general result of Grothendieck \cite{grothendieck},  which identifies the set in question with the first cohomology set of a certain  non-abelian sheaf of groups.
While Grothendieck's result does not require, in particular,  the existence of a lift $\theta$ of the characteristic homomorphism of the extension (\ref{extension}), in the case where such a lift exists,  Grothendieck's sheaf of groups has a particularly nice description (see Proposition \ref{prop-grothendieck}).

In view of the results in Section \ref{coverings-twist}, it seems
clear that the first cohomology set of the Grothendieck sheaf over $X$
described in Proposition \ref{prop-grothendieck} must be related to the  twisted equivariant bundles over a covering $Y$ of
$X$. This relation is treated in Section
\ref{twisted-equivariant-cohomology}.  We first give a cohomological
description of the set of equivalence classes of twisted
$\Gamma$-equivariant $G$-bundles and establish the natural long exact
sequences in twisted equivariant cohomology.  In particular, the long
exact sequence analysis allows us to give a criterion for existence of
twisted equivariant bundles (Theorem \ref{prop-long-exact-seq-Y}). In
Theorem \ref{prop-iso-karoubi-ghat-bundles} we finally establish the
bijection between the twisted equivariant cohomology set on the Galois covering $Y$ and
the cohomology set of the Grothendieck sheaf over $X$ mentioned above.

It is important to point out that the notion of twisted
$\Gamma$-equivariant principal bundle over a manifold $X$ acted upon
by $\Gamma$, given in Definition
\ref{def:gamma-twisted-principal-bundle}, does not require that the
action of $\Gamma$ on $X$ be free (notice that here in the notation we
are changing the roles of $X$ and $Y$). In fact the twisted
non-abelian sheaf cohomology analysis given in Sections
\ref{section-definitions}, \ref{reduced-twisted-cohomology} and
\ref{twisted-long-exact-sequences} is valid without the assumption of
$\Gamma$ acting freely on $X$. However, the relation of  twisted
$\Gamma$-equivariant $G$-bundles on $X$ with objects on the quotient
$Y=X/\Gamma$ when $\Gamma$ does not act freely is much more involved
since, in particular, now $Y$ is generally singular.

Twisted equivariant principal $G$-bundles  have been considered in the literature in connection with the study of fixed points under the action of a finite group $\Gamma$ on moduli spaces of  $G$-bundles and $G$-Higgs bundles on a compact Riemann surface $X$, for a complex reductive group $G$. When $\Gamma$ is the group $\Z/2$ generated by an antiholomorphic involution of $X$ and $\Gamma$ acts on $G$ anti-holomorphically, twisted   equivariant principal $G$-bundles coincide with the pseudo-real $G$-bundles studied in \cite{BG,BGH1,BGH2}. For other groups $\Gamma$ these objects have been considered in \cite{basu-garcia-prada,BCFG,BCG,garcia-prada-wilkin}.

We found out recently that twisted equivariant principal $G$-bundles (without the extra twisting given by a 2-cocycle) have been studied in the topological category by  T. tom Dieck in \cite{tom-dieck}. It seems that one of his motivations was related to the notion of real holomorphic vector bundle that had been introduced  a bit earlier by Atiyah \cite{atiyah}. In fact the notion of  pseudo-real $G$-bundle mentioned above  is the generalisation of Atiyah's real  holomorphic vector bundle to the context of holomorphic principal bundles.  Related  objects have also appeared in the work of Balaji--Seshadri \cite{balaji-seshadri} in connection to parahoric bundles, and Donagi--Gaitsgory and others  in the description of the Hitchin fibration for $G$-Higgs bundles in terms of generalised Prym varieties \cite{donagi-gaitsgory}. In the process of completing this paper we came across  recent work by Damiolini \cite{damiolini} where a description of twisted equivariant principal $G$-bundles (without the extra twisting given by a 2-cocycle) is given in terms of parahoric bundles in the algebraic category. In Section \ref{torsors} we briefly comment on the relation of our work to her approach.

This paper emerged from a plan  to generalize Hitchin--Kobayashi and non-abelian Hodge correspondences to the case in which the structure group is non-connected. While there is an extensive literature on the connected case,
to our knowledge, very little has been written when the group is non-connected, even for the case of principal bundles over a Riemann surface. Some work to this end  is carried out in \cite{biswas-gomez}. It is important to point out that even in the study of moduli spaces where the starting structure group is connected, there are objects  with non-connected structure group that emerge naturally.
This happens for example when studying  Cayley correspondences in the context of higher Teichm\"uller theory \cite{BCGGO} or in the  study of fixed points in the moduli space of bundles under the action of a finite group. Such a situation appears when one considers  a non-trivial finite order element of the set of $Z(G)$-bundles, where $Z(G)$ is the centre of a reductive group $G$, acting by tensorisation on the moduli space of $G$-bundles. The fixed points under this action generally  reduce their structure  group to a non-connected group
(see \cite{garcia-prada-ramanan}), and our approach in this paper has been exploited to give a Prym--Narasimhan--Ramanan type construction of fixed points in the moduli space of $G$-bundles and
$G$-Higgs bundles  (see \cite{barajas-garcia-prada1,barajas-basu-garcia-prada}). Non-connected structure groups appear also in the study of fixed points under the action of $\CC^\ast$ on the moduli space of $G$-Higgs bundles, the so-called Hodge bundles \cite{BCGT}.

In forthcoming papers \cite{GGM1,GGM2}, the theory developed here is used to give a  Hitchin--Kobayashi correspondence for Higgs pairs and a non-abelian Hodge correspondence on compact Riemann surfaces
when the structure group is non-connected.
An important ingredient that allows to reduce the necessary existence theorems to the connected group case is the fact that the bundles associated to a twisted
$\Gamma$-equivariant principal bundle relevant for these
correspondences acquire actually a true $\Gamma$-equivariant structure
and the analysis and geometry reduce basically to a $\Gamma$-invariant
version of those appearing in the connected case
\cite{garcia-prada-gothen-mundet:2009}.

We thank the referee for useful comments and corrections.

\newpage

%%%%%%%%%%%%%%%%%%%%%%%%%%%%%%%%%%%%%%%%%%%%%%%%%%%%%%%%%%%%%%%%%%%%
\section{Non-connected Lie groups and their actions}
\label{lie-theory}
%%%%%%%%%%%%%%%%%%%%%%%%%%%%%%%%%%%%%%%%%%%%%%%%%%%%%%%%%%%%%%%%%%%%

%%%%%%%%%%%%%%%%%%%%%%%%%%%%%%%%%%%%%%%%%%%%%%%%%
\subsection{Extensions of Lie groups}
\label{sec:extensions-lie}
%%%%%%%%%%%%%%%%%%%%%%%%%%%%%%%%%%%%%%%%%%%%%%%%%

A convenient reference for the material in this section is Hilgert and
Neeb \cite[Chap.~18]{hilgert-neeb:2012}.

Let $G$ be Lie group (not necessarily connected) and let $\Gamma$ be
a discrete group. Consider an extension
\begin{equation}
  \label{eq:extension-Gamma-G0}
  1 \to G \to \hat{G} \xra{q} \Gamma \to 1.
\end{equation}
The adjoint action of $\hat{G}$ on the normal subgroup $G$ defines the
\textbf{characteristic homomorphism}
$\bar{\theta}\colon\Gamma \to \Out(G)$ of the extension
\eqref{eq:extension-Gamma-G0} to the outer automorphism group
$\Out(G)=\Aut(G)/\Int(G)$. We will assume that there is a lift of
$\bar{\theta}$ to a homomorphism $\theta\colon \Gamma \to \Aut(G)$
making the diagram
\begin{equation}\label{lift}
  \xymatrix{
    \Gamma \ar[r]^(.35){\theta} \ar[dr]_{\bar\theta} & \Aut(G) \ar[d] \\
    & \Out(G)
  }
\end{equation}
commutative. We often write
\begin{math}
  \theta_\gamma
\end{math}
for $\theta(\gamma)$ and
\begin{math}
  g^\gamma=\theta_\gamma(g)
\end{math}
for the action of $\gamma\in\Gamma$ on $g\in G$.

Under the assumption of the existence of the lift $\theta$  in (\ref{lift}), equivalence
classes of extensions \eqref{eq:extension-Gamma-G0} with
characteristic
homomorphism $\bar{\theta}$ are classified by the second
cohomology group $H^2_{\theta}(\Gamma,Z(G))$, where we recall that a \textbf{$2$-cochain} $c\in
C^2_{\theta}(\Gamma,Z(G))$ is a map
\begin{displaymath}
  c\colon\Gamma\times\Gamma\to Z(G)
\end{displaymath}
which satisfies $c(\gamma,1)=c(1,\gamma)=1$ for all $\gamma\in\Gamma$
and that the subgroup $Z^2_{\theta}(\Gamma,Z(G))$ of
\textbf{$2$-cocycles} consists of those $c$ which satisfy the \textbf{cocycle
condition}
\begin{equation}\label{cocycle-condition}
  c(\gamma_1,\gamma_2)^{\gamma_0}c(\gamma_0,\gamma_1\gamma_2)
  =c(\gamma_0,\gamma_1)c(\gamma_0\gamma_1,\gamma_2),
\end{equation}
for $\gamma_0,\gamma_1,\gamma_2\in\Gamma$. The extension corresponding to
$c$ is $G\times\Gamma$ as a set, with the product given by
\begin{equation}
  \label{eq:cocycle-product}
  (g_1,\gamma_1)\cdot(g_2,\gamma_2)
  =(g_1{g_2}^{\gamma_1} c(\gamma_1,\gamma_2),\gamma_1\gamma_2).
\end{equation}
We write $G\times_{(\theta,c)} \Gamma$ for the Cartesian product
with the group structure \eqref{eq:cocycle-product}, and the extension
\begin{displaymath}
  1 \to G \to G\times_{(\theta,c)} \Gamma \to \Gamma \to 1
\end{displaymath}
is given by the obvious maps $g\mapsto(g,1)$ and
$(g,\gamma)\mapsto\gamma$.

In order to realise a given extension \eqref{eq:extension-Gamma-G0} in
this way,
note that for $\gamma\in\Gamma$, the image of the map
$q^{-1}(\gamma)\to\Aut(G)$ given by conjugation corresponds to the whole
class $\ot(\gamma)\in\Out(G)=\Aut(G)/\Int(G)$. Therefore, given a lift $\theta$ of $\ot$, we may
choose a section $\Gamma\to\hat G$ of $q$ whose composition with
$\hat G\to \Aut(G)$ is equal to $\theta$. We call such a section a
\textbf{normalised section}. Thus the action of $\theta_\gamma$ on $G$
can be written
  \begin{equation}
      \label{eq:Gamma-acts-conj}
    \theta_\gamma(g)=s(\gamma)gs(\gamma)^{-1}
  \end{equation}
  for $\gamma\in\Gamma$ and $g\in G$. In other words
  $\theta_\gamma = \Ad_{s(\gamma)}\colon G\to G$.

% \color{red}[What's this? I guess you
% need to mention the following: let $\gamma\in\Gamma$ and $p:\hat
% G\to\Gamma$ be the projection. Then the image of the map
% $p^{-1}(\gamma)\to\Aut(G)$ given by conjugation is the whole class
% $\ot(\gamma)$. Therefore, given a lift $\theta$ of $\ot$, we may
% choose a section $\Gamma\to\hat G$ of $p$ whose composition with
% $\hat G\to \Aut(G)$ is equal to $\theta$. We call such a section a
% normalized section.]\color{black}
% {\color{blue}  ---  rephrased slightly your suggestion}

Then the cocycle $c$ is given by
\begin{displaymath}
  s(\gamma_1)s(\gamma_2) = c(\gamma_1,\gamma_2)s(\gamma_1\gamma_2)
\end{displaymath}
and there is an isomorphism
\begin{equation}
  \label{eq:1}
  \begin{aligned}
    G\times_{(\theta,c)} \Gamma &\xra{\cong}\hat{G}\\
    (g,\gamma) &\mapsto gs(\gamma).
  \end{aligned}
\end{equation}
Under this isomorphism, the normalised section $s$ is
\begin{equation}
\label{eq:normalised-section}
\begin{aligned}
  s\colon \Gamma &\to \hat{G}, \\
  \gamma&\mapsto (1,\gamma)
\end{aligned}
\end{equation}
Note also that the cocycle $c$ measures the failure of $s$ to be a group
homomorphism.
In particular, the trivial cocycle corresponds to the
split extension, i.e., the semidirect product
$\hat{G} = G\rtimes_{\theta}\Gamma$ defined by $\theta$.

It is useful to note that for $g\in G$ and
$\gamma\in\Gamma$ we have
\begin{equation}\label{inverse}
(1,\gamma)^{-1}=((c(\gamma,\gamma^{-1})^{-1})^{\gamma^{-1}},\gamma^{-1})=(c(\gamma^{-1},\gamma)^{-1},\gamma^{-1})
\end{equation}
A second useful observation is the following. Since $\Gamma$ acts
on $Z(G)$ via ${\theta}$ and the abelian group $Z(G)$ is its
own centre, we may use the same construction to obtain an extension
\begin{displaymath}
  1 \to Z(G) \to \hat{\Gamma}_{\theta,c} \to \Gamma \to 1,
\end{displaymath}
where
\begin{equation}
  \label{eq:Gamma-Z-extension}
  \hat{\Gamma}_{\theta,c} = Z(G)\times_{({\theta},c)}\Gamma \subset \hat{G}
\end{equation}
and we have the
commutative diagram
\begin{equation}
  \label{eq:ZG-G-Gamma-extension}
  \begin{CD}
    1 @>>> Z(G) @>>> \hat{\Gamma}_{\theta,c} @>>> \Gamma @>>> 1 \\
    &&@VVV @VVV @VVV\\
    1 @>>> G @>>> \hat{G} @>>> \Gamma @>>> 1 \ .
  \end{CD}
\end{equation}

%\comment{We don't seem to need the following about 1-cocycles. I leave
 % it for now but we may remove it in the end.}

Finally, for completeness, we recall that a
\textbf{1-cochain}
$a\in B^1_{\theta}(\Gamma,Z(G))$ is a
map $a\colon\Gamma\to Z(G)$ and its coboundary
is $\delta a\in Z^2_{\theta}(\Gamma,Z(G))$ given by
\begin{displaymath}
  \delta a(\gamma_0,\gamma_1) =
  a(\gamma_1)^{\gamma_0} a(\gamma_0\gamma_1)^{-1}a(\gamma_0);
\end{displaymath}
note that the group of 1-cocycles is precisely the kernel of $\delta$.
Two extensions defined by cocycles which differ by a coboundary
$\delta a$ are isomorphic, the isomorphism being given explicitly by
\begin{displaymath}
  (g,\gamma) \mapsto (ga(\gamma),\gamma).
\end{displaymath}
We define $H^2_{\theta}(\Gamma,Z(G)):=Z^2_{\theta}(\Gamma,Z(G))/\delta B^1_{\theta}(\Gamma,Z(G))$, the \textbf{second Galois cohomology group of $\Gamma$ with values in $Z(G)$}. See \cite{hilgert-neeb:2012, serre-galois} for more information on this object.

\begin{remark}
  We note the obvios fact that the action of
$\Gamma$ on the centre $Z(G)$ of $G$ does not depend on the choice
of the lift $\theta$, in other words, $\bar\theta$ defines a canonical
action of $\Gamma$ on $Z(G)$. Thus objects such as
$\hat{\Gamma}_{\theta,c}$ and $H^2_{\theta}(\Gamma,Z(G))$ defined in
terms of this action really only depend on $\bar\theta$.
\end{remark}

%%%%%%%%%%%%%%%%%%%%%%%%%%%%%%%%%%%%%%%%%%%%%%%%%%%%%%%%%%%%%%%%%%%
\subsection{Non-connected compact and complex reductive Lie groups}
\label{non-connected reductive}
%%%%%%%%%%%%%%%%%%%%%%%%%%%%%%%%%%%%%%%%%%%%%%%%%%%%%%%%%%%%%%%%%%%

We will be specially interested in the case in which $G$ is a compact or complex reductive Lie group
with identity component $G_0$ and $\Gamma=\pi_0(G)$ is the group
of components. In particular, this means that $G$ has a maximal compact
subgroup $K$ which meets all its connected components, and thus $\Gamma$ is finite.
By a theorem of
de Siebenthal \cite{siebenthal:1956}, the exact sequence
\begin{displaymath}
  1 \to \Int(K_0) \to \Aut(K_0) \to \Out(K_0) \to 1
\end{displaymath}
splits. Hence, by composing the characteristic homomorphism $\bar{\theta}\colon\Gamma \to \Out(K_0)$  with a splitting homomorphism
$\Out(K_0)\to \Aut(K_0)$, we obtain a lift of $\bar{\theta}$
to a homomorphism $\theta\colon \Gamma \to
\Aut(K_0)$,
and we can choose an associated normalised section giving
us a cocycle $c$.
Moreover, by the universal property of the
complexification of a compact Lie group, $\theta$ can also be viewed
as a lift of the characteristic homomorphism $\bar{\theta}\colon\Gamma
\to \Out(G_0)$, so that we have compatible commutative diagrams
(\ref{lift}) for the groups $K_0$ and $G_0$.
Note that $Z(K_0)\subset Z(G_0)$, again by the universal
property of the complexification and, hence, $c$ can also be viewed as
a cocycle for the $\Gamma$-action on $G_0$.
We thus have the following.

\begin{proposition}\label{group-structure}
Let $G$ be a complex reductive Lie group with identity component
$G_0$ and maximal compact subgroup $K$. Then $G$ is isomorphic to the group
$G_0\times_{(\theta,c)}\Gamma$ for $\theta$ and $c$ as above, and
$K$ can be identified with the subgroup
$K_0\times_{(\theta,c)}\Gamma$. \qed
\end{proposition}

\begin{remark}
  \label{rem:K0-Gamma-invariant}
Note, in particular, that the subgroup $K_0\subset G_0$ is
$\Gamma$-invariant. This is because the lift
$\theta$ of the characteristic homomorphism for the extension $G_0\to G\to\Gamma$
comes from a lift for the extension $K_0\to K\to\Gamma$, and is therefore compatible
with it.
\end{remark}

%%%%%%%%%%%%%%%%%%%%%%%%%%%%%%%%%%%%%%%%%%%%%%%%%
\subsection{Actions of extensions}
\label{sec:actions-extensions-lie}
%%%%%%%%%%%%%%%%%%%%%%%%%%%%%%%%%%%%%%%%%%%%%%%%%

Our objective in this section is to describe actions of a Lie group
$\hat{G}$ given as an extension as in Section~\ref{sec:extensions-lie}
in terms of what will be called \emph{$(\theta,c)$-twisted
  $(\Gamma,G)$-actions}.  For simplicity we shall state definitions
and results simply in terms of actions on sets. We leave to the reader
to make the obvious modifications for smooth actions on
manifolds, holomorphic actions on complex manifolds, linear actions
on vector spaces, etc.

Let $\Gamma$ be a group which acts (on the left) on a group $G$ via
$\theta\colon\Gamma\to\Aut(G)$; we write $g^\gamma =
\theta_\gamma(g)$ where
$\theta_\gamma=\theta(\gamma)$. Let $Z=Z(G)$ be the centre of $G$
and let $c\colon\Gamma\times\Gamma\to Z$
be a cocycle for the $\Gamma$-action on $Z$ induced by $\theta$.

\begin{definition}
  \label{def:G-Gamma-twisted-action}
  A \textbf{$(\theta,c)$-twisted left (respectively, right) $(G,\Gamma)$-action} on a
  set $M$ is given by the following data:
\begin{itemize}
\item a left (respectively, right) action of $G$ on $M$, for
  $g\in G$ and $m\in M$ written $m\mapsto g\cdot m$ in the case of a
  left action, and $m\mapsto m\cdot g$ in the case of a right action.
  \end{itemize}
  Moreover, in the case of a \emph{left action},
  \begin{itemize}
  \item a map
    \begin{math}
      \Gamma\times M\to M, \; (\gamma,m) \mapsto \gamma\cdot m
    \end{math}
    satisfying
    \begin{align*}
      1\cdot m&=m, &m\in M, \tag{\textit{i}}\\
      \gamma\cdot(g\cdot m) &= g^\gamma\cdot(\gamma\cdot m),
                   & g\in G,\;\gamma \in \Gamma, \tag{\textit{ii}}\\
      \gamma_1\cdot(\gamma_2\cdot m)
              &=c(\gamma_1,\gamma_2)
                \cdot((\gamma_1\gamma_2)\cdot m),
                   & m\in M,\;\gamma_1,\gamma_2\in\Gamma, \tag{\textit{iii}}
    \end{align*}
  \end{itemize}
  while in the case of a \emph{right action},
  \begin{itemize}
  \item a map
    \begin{math}
      M\times \Gamma \to M, \; (m,\gamma) \mapsto m\cdot \gamma
    \end{math}
    satisfying
    \begin{align*}
      m\cdot 1&=m, &m\in M,\tag{\textit{i}}\\
      (m\cdot g^\gamma)\cdot \gamma &= (m\cdot \gamma)\cdot g,
                   &g\in G,\;\gamma \in \Gamma,\tag{\textit{ii}}\\
      (m\cdot \gamma_1)\cdot \gamma_2
              &=(m\cdot c(\gamma_1,\gamma_2))\cdot(\gamma_1\gamma_2),
                   &m\in M,\;\gamma_1,\gamma_2\in\Gamma. \tag{\textit{iii}}
    \end{align*}
  \end{itemize}
\end{definition}

\begin{remark}
  Sometimes we shall be interested in the case when $G=Z$ is
  abelian. In that case we usually speak simply of a \textbf{$(\theta,c)$-twisted
    $\Gamma$-action}. In particular, given a $(\theta,c)$-twisted
    $(G,\Gamma)$-action, there is an associated $(\theta,c)$-twisted
  $\Gamma$-action, obtained by restricting the action of $G$ to its
  centre $Z$. In this situation, notice that if the subgroup of $Z$ generated by
  the image of $c$ acts trivially on $M$, then  a $(\theta,c)$-twisted
  $\Gamma$-action is  a true  $\Gamma$-action.
\end{remark}

\begin{remark}
  There are several equivalent ways of formulating the conditions of
  Definition~\ref{def:G-Gamma-twisted-action}. For example, using
  condition (\textit{ii}) for a twisted left action, condition (\textit{iii}) can be
  written as
  \begin{displaymath}
    \gamma_1\cdot(\gamma_2\cdot m)=
    (\gamma_1\gamma_2)\cdot(c(\gamma_1,\gamma_2)^{(\gamma_1\gamma_2)^{-1}}
    \cdot m).
  \end{displaymath}
  Similarly, the condition (\textit{iii}) for a right action can written as
  \begin{displaymath}
      (m\cdot \gamma_1)\cdot \gamma_2
        =(m\cdot(\gamma_1\gamma_2))\cdot
        c(\gamma_1,\gamma_2)^{(\gamma_1\gamma_2)^{-1}}.
  \end{displaymath}
\end{remark}

% \comment{In fact, we need a mixed situation where the $G$-action on
% $M$ is on the left and the $\Gamma$-action on $M$ is on the right!
% This is necessary to construct associated bundles later on in
% Section \ref{associated-bundles}.}

Assume now that $\Gamma$ is discrete, and let $\hat{G} = G\times_{(\theta,c)} \Gamma$
be the extension of $\Gamma$ by $G$ determined by $c$ and $\theta$ as
in the preceding section.

% \begin{remark}
%   \label{rem:twisted-left-right-conversion}
%   Suppose we have a left $(\theta,c)$-twisted left $(G,\Gamma)$-action
%   on $M$. Then we can define a right $G$ action on $M$ in the standard
%   way, by $m\cdot g :=g^{-1}\cdot m$ and a twisted right
%   $\Gamma$-action on $M$ by $m\cdot\gamma:=\gamma^{-1}\cdot m$. Let us
%   check that in this way we obtain a left $(\theta,c)$-twisted left $(G,\Gamma)$-action
%   on $M$. To check Condition (\textit{ii}) we have
%   \begin{align*}
%     (m\cdot g^\gamma)\cdot \gamma
%     &= \gamma^{-1}\cdot((g^\gamma)^{-1}\cdot m)\\
%     &= \gamma^{-1}\cdot((g^{-1})^{\gamma}\cdot m)\\
%     &= g^{-1}\cdot(\gamma^{-1}\cdot m)\\
%     &= (m\cdot \gamma)\cdot g,
%   \end{align*}
%   as desired. For condition (\textit{iii}) DOESN'T HOLD we have:
%   \begin{align*}
%     (m\cdot \gamma_1)\cdot \gamma_2
%     &= \gamma_2^{-1}\cdot(\gamma_1^{-1}\cdot m)\\
%     &= c(\gamma_2^{-1},\gamma_1^{-1})
%                 \cdot((\gamma_2^{-1}\gamma_1^{-1})\cdot m),\\
%     &= c(\gamma_2^{-1},\gamma_1^{-1})
%                 \cdot(((\gamma_1\gamma_2)^{-1})\cdot m),\\
%                   &\neq(m\cdot c(\gamma_1,\gamma_2))\cdot(\gamma_1\gamma_2)
%   \end{align*}
% \end{remark}

\begin{proposition}
\label{prop:-G-Gamma-twisted-actions}
There is a bijective correspondence between
$\hat{G}$-actions on a set $M$ and
$(\theta,c)$-twisted $(G,\Gamma)$-actions on $M$,
given as follows:
\begin{itemize}
\item Given an action of $\hat{G}$ on $M$, the action of the subgroup
  $G\subseteq \hat{G}$ on $M$ is defined by restriction, and the action of
  $\gamma\in \Gamma$ on $M$ is defined to be the action of
  $(1,\gamma)\in \hat{G}$.
\item Given a $(\theta,c)$-twisted $(G,\Gamma)$-action on $M$, in
  the case of a left action, we define the $\hat{G}$-action on $M$ by
  \begin{displaymath}
    (g,\gamma)\cdot m = g\cdot(\gamma\cdot m)
  \end{displaymath}
  while, in the case of a right action, we define
  \begin{displaymath}
    m\cdot(g,\gamma) = (m\cdot g)\cdot\gamma.
  \end{displaymath}
\end{itemize}
\end{proposition}

\begin{proof}
  Suppose we have an action of $\hat{G}$ on $M$. Let us check the
  conditions of Definition~\ref{def:G-Gamma-twisted-action} with the
  actions of $G$ and $\Gamma$ defined as in the statement of the
  proposition. Let $g\in \hat{G}$ and
  $\gamma,\gamma_1,\gamma_2\in\Gamma$. Then we have identities in
  $\hat{G}$,
  \begin{align*}
    &(1,\gamma)(g,1) = (g^\gamma,\gamma)
      = (g^\gamma,1)(1,\gamma)\quad\text{and}\\
    &(1,\gamma_1)(1,\gamma_2) = (c(\gamma_1,\gamma_2),\gamma_1\gamma_2)
      = (c(\gamma_1,\gamma_2),1) (1,\gamma_1\gamma_2),
  \end{align*}
  which show that the conditions for a $(G,\Gamma)$-twisted left
  action are satisfied when $\hat{G}$ acts on the left.
    The same identities show that the conditions for a
    $(G,\Gamma)$-twisted right action are satisfied when $\hat{G}$ acts
    on the right.

  Conversely, suppose we have a $(G,\Gamma)$-twisted
  action on $M$, and define the action of $\hat{G}$ as in the statement of
  the proposition (it is worth noting that this definition is forced
  upon us by the identity
  \begin{math}
    (g,\gamma) = (g,1)(1,\gamma)
  \end{math}).
  We must check that this in
  fact defines an honest $\hat{G}$-action. In the case of a
  left action we have, for $g_1,g_2\in G$ and $\gamma_1,\gamma_2\in\Gamma$:
  \begin{align*}
    (g_1,\gamma_1)\cdot((g_2,\gamma_2)\cdot m)
      &=(g_1,\gamma_1)\cdot (g_2\cdot(\gamma_2\cdot m))\\
      &=g_1\cdot(\gamma_1\cdot(g_2\cdot(\gamma_2\cdot m)))\\
      &=g_1\cdot(g_2^{\gamma_1}\cdot(\gamma_1\cdot(\gamma_2\cdot m)))\\
      &=g_1\cdot(g_2^{\gamma_1}\cdot(c(\gamma_1,\gamma_2)
        \cdot ((\gamma_1\gamma_2)\cdot m)))\\
      &=(g_1g_2^{\gamma_1}c(\gamma_1,\gamma_2))\cdot
        ((\gamma_1\gamma_2)\cdot m)\\
      &=((g_1,\gamma_1)(g_2,\gamma_2))\cdot m.
  \end{align*}
  We leave to the reader the analogous calculation in the case of a
  right action.
\end{proof}

\begin{remark}\label{Gamma-hat-action}
  Proposition \ref{prop:-G-Gamma-twisted-actions} implies in particular that there is a bijective correspondence between $\hat{\Gamma}$-actions on $M$ and $(\theta,c)$-twisted $\Gamma$-actions on $M$, where $\hat{\Gamma}$ is given by
(\ref{eq:Gamma-Z-extension}).
\end{remark}

A right $\hat{G}$-action on $M$ can be converted into a left
$\hat{G}$-action and vice-versa in the standard way. In the next
proposition we interpret this in terms of twisted
$(G,\Gamma)$-actions. Note the twist in the conversion of the twisted $\Gamma$-action.

\begin{proposition}
  Let $\hat{G}=G\times_{(\theta,c)}\Gamma$ be as above. Suppose
  $\hat{G}$ acts on $M$ on the left and define the corresponding right
  $\hat{G}$-action in the standard way by
  \begin{displaymath}
    m\cdot(g,\gamma) = (g,\gamma)^{-1}\cdot m.
  \end{displaymath}
  Then
  \begin{align*}
    m\cdot (g,1) &=(g^{-1},1)\cdot m,\\
    m\cdot(1,\gamma)&=(c(\gamma^{-1},\gamma)^{-1},1)
                  \cdot((1,\gamma^{-1})\cdot m).
  \end{align*}
  Similarly, if $\hat{G}$ acts on $M$ on the right, then the left
  $\hat{G}$-action $(g,\gamma)\cdot m=m\cdot (g,\gamma)^{-1}$ satisfies
  \begin{align*}
    (g,1)\cdot m &= m \cdot(g^{-1},1),\\
    (1,\gamma)\cdot m &= (m\cdot
      (c(\gamma^{-1},\gamma)^{-1},1))\cdot(1,\gamma^{-1}).
  \end{align*}
\end{proposition}

\begin{proof}
  Consider the case of a left $\hat{G}$-action (the case of a right
  action is analogous). The formula for the right action of $(g,1)$
  follows from $(g,1)^{-1}=(g^{-1},1)$ and the formula for the right
  action of $(1,\gamma)$ follows from
  \begin{displaymath}
    (1,\gamma)^{-1} = (c(\gamma^{-1},\gamma)^{-1},\gamma^{-1})
      = (c(\gamma^{-1},\gamma)^{-1},1)(1,\gamma^{-1}),
  \end{displaymath}
  where we have used \eqref{inverse}.
\end{proof}

This proposition motivates the following definition.

\begin{definition}
  Suppose we have a $(\theta,c)$-twisted $(G,\Gamma)$-left action on
  $M$. Then the \textbf{corresponding $(\theta,c)$-twisted
    $(G,\Gamma)$-right action on $M$} is defined by
  \begin{align*}
    m\cdot g &=g^{-1}\cdot m,\\
    m\cdot \gamma&=c(\gamma^{-1},\gamma)^{-1}
                  \cdot(\gamma^{-1}\cdot m).
  \end{align*}
  Similarly, given a $(\theta,c)$-twisted $(G,\Gamma)$-right action on
  $M$, the \textbf{corresponding $(\theta,c)$-twisted
    $(G,\Gamma)$-left action on $M$} is defined by
  \begin{align*}
    g\cdot m &= m \cdot g^{-1},\\
    \gamma\cdot m &= (m\cdot
      c(\gamma^{-1},\gamma)^{-1})\cdot\gamma^{-1}.
  \end{align*}
\end{definition}

\begin{definition}\label{equivarian-maps}
  A map $f\colon M\to N$ between sets with $(\theta,c)$-twisted
  $(G,\Gamma)$-actions is said to be \textbf{$(\theta,c)$-twisted
  $(G,\Gamma)$-equivariant} if it satisfies
\begin{align*}
  &f(g\cdot m) = g\cdot f(m)\quad\text{and}\quad f(\gamma\cdot m) =
  \gamma\cdot f(m), &g\in G,\;\gamma\in\Gamma, m\in M,
% \intertext{in the case of a left action, and}
%   &f(m\cdot g) = f(m)\cdot g\quad\text{and}\quad f(m\cdot\gamma) =
%   f(m)\cdot\gamma, &g\in G,\;\gamma\in\Gamma, m\in M
\end{align*}
in the case of a left action, with the obvious modification in the
case of a right action.
\end{definition}

The following proposition
shows that the correspondence of
Proposition~\ref{prop:-G-Gamma-twisted-actions} is natural.

\begin{proposition}
  \label{prop:twisted-equivariant-hat}
  A map $f\colon M\to N$ between sets with $(\theta,c)$-twisted
  $(G,\Gamma)$-actions is $(\theta,c)$-twisted
  $(G,\Gamma)$-equivariant if and only if it is $\hat{G}$-equivariant for
  $\hat{G}=G\times_{(\theta,c)}\Gamma$.
\end{proposition}

\begin{proof}
  The ``if'' part is clear. The ``only if'' part follows because
  the group $\hat{G}$ is generated by elements of the form $(g,1)$ and
  $(1,\gamma)$.
\end{proof}

We recall the following elementary fact about group
actions.

\begin{proposition}
Let $\hat{G}$ be a group given as an extension
\begin{displaymath}
  1 \to G \to \hat{G} \to \Gamma \to 1.
\end{displaymath}
Suppose $\hat{G}$ acts on a set $M$. Then there is an induced
action of $\Gamma$ on the orbit space $M/G$ and the quotient map
  $M\to M/G$ is equivariant for the homomorphism
  $\hat{G}\to\Gamma$. \qed
\end{proposition}

As an immediate consequence of this and
Proposition~\ref{prop:-G-Gamma-twisted-actions} we get the following result.

\begin{proposition}
  \label{prop:quotient-Gamma-action}
  Let $\hat{G}=G\times_{(\theta,c)}\Gamma$ be as above. Suppose we
  have a $(\theta,c)$-twisted $(G,\Gamma)$-action on $M$. Then the quotient
  of $M$ by $G$ has a natural $\Gamma$-action and the quotient map
  $M\to M/G$ is equivariant for the homomorphism
  $\hat{G}\to\Gamma$. \qed
\end{proposition}

The following lemma will be crucial later, to prove
Proposition~\ref{prop:EM-Gamma-euivariant-sections}. 

\begin{lemma}
  \label{lem:twisted-equivariant-sections-equivalence}
Let $\hat{G}$ be a group given as an extension
\begin{displaymath}
  1 \to G \to \hat{G} \to \Gamma \to 1.
\end{displaymath}
Suppose that $\hat{G}$ acts on sets $E$ and $M$ on the right and that
the induced $G$-action on $E$ is free. Let $\tilde{s}\colon E\to M$ be
a $G$-equivariant map and let $s\colon E/G\to (E\times M)/G$ be the
map induced by the $G$-equivariant map
$(\Id,\tilde{s})\colon E\to E\times M$. Then the map $\tilde{s}$ is
$\hat{G}$-equivariant if and only if the map $s$ is
$\Gamma$-equivariant.
  % and consider the diagram
  % \begin{displaymath}
  %   \begin{CD}
  %     E @>{1\times \tilde{s}}>> E\times M \\
  %     @VVV @VVV \\
  %     E/G @>{1\times s}>> (E\times M)/G\ .
  %   \end{CD}
  % \end{displaymath}
  % Then the following are equivalent:
  % \begin{enumerate}
  % \item The map $\tilde{s}$ is $\hat{G}$-equivariant.
  % \item The map $1\times \tilde{s}$ is $\hat{G}$-equivariant.
  % \item The map $1\times s$ is $\Gamma$-equivariant.
  % \end{enumerate}
\end{lemma}

\begin{proof}
  Let $\hat{g}\in\hat{G}$ and write
  $\gamma\in\Gamma$ for the image of $\hat{g}$ under the map $\hat{G} \to \Gamma$. Then, for $e\in E$, we have
  \begin{displaymath}
    s([e]\cdot\gamma) = s([e\cdot\hat{g}])
      =[(e\cdot\hat{g},\tilde{s}(e\cdot\hat{g}))].
  \end{displaymath}
  On the other hand
  \begin{displaymath}
    s([e]) \cdot\gamma
    =[(e,\tilde{s}(e))]\cdot\gamma
    =[(e\cdot\hat{g},\tilde{s}(e)\cdot\hat{g})].
  \end{displaymath}
  Now, since $G$ acts freely on $E$, we see that
  $s([e]\cdot\gamma) = s([e]) \cdot\gamma$ if and only if
  \begin{math}
    \tilde{s}(e\cdot\hat{g}) = \tilde{s}(e)\cdot\hat{g}.
  \end{math}
\end{proof}

\subsection{Twisted actions for different
  $\theta$.}\label{section-different-theta}

Up to this point we have worked with a fixed choice of lift $\theta\colon\Gamma\to\Aut(G)$
of the characteristic homomorphism
$\bar{\theta}\colon\Gamma\to\Out(G)=\Aut(G)/\Int(G)$ (see \eqref{lift}).
In this section we show that a different choice of lift leads to essentially
the same theory (see Proposition \ref{prop-different-theta} below).

Given another homomorphism $\theta':\Gamma\to\Aut(G)$ lifting $\bar{\theta}$, there exists a map $s:\Gamma\to G$ such that $\theta'=\Int_s\theta$; we may assume that $s(1)=1$. Here $\Int_s$ is the homomorphism of $\Gamma$ to $\Int(G)$ defined by
$\Int_s(\gamma)=\Int_{s(\gamma)}$. Since $\Int(G)$ acts trivially on $Z$, there is a natural identification $Z^2_{\theta}(\Gamma,Z)\cong Z^2_{\theta'}(\Gamma,Z)$ and so we talk about 2-cocycles as elements in any of these sets indistinctly. Since $\theta'$ is a homomorphism we have
\begin{align*}
    \Int_{s(\gamma)}\Int_{\tg(s(\gamma'))}\theta_{\gamma\gamma'}=\Int_{s(\gamma)}\tg\Int_{s(\gamma')}\theta_{\gamma'}=
    \theta'_{\gamma}\theta'_{\gamma'}=\theta'_{\gamma\gamma'}=\Int_{s(\gamma\gamma')}\theta_{\gamma\gamma'},
\end{align*}
i.e.
\begin{equation}\label{eq-ints-cocycle}
    \Int_{s(\gamma)}\Int_{\tg(s(\gamma'))}=\Int_{s(\gamma\gamma')}
\end{equation}
for every $\gamma$ and $\gamma'\in\Gamma$. In terms of Galois cohomology this means that $\Int_s$ is a 1-cocycle in $Z^1_{\theta}(\Gamma,\Int(G))$. Thus we may define a map
\begin{equation*}
    c_s:\Gamma\to\Gamma\to Z;\,(\gamma,\gamma')\mapsto s(\gamma)\tg(s(\gamma'))s(\gamma\gamma')^{-1}.
\end{equation*}
This is in fact a 2-cocycle, since
\begin{align*}
    \theta_{\gamma_0}(c_s(\gamma_1,\gamma_2))c_s(\gamma_0,\gamma_1\gamma_2)=
    \theta_{\gamma_0}( s(\gamma_1)\theta_{\gamma_1}(s(\gamma_2))s(\gamma_1\gamma_2)^{-1}) s(\gamma_0)\theta_{\gamma_0}(s(\gamma_1\gamma_2))s(\gamma_0\gamma_1\gamma_2)^{-1}=\\
    \theta_{\gamma_0}( s(\gamma_1))\theta_{\gamma_0\gamma_1}(s(\gamma_2)) s(\gamma_0)s(\gamma_0\gamma_1\gamma_2)^{-1}=\\
    s(\gamma_0)\theta_{\gamma_0}(s(\gamma_1))s(\gamma_0\gamma_1)^{-1}s(\gamma_0\gamma_1)\theta_{\gamma_0\gamma_1}(s(\gamma_2))s(\gamma_0\gamma_1\gamma_2)^{-1}
  =c_s(\gamma_0,\gamma_1)c_s(\gamma_0\gamma_1,\gamma_2)
\end{align*}
and
$$c_s(\gamma,1)=s(\gamma)\tg(s(1))s(\gamma)^{-1}=1=s(1)\theta_1(s(\gamma))s(\gamma)^{-1}=c_s(1,\gamma).$$
Thus the product $cc_s$ is also a 2-cocycle.

We have the following.

\begin{proposition}\label{prop-different-theta}
Let $s:\Gamma\to\Int(G)$ be a map satisfying (\ref{eq-ints-cocycle}). Let $\cc{}(\theta,c)$ be the category of pairs $(M,\cdot)$ consisting of a set $M$ and a $(\theta,c)$-twisted $(G,\Gamma)$-action on $M$, whose morphisms are $(\theta,c)$-twisted $(G,\Gamma)$-equivariant maps. Then the categories $\cc{}(\theta,c)$ and $\cc{}(\Int_s\theta,cc_s)$ are equivalent.
\end{proposition}
\begin{proof}
Set $\theta':=\Int_s\theta$ as above. Let $M$ be a set equipped with a $(\theta,c)$-twisted $(G,\Gamma)$-action, which we denote with a dot. We have a natural choice of $(\theta',cc_s)$-twisted $\Gamma$-equivariant action on $M$, namely
\begin{equation}
    M\times\Gamma\to M;\,(m,\gamma)\mapsto m*\gamma:=ms(\gamma)\cdot\gamma.
\end{equation}
This satisfies Definition \ref{def:G-Gamma-twisted-action}:
\begin{equation*}
    (mg)*\gamma=(ps(\gamma)s(\gamma)^{-1}gs(\gamma))\cdot\gamma=(ms(\gamma))\cdot\gamma\theta_{\gamma}^{-1}(s(\gamma)^{-1}gs(\gamma))=m*\gamma\theta'^{-1}_{\gamma}(g)
\end{equation*}
and
\begin{align*}
    (m*\gamma)*\gamma'=(ms(\gamma)\cdot\gamma)s(\gamma')\cdot\gamma'=((m\cdot\gamma)\cdot\gamma')\theta_{\gamma\gamma'}^{-1}(s(\gamma))\theta_{\gamma'}^{-1}(s(\gamma'))=\\
    ((mc(\gamma,\gamma'))\cdot\gamma\gamma')\theta_{\gamma\gamma'}^{-1}(s(\gamma)\theta_{\gamma}(s(\gamma')))=
    ((mc(\gamma,\gamma'))\cdot\gamma\gamma')\theta_{\gamma\gamma'}^{-1}(c_s(\gamma,\gamma')s(\gamma\gamma'))=\\
    (mc(\gamma,\gamma')c_s(\gamma,\gamma')s(\gamma\gamma'))\cdot\gamma\gamma'=
    (mc_s(\gamma,\gamma'))*\gamma\gamma'
\end{align*}
for each $g\in G$, $\gamma,\gamma'\in\Gamma$ and $m\in M$.

A $(\theta,c)$-twisted $(G,\Gamma)$-equivariant morphism of sets $f:M\to M'$ is also $(\theta',cc_s)$-twisted $\Gamma$-equivariant. Indeed, we have
\begin{equation*}
    f(m*\gamma)=f(ms(\gamma)\cdot\gamma)=f(m\cdot\gamma)\tg^{-1}(s(\gamma))=f(m)\cdot\gamma\tg^{-1}(s(\gamma))=f(m)s(\gamma)\cdot\gamma=f(m)*\gamma
\end{equation*}
for each $m\in M$ and $\gamma\in\Gamma$.

Finally it is clear that the functor is invertible, namely we get the $\cdot$ action by composing the $*$ action with multiplication by $s(\bullet)^{-1}$.
\end{proof}
\color{black}

%%%%%%%%%%%%%%%%%%%%%%%%%%%%%%%%%%%%%%%%%%%%%%%%%%%%
\section{Twisted equivariant structures on bundles}
\label{sec:twist}
%%%%%%%%%%%%%%%%%%%%%%%%%%%%%%%%%%%%%%%%%%%%%%%%%%%%

In this section  we review and further develop the theory of twisted equivariant
bundles introduced in  \cite{garcia-prada-wilkin,basu-garcia-prada}
(see also \cite{BG,BGH1,BGH2,BCFG,BCG}).
We work in the smooth category
for definiteness. However, everything could equally well be done in the topological and the
holomorphic categories. Let  $X$ be a smooth manifold (a topological or
complex manifold if we work in the topological or holomorphic category) and $G$ be a Lie
group (topological or complex if we work in the topological or holomorphic category).

%%%%%%%%%%%%%%%%%%%%%%%%%%%%%%%%%%%%%%%%%%%
\subsection{Twisted equivariant bundles}
\label{sec:twisted-bundles}
%%%%%%%%%%%%%%%%%%%%%%%%%%%%%%%%%%%%%%%%%%%

Let $\Gamma$ be a discrete group, and
let $G$ be a Lie group with centre $Z(G)=Z$. As in
Sections~\ref{sec:extensions-lie} and \ref{sec:actions-extensions-lie}
assume that $\Gamma$ acts on $G$ via $\theta\colon\Gamma\to\Aut(G)$
and let
\begin{math}
  c\in Z^2_{\theta}(\Gamma,Z)
\end{math}
be a $2$-cocycle for this action. Let $\hat{G}=G\times_{(\theta,c)}\Gamma$
be the extension of $\Gamma$ by $G$ defined by $\theta$ and $c$, and
write $\hat{\Gamma}_{\theta,c}$ for the corresponding central
extension of $\Gamma$ by $Z$.

\begin{definition}
  \label{twisted-G-action}
  A \textbf{$(\theta,c)$-twisted $(G,\Gamma)$-manifold} is a smooth
  manifold $M$ with a $(\theta,c)$-twisted $(G,\Gamma)$-action such
  that the maps defined by each $g\in G$ and $\gamma\in\Gamma$ are
  diffeomorphisms of $M$.
\end{definition}

In the case when $G=Z$, we simply use the expression
\textbf{$(\theta,c)$-twisted $\Gamma$-manifold} for a
$(\theta,c)$-twisted $(Z,\Gamma)$-manifold. If we want to specify
whether the action is on the right or on the left, we qualify the word
``manifold'' with the corresponding adjective.

\begin{remark}
  \label{rem-GGamma-Ghat-manif-eq}
  From Propositions~\ref{prop:-G-Gamma-twisted-actions} and
  \ref{prop:twisted-equivariant-hat} we see that the category of  $(\theta,c)$-twisted
  $(G,\Gamma)$-manifolds is equivalent to the category of
  $\hat{G}$-manifolds (with the obvious notions of morphisms).
\end{remark}

Let $X$ be a $\Gamma$-manifold.

\begin{definition}
  \label{def:gamma-twisted-bundle}
  A \textbf{$(\theta,c)$-twisted $\Gamma$-equivariant bundle on $X$}
  is a fibre bundle $F\to X$ with a $(\theta,c)$-twisted action of
  $\Gamma$ on the total space $F$, such that $Z$ acts fibrewise and
  the projection $\pi\colon F\to X$ is \textbf{$\Gamma$-equivariant},
  in the sense that, in the case of a right action,
  \begin{math}
    \pi(f\cdot\gamma) = \pi(f)\cdot\gamma,
  \end{math}
  and, in the case of a left action,
  \begin{math}
    \pi(\gamma\cdot f) = \gamma\cdot\pi(f)
  \end{math}
  for $\gamma\in\Gamma$ and $f\in F$.
\end{definition}

\begin{remark}
  In view of Proposition~\ref{prop:-G-Gamma-twisted-actions} we see
  that a $(\theta,c)$-twisted $\Gamma$-equivariant bundle on $X$ is
  the same thing as a bundle $F\to X$ with a
  $\hat{\Gamma}_{\theta,c}$-action on $F$, such that the projection is
  equivariant for the homomorphism $\hat{\Gamma}_{\theta,c}\to\Gamma$.
\end{remark}

Let $X$ be a right $\Gamma$-manifold.

\begin{definition}\label{def:gamma-twisted-principal-bundle}
  A \textbf{$(\theta,c)$-twisted $\Gamma$-equivariant principal
    $G$-bundle on $X$} is a right $(\theta,c)$-twisted
  $\Gamma$-equivariant bundle $E\to X$ with a right $G$-action such that
  \begin{enumerate}
  \item the actions of $G$ and $\Gamma$ make $E$ into a
    $(\theta,c)$-twisted $(G,\Gamma)$-manifold, and
  \item the action of $G$ makes $E\to X$ into a principal $G$-bundle.
  \end{enumerate}
\end{definition}

Explicitly, the preceding definition means that there is smooth a
$(\theta,c)$-twisted right $\Gamma$-action on the principal $G$-bundle
$E$, covering the $\Gamma$-action on $X$, and such that
\begin{align}\label{eq:twist-gamma-1}
  (eg^{\gamma})\cdot\gamma &= (e\cdot\gamma)g, \\
  (e\cdot \gamma_1)\cdot \gamma_2 &=
  (ec(\gamma_1,\gamma_2))\cdot(\gamma_1\gamma_2).
\label{eq:twist-gamma-2}
\end{align}
Note that, in accordance with custom, we usually omit the ``$\cdot$''
in the notation for the right action of $G$ on $E$.

\begin{remark}
  In view of Proposition~\ref{prop:-G-Gamma-twisted-actions} we have
  the following alternative definition: a $(\theta,c)$-twisted
  $\Gamma$-equivariant principal $G$-bundle on a right
  $\Gamma$-manifold $X$ is a right $\hat{G}$-manifold $E$ together
  with a smooth map $\pi\colon E\to X$ which is equivariant for the
  homomorphism $\hat{G}\to\Gamma$, and defines a structure of
  principal $G$-bundle on $E$ over $X$.
\end{remark}

\begin{remark}
  If $\Gamma$ acts trivially on $X$, a $(\theta,c)$-twisted
  $\Gamma$-equivariant principal $G$-bundle on $X$ is the same thing
  as a principal $\hat{G}$-bundle on $X$.
\end{remark}

The next definition completes the description of the \textbf{category
  of $(\theta,c)$-twisted $\Gamma$-equivariant principal $G$-bundles}.

\begin{definition}
  \label{def:morphism-twq}
  A \textbf{morphism} between $(\theta,c)$-twisted
  $\Gamma$-equivariant principal $G$-bundles is a commutative diagram
\begin{displaymath}
  \xymatrix{
    E \ar[r]^{\tilde{f}} \ar[d] & F \ar[d] \\
    X \ar[r]^f & Y
  }
\end{displaymath}
where $f\colon X\to Y$ is $\Gamma$-equivariant, and
$\tilde{f}\colon E\to F$ is $(\theta,c)$-twisted
$(G,\Gamma)$-equivariant, i.e.,
$\tilde{f}(e\cdot\gamma)= \tilde{f}(e)\cdot\gamma$ and
$\tilde{f}(eg)=\tilde{f}(e)g$ for $e\in E$, $\gamma\in\Gamma$ and
$g\in G$.
\end{definition}

\begin{remark}
In the preceding definition, in view of
Proposition~\ref{prop:twisted-equivariant-hat}, we may equivalently
require $\tilde{f}$ to be $\hat{G}$-equivariant.
\end{remark}

\begin{remark}
  According to Definition~\ref{def:morphism-twq}, two
  $(\theta,c)$-twisted $\Gamma$-equivariant principal $G$-bundles on
  the same base $X$ are isomorphic if there is an isomorphism between
  them covering a $\Gamma$-equivariant automorphism of $X$. Below we
  shall use notions of isomorphism covering more restricted classes of
  $\Gamma$-equivariant automorphisms of $X$, e.g., the identity, or
  the action of $Z(\Gamma)$.
\end{remark}

Proposition \ref{prop-different-theta} implies the following.
\begin{proposition}
Let $s:\Gamma\to G$ be a map making $\Int_s\theta:\Gamma\to \Aut(G)$ a homomorphism. Let $c_s\in Z^2_{\theta}(\Gamma,Z)$ be the 2-cocycle defined in Section \ref{section-different-theta}. Then the category of $(\theta,c)$-twisted $\Gamma$-equivariant $G$-bundles on is equivalent to the category of $(\Int_s\theta,cc_s)$-twisted $\Gamma$-equivariant $G$-bundles on $X$.
\end{proposition}

\color{black}
%\subsubsection*{pull-back}
We can define the pull-back of a $(\theta,c)$-twisted $\Gamma$-equivariant $G$-bundle
$E\xra{\pi} Y$ under a $\Gamma$-equivariant map $f\colon X\to Y$ in the usual way
and obtain a $(\theta,c)$-twisted $\Gamma$-equivariant $G$-bundle $f^*E\to
X$, as follows. Recall that the pull-back can be defined as
\begin{displaymath}
  f^*E=\{(x,e)\in X\times E \suchthat f(x)=\pi(e)\},
\end{displaymath}
and this is principal $G$-bundle in a natural way.
We define the twisted $\Gamma$-action by coordinate-wise multiplication
\begin{displaymath}
  (x,e)\cdot\gamma = (x\cdot\gamma,e\cdot\gamma)
\end{displaymath}
and one checks that $f^*E\to X$ is indeed a $(\theta,c)$-twisted $\Gamma$-equivariant
bundle. Moreover, the natural bundle map $f^*E\to E$
covering $f\colon X\to Y$ is a morphism of $(\theta,c)$-twisted
$\Gamma$-equivariant $G$-bundles.

% We will sometimes write $\eta\colon\Gamma\to\Aut(X)$ for
% the homomorphism corresponding to the $\Gamma$-action on $X$. Also,
% we may write $\eta_\gamma$ for $\eta(\gamma)$. Similarly, we sometimes
% write $\tilde{\eta}(\gamma)$ or $\tilde{\eta}_{\gamma}$ for the
% $\theta(\gamma)$-twisted automorphism of $E$ covering the map
% $\eta(\gamma)$ on $X$.

\begin{remark}
  If $f=\eta_\gamma\colon X\to X,\; x\mapsto x\cdot \gamma^{-1}$
  corresponds to the action of $\gamma\in\Gamma$, then we have a
  natural isomorphism of $(\theta,c)$-twisted $\Gamma$-equivariant $G$-bundles
  \begin{displaymath}
    \eta_\gamma^*E\xra{\cong}E
  \end{displaymath}
  covering $\eta_\gamma\colon X\to X$.
  Note that the action of $\Gamma$ on $X$ making
  $\eta_\gamma^*E$ into a twisted $\Gamma$-equivariant bundle is given
  as follows: an element $\lambda\in\Gamma$ acts by $x\mapsto
  x\cdot(\gamma^{-1}\lambda\gamma)$. Composing with the inverse of
  this isomorphism we see that
  $\tilde{\eta}_\gamma\colon E\to E,\;e\mapsto e\cdot\gamma^{-1}$
  covering $\eta_\gamma\colon X\to X$ corresponds to an isomorphism
  \begin{displaymath}
    E \xra{\cong} \eta_\gamma^*E
  \end{displaymath}
  of $(\theta,c)$-twisted $\Gamma$-equivariant $G$-bundles covering the identity.
\end{remark}

\begin{definition}\label{twisted-equivariant-section}
  A smooth section $s\colon X\to F$ of a twisted $\Gamma$-equivariant
  fibre bundle $F\to X$ is said to be \textbf{twisted
    $\Gamma$-equivariant} if it satisfies
  \begin{displaymath}
    s(x\cdot\gamma) = s(x)\cdot\gamma, \quad x\in X,\;\gamma\in\Gamma
  \end{displaymath}
  in the case of a right action,
  with the obvious modification in the case of a left action.
\end{definition}

\begin{remark}
 We can make $\hat{\Gamma}_{\theta,c}$ act on the
  space of sections of $F\to X$.
  For definiteness assume that the twisted
  $\Gamma$-action on $F\to X$ is on the right. Then we can define
\begin{align*}
  (s\cdot\gamma)(x) &= s(x\cdot \gamma^{-1})\cdot\gamma,\\
  (s\cdot z)(x) &= s(x)\cdot z
\end{align*}
for $\gamma\in\Gamma$, $z\in Z$, $x\in X$ and $s\in C^\infty(X,F)$.
It is simple to check
that this defines a twisted right
$\Gamma$-action on $C^\infty(X,F)$. In other words, we have a
$\hat\Gamma_{\theta,c}$-action on $C^\infty(X,F)$.
Moreover, $s$ is twisted $\Gamma$-equivariant if and only if $s\cdot\gamma
= s$ for all $\gamma\in\Gamma$.
\end{remark}

We denote the space of twisted $\Gamma$-equivariant smooth sections of
a twisted $\Gamma$-equivariant fibre bundle $F\to X$ by
$C^\infty(X,F)^{\Gamma}$. Of course this includes the case of a  $\Gamma$-twisted principal
$G$-bundle $E\xra{\pi}X$.

%%%%%%%%%%%%%%%%%%%%%%%%%%%%%%%%%%%%%%%%%%%%%%%%%%%%%%%%%%%%%%%%%%
\subsection{Twisted equivariant structures and associated bundles}
\label{associated-bundles}
%%%%%%%%%%%%%%%%%%%%%%%%%%%%%%%%%%%%%%%%%%%%%%%%%%%%%%%%%%%%%%%%%%

Let $E$ be a principal $G$-bundle over $X$ and $M$ be
any set on which $G$ acts. In the (most common) case, when the $G$-action on
$M$ is on the left, we convert it into a right action in the standard
way, by defining
\begin{displaymath}
  m\cdot g = g^{-1}\cdot m.
\end{displaymath}
We denote by $E(M)$ the orbit space $(E\times M)/G$. In the case
of a left action of $G$ on $M$, we thus have the twisted product
$E(M)=E\times_GM$, which can be viewed as the quotient of
$E\times M$ by the equivalence relation
\begin{equation}\label{equiv-associated}
  (eg, m)\sim(e,g\cdot m)
\end{equation}
for any
$e\in E$, $g\in G$ and $m\in M$.

\begin{remark}
  \label{rem:sections-equivariant-correspondence}
  The smooth sections $s$ of $E(M)$ are in natural bijection with the
  smooth maps $\tilde{s}:E\to M$ satisfying the $G$-equivariance
  condition
  \begin{equation}\label{antiequivariant}
    \tilde{s}(eg)=\tilde{s}(e)\cdot g \;\; \mbox{for any}\;\; e\in E\;\; \mbox{and} \;\; g\in G.
  \end{equation}
  % Furthermore, $\tilde{s}$ is smooth if and only if $s$ is smooth.
  To see this, let $\tilde{s}:E\to M$ satisfy
  \eqref{antiequivariant}. Then $(\Id,\tilde{s})\colon E\to E\times M$
  is $G$-equivariant and, therefore, descends to the quotient so that
  we have a commutative diagram
  \begin{displaymath}
    \begin{CD}
      E @>{(\Id,\tilde{s})}>> E\times M \\
      @VVV @VVV \\
      E/G @>{s}>> (E\times M)/G\ ,
    \end{CD}
  \end{displaymath}
  and $s\colon E/G = X \to E(M) = (E\times M)/G$ is the section
  corresponding to $\tilde{s}$. Conversely, given a section $s$, we
  can recover $\tilde{s}$ by setting $\tilde{s}(e) = m$, where
  $s([e])=[(e,m)]$; this is well-defined because the fibres of
  $E\to X$ are $G$-torsors.
\end{remark}

We view $E$ as the $G$-frame bundle of $E(M)$ in the usual way: an
element $e\in E$ with $\pi(e)=x$ defines the frame
\begin{equation} \label{frame}
\begin{aligned}
  M &\xra{\cong} E(M)_x, \\
  m &\mapsto [e,m].
\end{aligned}
\end{equation}

\begin{proposition}
  \label{prop:EM-Gamma-equivariant}
  Let $M$ be a $(\theta,c)$-twisted right $(G,\Gamma)$-manifold and
  let $\pi\colon E\to X$ be a  $(\theta,c)$-twisted $\Gamma$-equivariant principal
  $G$-bundle. Then the associated fibre
  bundle with typical fibre $M$,
\begin{align*}
  E(M):=E\times_{G}M &\to X,\\
  [e,m] &\mapsto \pi(e),
\end{align*}
defined by the above construction is a $\Gamma$-equivariant fibre bundle.
\end{proposition}

\begin{proof}
  Proposition~\ref{prop:quotient-Gamma-action} shows that
\begin{displaymath}
  [e,m]\cdot\gamma = [e\cdot\gamma,m\cdot\gamma]
\end{displaymath}
defines a $\Gamma$-action on $E(M)$, and $\Gamma$-equivariance of the
projection is immediate from (twisted) $\Gamma$-equivariance of $\pi\colon E\to X$.
\end{proof}

% \begin{remark}
% If the subgroup of $Z$ generated by the image of  the cocycle $c$ acts trivially on $M$, then $E(M)$ acquires a true $\Gamma$-equivariant structure.
% \end{remark}

\begin{remark}
  Note that the map (\ref{frame}) has no equivariance properties with
  respect to $\Gamma$.
\end{remark}

Let $C^\infty(E,M)^{G,\Gamma}$ be the space of $(\theta,c)$-twisted
$(G,\Gamma)$-equivariant smooth maps $\tilde{s}\colon E\to M$ in the
sense of Definition \ref{equivarian-maps}, and let
$C^{\infty}(X,E(M))^{\Gamma}$ be the space of $\Gamma$-equivariant
sections of $E(M)\to X$.
% \footnote{In case we have a left action on $M$ we convert it into a
%   right action in the standard way.}

\begin{proposition}
\label{prop:EM-Gamma-euivariant-sections}
With the above notation, there is a bijection
\begin{displaymath}
  C^\infty(E,M)^{G,\Gamma} \xra{\cong} C^{\infty}(X,E(M))^{\Gamma}.
\end{displaymath}
% sending $\tilde{s}\colon E\to M$ to $s\colon X\to E(M)$ defined by
% \begin{displaymath}
%   s(x) = [e, \tilde{s} (e)]
% \end{displaymath}
% for any $e\in E$ with $\pi(e)=x\in X$.
% The inverse takes a section $s\colon X\to E(M)$ to $\tilde{s}\colon
% E\to M$ defined by
% \begin{displaymath}
%   \tilde{s}(e) = m,\quad\text{where $s(\pi{(e)}) = [e,m]$}.
% \end{displaymath}
\end{proposition}

\begin{proof}
  The correspondence between $G$-equivariant maps $E\to M$ and
  sections of $E(M)\to X$ is given by Remark~\ref{rem:sections-equivariant-correspondence}
and the $\Gamma$-equivariance statement follows from Lemma~\ref{lem:twisted-equivariant-sections-equivalence}.
\end{proof}

\begin{remark}
  Since usually the twisted $(G,\Gamma)$-action on $M$ is on the left,
  we spell out the correspondence in the case: there is a bijective
  correspondence between $\Gamma$-equivariant smooth sections of
  $E(M)\to X$ and smooth maps $\tilde{s}\colon E\to M$ satisfying
  \begin{align*}
    \tilde{s}(eg) &= g^{-1}\cdot\tilde{s}(e),\\
    \tilde{s}(e\cdot\gamma) &=
      c(\gamma^{-1},\gamma)^{-1}\cdot(\gamma^{-1}\cdot\tilde{s}(e)).
  \end{align*}
\end{remark}

%%%%%%%%%%%%%%%%%%%%%%%%%%%%%%%%%%%%%%%%%%%%%%%%%%%%%%%%%%%%%%%%%%%%%%%%%%%%%
\subsection{Extension and reduction of structure group of twisted equivariant bundles}
\label{reductions-twist}
%%%%%%%%%%%%%%%%%%%%%%%%%%%%%%%%%%%%%%%%%%%%%%%%%%%%%%%%%%%%%%%%%%%%%%%%%%%%%

Let $H$ be a Lie group equipped with a homomorphism $\tau:\Gamma\to\Aut(H)$. Let $\psi\colon G\to H$ be a $\Gamma$-equivariant Lie group homomorphism such that $\psi(c(\Gamma\times\Gamma))\subseteq Z(H)$, the centre of $H$. Then $\psi\circ c:\Gamma\times\Gamma\to Z(H)$ is an element of $Z^2_{\tau}(\Gamma,Z(H))$.
 This setting allows us to
include extensions of structure group in our framework. To do this, recall that
the principal $H$-bundle obtained by extension of structure group of
the principal $G$-bundle $E$ via
$\psi\colon G\to H$ is $F=E\times_{G}H$, where the equivalence
relation is $(e,h)\sim (eg^{-1},\psi(g)h)$. Then,
if $E$ is a
twisted $\Gamma$-equivariant bundle, we can define a twisted $\Gamma$-action on $F=E\times_{G}H$ by
\begin{displaymath}
   [e,h]\cdot\gamma = [e\cdot\gamma,h\cdot\gamma].
\end{displaymath}
It is easy to check that this is well defined and makes $F\to X$ into
a $(\tau,\psi\circ c)$-twisted
$\Gamma$-equivariant $H$-bundle such that the natural map $E\to F$ is a
morphism of twisted $\Gamma$-equivariant principal bundles.

Let $E$ be a $G$-bundle over $X$ and let $H\subseteq G$ be a subgroup. A {\bf reduction of structure
  group} of $E$ to $H$ is a section $\sigma$  of $E(G/H)$.
If we
view $E$ as the bundle of $G$-frames of itself, then $\sigma$ is a section
of the bundle of $H$-equivalence classes of $G$-frames.
Since $E(G/H)\cong E/H$ canonically and the quotient $E\to
E/H$ has the structure of a principal $H$-bundle, the pullback
$E_{\sigma}:=\sigma^*E$ is a principal $H$-bundle over $X$, and
we can identify canonically $E\cong E_{\sigma}\times_{H}G$ as
principal $G$-bundles (this justifies saying that $\sigma$ gives a
reduction of the
structure group of $E$ to $H$).

Now assume that $H\subseteq G$ is invariant under $\theta$ and that
$E$ is a $(\theta,c)$-twisted $\Gamma$-equivariant principal $G$-bundle.
We want to obtain sufficient and necessary conditions
for the $\Gamma$-action on $E$ to preserve $E_{\sigma}$, regarded as
an $H$-invariant submanifold of $E$. Note that the section of
$E(G/H)$ determining the reduction of structure group may be
regarded as a $G$-equivariant map $s:E\to G/H$, where $G$ acts on
$G/H$ by inverse left multiplication. This map is determined by
$E_{\sigma}$, namely $e\in E$ is sent to $g^{-1}H$, where $g\in G$
is such that $e\in E_{\sigma}g$. For every $\gamma\in\Gamma$, from (\ref{eq:twist-gamma-1}), we have
$$
e\cdot\gamma\in (E_{\sigma}g)\cdot\gamma = (E_{\sigma}\cdot\gamma)\tg^{-1}(g)
$$
which shows that $E_{\sigma}=E_{\sigma}\cdot\gamma$ if and only if $s$ is $\Gamma$-equivariant, where $\Gamma$ acts on $G/H$ on the right via the inverse of $\theta$. This is a $\theta$-twisted right $\Gamma$-equivariant action:

\begin{proposition}
  \label{prop:GammaGmodHaction}
  Let $H\subseteq G$
be a Lie subgroup which is preserved by the
$\Gamma$-action, i.e., such that $\gamma\cdot h\in H$ for all $h\in
H$. Consider the usual $G$-action on the homogeneous space $G/H$ given by
$g_1\cdot (gH) = (g_1g)H$. Then
\begin{displaymath}
  \gamma\cdot (gH) = (\gamma\cdot g)H.
\end{displaymath}
defines a $\theta$-twisted $(G,\Gamma)$-action on
$G/H$, in other words, an action of the semidirect product
$G\rtimes_{\theta}\Gamma$ on $G/H$ defined by $\theta$.
\end{proposition}

\begin{proof}
For every $g,g'\in G$ and $\gamma,\gamma'\in\Gamma$ we have
\begin{equation*}
  \gamma\cdot (g'gH)=\tg(g'g)H=\tg(g')\tg(g)H=\tg(g')(\gamma\cdot gH)
\end{equation*}
and
\begin{equation*}
  \gamma\cdot(\gamma'\cdot gH)=\tg\theta_{\gamma'}(g)H=\theta_{\gamma\gamma'}(g)H=\gamma\gamma'\cdot gH,
\end{equation*}
as required.
% \color{blue}
%   This is an easy check which we leave to the reader. The reason why
%   the action is not twisted by the cocycle  is that $\Gamma$ acts on $G$ via the
%   normalised section by conjugation and, therefore, the cocycle (which
%   takes values in $Z(G)$) has no effect.
\end{proof}

By Section \ref{associated-bundles} we have a right $\Gamma$-action on the space of sections of $E(G/H)$, and we thus have the following.
\begin{proposition}
Given a reduction of structure group $\sigma\in E(G/H)$, the corresponding $H$-bundle $E_{\sigma}\subset E$ is $\Gamma$-invariant if and only if $\sigma$ is $\Gamma$-invariant. If any of these two equivalent conditions is met then there is an induced $(\theta,c)$-twisted $\Gamma$-equivariant structure on $E_{\sigma}$.
\end{proposition}

This motivates the following:
\begin{definition}
  Given a $(\theta,c)$-twisted $\Gamma$-equivariant $G$-bundle $E$ over $Y$ and a $\Gamma$-invariant subgroup $H\subset G$, a \textbf{$(\theta,c)$-twisted $\Gamma$-equivariant reduction of
structure group of $E$ to $H$} is a $\Gamma$-invariant section of $E(G/H)$.
\end{definition}

% Equivalently, from Section~\ref{associated-bundles} we can also
% regard a reduction of structure group of $E$ to $H$ as
% $G$-equivariant map $\tilde{\sigma}\colon E\to G/H$ in the sense of
% equation (\ref{antiequivariant}). The subbundle
% $E_\sigma:=\tilde{\sigma}^{-1}(H)\subseteq E$ is then preserved by
% the action of $H\subset G$ and hence it is a principal $H$-bundle
% (here we denoted by $H\in G/H$ the coset of the neutral element in
% $G$).

\begin{remark}
In view of
  Proposition~\ref{prop:EM-Gamma-euivariant-sections}, the map
  $\tilde{\sigma}$ corresponds to a $\Gamma$-equivariant section
  $\sigma$ of the $\Gamma$-equivariant bundle $E(G/H)\to X$.
\end{remark}

\section{Twisted equivariant bundles and coverings}
\label{coverings-twist}
%%%%%%%%%%%%%%%%%%%%%%%%%%%%%%%%%%%%%%%%%%%%%%%%%%%%%%%%%%%%%%%%%

% \comment{
%   \begin{itemize}
%   \item Do we really want to use ``Galois'' here? I think it is
%     preferable to use topological language.
%   \item Might just assume $\Gamma$ discrete in this section?
%   \end{itemize}
% }

Let $G$ be a connected Lie group and let $\Gamma$ be a discrete group.
Consider a group $\hat{G}$ given as an extension
\begin{equation}
  \label{eq:extension-Gamma-G}
  1 \to G \to \hat{G} \to \Gamma \to 1.
\end{equation}
As in Section \ref{sec:extensions-lie}, we will assume that there is a
lift of the characteristic homomorphism $\bar{\theta}\colon\Gamma \to \Out(G)$
of the extension \eqref{eq:extension-Gamma-G}
to a homomorphism $\theta\colon \Gamma \to \Aut(G)$ making the diagram
\begin{displaymath}%\label{lift2}
  \xymatrix{
    \Gamma \ar[r]^(.35){\theta} \ar[dr]_{\bar\theta} & \Aut(G) \ar[d] \\
    & \Out(G)
  }
  \end{displaymath}
commutative and identify $\hat{G}=G\times_{\theta,c}\Gamma$.

Let $X$ be a connected smooth manifold.  In this section we relate principal
$\hat{G}$-bundles on $X$ with twisted $\Gamma$-equivariant principal
$G$-bundles over a covering $p\colon Y\to X$ with covering group
$\Gamma$. We work in the smooth category for definiteness, but
everything could equally well be done in the holomorphic category.

%%%%%%%%%%%%%%%%%%%%%%%%%%%%%%%%%%%%%%%%%
\subsection{An equivalence of categories}
\label{equivalence-of-categories}
%%%%%%%%%%%%%%%%%%%%%%%%%%%%%%%%%%%%%%%%%

The goal of this section is to relate principal $\hat{G}$-bundles on a
manifold $X$ to $(\theta,c)$-twisted $\Gamma$-equivariant principal
$G$-bundles on a suitable covering $Y\to X$.
Since we wish to understand isomorphism classes, we shall
make categorical statements. In order to find the correct notion of
morphism we start upstairs.

Let $Y$ be a connected smooth manifold equipped with a smooth $\Gamma$-action $\Gamma\to\Aut(Y)$ and let $E\to Y$ be a
$(\theta,c)$-twisted $\Gamma$-equivariant principal $G$-bundle. We
have a canonical isomorphism $E/G\cong Y$ and a right
$\Gamma$-action on $Y$ induced by the $G$-action on $E$. Moreover, if
the $\Gamma$-action on $Y$ is free and properly discontinuous, then
$E\to X=Y/\Gamma$ is a principal $\hat{G}$-bundle (cf.\ Remark~\ref{rem-GGamma-Ghat-manif-eq}).

\begin{proposition}
  \label{prop:morhisms-fixed-covering}
  Let $E\to Y$ and $E'\to Y$ be $(\theta,c)$-twisted $\Gamma$-equivariant
  principal $G$-bundles. Assume that the induced $\Gamma$-action on
  $E/G \cong Y \cong E'/G$ is free and write $Y\to X=Y/\Gamma$ for
  the corresponding (right) $\Gamma$-covering.  Let
  $\phi\colon E\to E'$ be a morphism of the corresponding principal
  $\hat{G}$-bundles on $X$, i.e., a
  $\hat{G}$-equivariant map covering the
  identity on $X$. Then there is a unique covering transformation
  $\bar{\phi}\colon Y\to Y$ making the diagram
  \begin{displaymath}
    \begin{CD}
      E @>{\phi}>> E'\\
      @VVV @VVV \\
      Y @>{\bar{\phi}}>> Y \\
      @VVV @VVV \\
      X @>{\mathrm{Id}}>> X
    \end{CD}
  \end{displaymath}
  commutative. Moreover $\bar{\phi}\in Z(\Gamma)\subset\Gamma=\Aut(Y/X)$.
\end{proposition}

\begin{proof}
  Since $\phi$ is $\hat{G}$-equivariant, it is also $G$-equivariant.
  Hence $\phi$ descends to a unique $\bar{\phi}$ which, being a
  covering transformation, is of the form
  $\bar{\phi}(y)=y\cdot\bar{\gamma}$ for a fixed
  $\bar{\gamma}\in\Gamma$
  Moreover, $\hat{G}$-equivariance of $\phi$
  implies $\Gamma$-equivariance of $\bar{\phi}$, so that for any
  $\gamma\in\Gamma$ and $y\in Y$,
  \begin{displaymath}
    y\cdot\gamma\bar{\gamma}
    =(y\cdot\gamma)\cdot\bar{\gamma}
    =\bar{\phi}(y\cdot\gamma)
    =\bar{\phi}(y)\cdot\gamma
    =(y\cdot\bar{\gamma})\cdot\gamma
    =y\cdot\bar{\gamma}\gamma.
  \end{displaymath}
  We conclude that $\gamma\bar{\gamma} = \bar{\gamma}\gamma$ because
  the $\Gamma$-action is free.
\end{proof}

Let $Y\to X$ be a fixed connected $\Gamma$-covering, where the
discrete group $\Gamma$ acts on the right.

\begin{definition}
  We denote by $\mathcal{C}_1$ the category whose
  \begin{itemize}
  \item \textbf{objects} are $(\theta,c)$-twisted $\Gamma$-equivariant
    principal $G$-bundles $E\to Y$ such that the twisted
    $\Gamma$-action descends to the action of $\Gamma$ as covering
    transformations of the fixed $\Gamma$-covering $E/G\cong Y\to X$,
    and whose
  \item \textbf{morphisms} are $(\theta,c)$-twisted $(G,\Gamma)$-equivariant
    maps $\phi\colon E\to E'$ such that the diagram
    \begin{displaymath}
      \begin{CD}
      E @>{\phi}>> E'\\
      @VVV @VVV \\
      Y @>{\bar{\phi}}>> Y
      \end{CD}
    \end{displaymath}
    commutes and the induced map $\bar{\phi}\colon Y\to Y$  is
    a covering transformation which belongs to $Z(\Gamma)$.
  \end{itemize}
\end{definition}

\begin{proposition}
  \label{prop:principal-Ghat-fixed-cover}
  The category $\mathcal{C}_1$ is isomorphic to the category whose
  objects are principal $\hat{G}$-bundles $E\to X$ together with an
  identification $E/G\xra{\cong}Y\to X$ of $E/G$ with the fixed
  $\Gamma$-covering $Y\to X$, and whose morphisms are morphisms of
  principal $\hat{G}$-bundles.
\end{proposition}

\begin{proof}
  As explained in the paragraph preceding
  Proposition~\ref{prop:morhisms-fixed-covering} objects of one
  category can be viewed as objects of the other, and in view of
  Proposition~\ref{prop:morhisms-fixed-covering} morphisms are also
  the same.
\end{proof}

\begin{definition}
  We denote by $\mathcal{C}_2$ the category whose
  \begin{itemize}
  \item \textbf{objects} are principal $\hat{G}$-bundles on $X$ such
    that there is an isomorphism $E/G\cong Y$ covering the identity on
    $X$
    and whose
  \item \textbf{morphisms} are morphisms of principal
    $\hat{G}$-bundles.
  \end{itemize}
\end{definition}

In view of Proposition~\ref{prop:principal-Ghat-fixed-cover}, the
only difference between $\mathcal{C}_1$ and $\mathcal{C}_2$ is that in
the latter category we do not specify the identification $E/G\cong Y$.

Recall that a functor is an equivalence of categories if it is fully
faithful and essentially surjective. In particular, it induces a
bijection on isomorphism classes of objects.

\begin{proposition}
  \label{prop:G-Ghat-equivalence}
  The functor from $\mathcal{C}_1$ to $\mathcal{C}_2$ which forgets
  the identification $E/G\cong Y$ is an equivalence of categories.
\end{proposition}

\begin{proof}
  Proposition~\ref{prop:morhisms-fixed-covering} says that the functor
  is full (surjective on hom-sets) and faithful (injective on
  hom-sets). Finally, given a principal $\hat{G}$-bundle on $X$ such
    that there is an isomorphism $E/G\cong Y$ covering the identity on
    $X$, we can of course choose such an isomorphism. Hence the
    functor is in fact surjective.
\end{proof}

\begin{remark}
  Even though an inverse functor from $\mathcal{C}_2$ to
  $\mathcal{C}_1$ can be constructed on abstract grounds, using the
  axiom of choice to specify identifications $E/G\cong Y$, there is no
  canonical way of doing this.
\end{remark}

Now let $\cc 3$ be the category of $\Gamma$-equivariant
$\hat G$-bundles over $Y$, with morphisms being $\Gamma$-equivariant
morphisms of $\hat G$-bundles (note that these induce the identity on $Y$). Recall that there is a natural
equivalence of categories between $\cc 2$ and $\cc 3$ given by
pullback, so that necessarily we must have a natural equivalence of
categories between $\cc1$ and $\cc 3$. A candidate map on objects that
makes this equivalence explicit is given as follows: given a twisted
equivariant $G$-bundle $E$ with $\Gamma$-action $\cdot$, consider its
extension of structure group $\hat E:=E\times_{G}\hat G$ to $\hat
G$. Note that $E$ may be regarded as a submanifold of $\hat E$, and
moreover
\begin{equation*}
\hat E=\bigsqcup_{\gamma\in\Gamma} E(1,\gamma).
\end{equation*}

We may define a $\Gamma$-action on $\hat E$ via its restriction to $E$: first define a product
\begin{equation}\label{eq-equivariant-from-twisted-equivariant}
    E\times\Gamma\to\hat E;\,(e,\gamma)\mapsto e*\gamma:= (e\cdot\gamma)(1,\gamma)^{-1}.
\end{equation}
This is $G$-equivariant: for every $g\in G$, $e\in E$ and $\gamma\in\Gamma$,
\begin{align*}
    (eg)* \gamma=((eg)\cdot\gamma)(1,\gamma)^{-1}=
    (e\cdot \gamma) \theta_{\gamma^{-1}}(g)(1,\gamma)^{-1}=(e\cdot\gamma)(1,\gamma)^{-1}g=
    (e*\gamma)g.
\end{align*}
Hence we may extend $*$ to a $\hat G$-equivariant $\Gamma$-action $*$ on $\hat E$. This is an honest group action since, for every $g\in G$, $e\in E$ and $\gamma$ and $\gamma'\in\Gamma$,
\begin{align*}
    (e*\gamma)*\gamma'=((e\cdot\gamma)(1,\gamma)^{-1})*\gamma'=((e\cdot\gamma)\cdot\gamma')(1,\gamma')^{-1}(1,\gamma)^{-1}=\\
    ((ec(\gamma,\gamma'))\cdot\gamma\gamma')((1,\gamma)(1,\gamma'))^{-1}=
    (e\cdot\gamma\gamma')\theta^{-1}_{\gamma\gamma'}(c(\gamma,\gamma'))(c(\gamma,\gamma')(1,\gamma\gamma'))^{-1}=\\
    (e\cdot\gamma\gamma')\theta^{-1}_{\gamma\gamma'}(c(\gamma,\gamma'))(1,\gamma\gamma')^{-1}c(\gamma,\gamma')^{-1}=(e\cdot\gamma\gamma')(1,\gamma\gamma')^{-1}c(\gamma,\gamma')c(\gamma,\gamma')^{-1}=e*\gamma\gamma'.
\end{align*}
Thus we have a map $\ff$ from objects of $\cc1$ to objects of $\cc 3$ mapping $(E,\cdot)$ to $(\hat E,*)$.

%Recall that a morphism of twisted equivariant bundles $E\to F$ is a $\Gamma$-equivariant homomorphism of $G$-bundles $E\to\gamma^*F$ for some $\gamma\in Z(\Gamma)$. The extension of structure group of $\gamma^*F$ to $\hat G$ is equal to the extension $\hat F$ of $F$. Indeed, the twisted equivariant structure on $F$ implies that $\gamma^*F\cong\tg(F)$, where $\tg(F)$ is the $G$-bundle with total space $F$ and $G$-action equal to the original $G$-action on $F$ twisted by $\tg^{-1}$. But this is equal to the $G$-action on $F(1,\gamma)^{-1}\subset\hat F$ given by its natural identification with $F$, since $F(1,\gamma)^{-1}g=F\tg^{-1}(g)(1,\gamma)^{-1}$. By extension of structure group we get a $\Gamma$-equivariant homomorphism of $\hat G$-bundles $\hat E\to\hat F$, so we have defined a functor $\ff:\cc2\to\cc 3$.

\begin{proposition}\label{prop-equivalence-twisted-equivariant-vs-equivariant}
The map $\ff:\ob(\cc1)\to\ob(\cc 3)$ induces an equivalence of categories $\cc1\to\cc3$ which we also call $\ff$. This fits in a commutative diagram
% \begin{equation}\label{eq-commutative-diagram-equivariant-twisted-equivariant-ghat}
%    \begin{tikzcd}
%        \cc1\arrow[r,"\ff"]   &   \cc3\arrow[d,shift left = 1,"/\Gamma"]    \\
%        &\cc2\arrow[ul]\arrow[u,shift left = 1,"p^*"],
%    \end{tikzcd}
%  \end{equation}
\begin{equation}
\label{eq-commutative-diagram-equivariant-twisted-equivariant-ghat}
   \xymatrix{
     \cc1\ar[r]^{\ff}\ar[dr] & \cc3\ar@<1ex>[d]^{/\Gamma} \\
     & \cc2,\ar[u]^{p^*}
   }
\end{equation}
where $p:Y\to X$ is the \'etale cover morphism and the diagonal functor is defined in Proposition \ref{prop:G-Ghat-equivalence}.

\end{proposition}
\begin{proof}
It is enough to prove the commutativity of the  diagram at the level of objects; indeed, this would determine $\ff$ as the composition of the diagonal functor and $p^*$.

Let $q:E_G\to Y$ be a $(\theta,c)$-twisted $\Gamma$-equivariant $G$-bundle over $Y$ and let $E_{\hat G}:=E_G\times_{G}\hat G$ be the extension of structure group of $E_G$ to $\hat G$. Consider the corresponding $\hat G$-bundle over $X$ defined by Proposition \ref{prop:G-Ghat-equivalence}, which we call $E$; recall that the total spaces of $E$ and $E_G$ are equal, the projection of $E_G$ on $X$ is just the composition $p\circ q$ and the $\hat G$-action on $E$ is determined by the original actions of $G$ and $\Gamma$ on $E$ following Proposition \ref{prop:-G-Gamma-twisted-actions}. Recall also that $p^*E:=E\times_{X}Y$ is equipped with the pullback $\Gamma$-equivariant action induced by the right action of $\Gamma$ on $Y$, which we denote with a dot.

We have a natural isomorphism of $\hat G$-bundles
\begin{equation}\label{eq-iso-extension-pullback}
    E_{\hat G}\to p^*E;\,(e,g)\mapsto (eg,q(e)).
\end{equation}
Recall that, if `$\cdot$' denotes the $\Gamma$-action on $E_G$, we have a $\Gamma$-equivariant action $*$ on $E_{\hat G}$ defined by (\ref{eq-equivariant-from-twisted-equivariant}). To finish the proof we need to check that the $\Gamma$-action on $p^*E$ induced by $*$ and (\ref{eq-iso-extension-pullback}) is precisely the pullback action. Indeed, for every $e\in E_G$ and $\gamma\in\Gamma$ we have
\begin{equation*}
    (e,1)*\gamma=(e\cdot\gamma,(1,\gamma)^{-1})\mapsto
    (e\cdot\gamma(1,\gamma)^{-1},q(e\cdot\gamma))=(e,q(e)\cdot\gamma)=\gamma^*(e,q(e)),
\end{equation*}
where the second equation follows from the fact that the action of $\gamma$ on $E_G$ is equal to the action of $(1,\gamma)$ on $E$, together with the $\Gamma$-equivariance of $q$. Thus by $\hat G$-equivariance of both $*$ and the pullback action we have, for every $(e,g)\in E_{\hat G}$ and $\gamma\in\Gamma$,
\begin{equation*}
    (e,g)*\gamma=((e,1)*\gamma) g\mapsto (\gamma^*(e,q(e)))g=\gamma^*(eg,q(e)),
\end{equation*}
so that $*$ induces the pullback action as required.
\end{proof}

Finally, we have the following important result.

\begin{proposition}
  Let $E\to Y$ be a $(\theta,c)$-twisted $\Gamma$-equivariant $G$
  bundle over $Y$ and $\hat{E}=E\to X$ be the corresponding
  $\hat{G}$-bundle over $X$. Let $M$ be a $\hat{G}$-manifold. Then
  $E(M)\to Y$ has the structure of a $\Gamma$-equivariant bundle, and
  there is a one-to-one correspondence between sections of
  $\hat{E}(M)\to X$ and $\Gamma$-equivariant sections of $E(M)\to Y$.
\end{proposition}

\begin{proof}
  In view of Proposition~\ref{prop:twisted-equivariant-hat} there is a
  one-to-one correspondence between sections of $\hat{E}(M)$ and
  twisted $(G,\Gamma)$-equivariant maps $\hat{E}\to M$. The result now
  follows from Propositions~\ref{prop:EM-Gamma-equivariant} and
  \ref{prop:EM-Gamma-euivariant-sections}.
\end{proof}

%%%%%%%%%%%%%%%%%%%%%%%%%%%%%%%%%%%%%%%%%
\subsection{Coverings and monodromy}
\label{sec:coverings-monodromy}
%%%%%%%%%%%%%%%%%%%%%%%%%%%%%%%%%%%%%%%%%

Let $\hat{E}\to X$ be a principal $\hat{G}$-bundle (we do not assume
that $\hat{E}$ is connected). We obtain a principal $\Gamma$-bundle
\begin{displaymath}
  Y:=\hat{E}/G \xra{p} X,
\end{displaymath}
where the $\Gamma=\hat{G}/G$-action is induced by the $\hat{G}$-action.
Since $\Gamma$ is discrete, $p\colon Y\to X$ is a smooth covering of $X$
with covering group $\Gamma$. Notice that $\Gamma$ acts on $Y$ on the
right.  Write
\begin{equation}
  \label{eq:tilde-p}
  \tilde{p}\colon\hat{E}\to Y=\hat{E}/G
\end{equation}
for the quotient map, then, since $G$ is connected, $\tilde{p}$
induces an isomorphism $\pi_0\hat{E}\xra{\cong}\pi_0Y$.

Choose compatible base points in $X$ and $Y$. Fundamental groups will
be taken with respect to these base points and can be identified with
the corresponding covering groups of the (common) universal covering of $X$ and
$Y$. Let $Y'$ be the
connected component of $Y$ containing the base point.
The $\Gamma$-action on $Y$ induces a $\Gamma$-action on $\pi_0Y$. Let
$\Gamma'\subseteq\Gamma$ be the kernel of the corresponding
homomorphism $\Gamma\to\Aut(\pi_0Y)$. Then $Y'\to X$ is a connected
$\Gamma'$-covering. Moreover, we have the exact sequence
\begin{displaymath}
  1 \to \pi_1Y' \to \pi_1X \xra{w} \Gamma' \to 1.
\end{displaymath}
Identifying $\pi_1X$ with the covering group of the universal covering
$\tilde{X}\to X$, the \textbf{monodromy} $w\colon\pi_1X\to\Gamma'$ takes a
covering transformation of $\tilde{X}\to X$ to the induced covering
transformation of $Y'\to X$ which fixes the base point in
$Y'$. Equivalently, if $[\alpha]\in\pi_1X$, then
$w([\alpha])\in\Gamma'$ is the unique element relating the endpoints
of the lift of the loop $\alpha$ starting at the base point of
$Y'$. We shall sometimes refer to
$w\colon\pi_1X\to\Gamma'\subset\Gamma$ as the \textbf{monodromy of the
  $\hat{G}$-bundle $\hat{E}\to X$} and to $\Gamma'$ as the
\textbf{monodromy group of $\hat{E}$}.

\begin{proposition}
  \label{prop:connected-reduction}
  Let $\hat{E}\to X$ be a principal $\hat{G}$-bundle with monodromy
  $w\colon\pi_1X\to \Gamma'\subseteq\Gamma$. Then $\hat{E}$ admits a
  reduction of structure group to $\hat{G}'\subseteq \hat{G}$, where
  $\hat{G}':=G\times_{\theta,c}\Gamma'$ is defined by
  restricting $\theta$ and $c$ to $\Gamma'$. Moreover, the total space
  of the corresponding $\hat{G}'$-bundle $\hat{E}'\subseteq \hat{E}$ is
  connected, and $Y'=\hat{E}'/G\to X$ is a connected
  $\Gamma'$-covering with surjective monodromy
  $w\colon\pi_1X\to\Gamma'$.
\end{proposition}

\begin{proof}
  Let $\hat{E}'=\tilde{p}^{-1}(Y')$, where $\tilde{p}$ was defined in
  \eqref{eq:tilde-p}. Then $\hat{E}'$ is connected and, by
  construction, the $\hat{G}$-action on $\hat{E}$ restricts to a
  $\hat{G}'$-action on $\hat{E}'$ which makes $\hat{E}'\to X$ into a
  principal $\hat{G}'$-bundle.
\end{proof}

We have the following immediate corollary.

\begin{corollary}
  \label{cor:connected-reduction}
  Let $\hat{E}\to X$ be a principal $\hat{G}$-bundle. Then $\hat{E}$
  admits a reduction to the connected component of the identity
  $G\subseteq\hat{G}$ if and only if its monodromy is trivial. \qed
\end{corollary}

In the study of principal $\hat{G}$-bundles on $X$, the first
topological invariant to fix is the monodromy. Moreover,
in view of Proposition~\ref{prop:connected-reduction} we may reduce
to the case when the monodromy group is $\Gamma$ (and the total space
is connected). Putting together the preceding results we
then have the following, which transforms this study into the
study of twisted bundles with connected structure group.

\begin{theorem}\label{category-monodromy}
  Let $w\colon\pi_1X\to\Gamma$ be a fixed surjective homomorphism and
  let $Y\to X$ be the corresponding connected covering.
  The category of principal $\hat{G}$-bundles on $X$ with monodromy
  $w$ is equivalent to the category $\mathcal{C}_1$ of
  $(\theta,c)$-twisted $\Gamma$-equivariant principal $G$-bundles on
  $Y$, defined above.
\end{theorem}

\begin{proof}
  This follows from Proposition~\ref{prop:G-Ghat-equivalence}.
\end{proof}

% \comment{Relate the set $\cM'$ of equivalence classes of
%   $\hat{G}'$-bundles to the set $\cM$ of equivalence classes of
%   $\hat{G}$-bundles.  There is a map $\cM'\to \cM$. Question: Is $\cM$
%   the quotient of $\cM'$ by $\Gamma/\Gamma'$?
%   ]}

Let $\nog$ be the normalizer of $\Gamma'$ in $\Gamma$. Every element
$\gamma\in\nog$ defines an element $(1,\gamma)\in\hatg$ which acts by
conjugation on $\hatgg$ since, for every $(g,\gamma')\in\hatgg$,
$(1,\gamma)(g,\gamma')(1,\gamma)^{-1}$ is equal to
$(x,\gamma\gamma'\gamma^{-1})$ for some $x\in G$; note that the
converse is also true, i.e. if conjugation by $(1,\gamma)$ preserves
$\hatgg$ then $\gamma\in\nog$. This defines an action of $\nog$ on the
set $\cM'$ of equivalence classes of $\hatgg$-bundles by extension of structure group. Moreover, this action preserves the set $\cMg$ of equivalence classes of $\hatgg$-bundles $E$ such that $E/G$ has monodromy equal to $\Gamma'$.

\begin{proposition}
Let $\cM$ be the set of equivalence classes of $\hatg$-bundles, and let $\cMg$ be the set of equivalence classes of $\hatgg$-bundles $E$ such that $E/G$ has monodromy $\Gamma'$. Assume that $G$ is connected. Then the extension of structure group morphism $\cMg\to\cM$ factors through an embedding $\cMg/\nog\hookrightarrow\cM$, where $\nog$ acts by extension of structure group.
\end{proposition}
\begin{proof}
Consider a $\hatgg$-bundle $E$ and $\gamma\in\nog$. The element $s:=(1,\gamma)\in \hatg$ determines an automorphism $\theta:=\Int_{s^{-1}}$ of $\hatgg$ which defines an extension of structure group $\theta(E)$. Let $\hat E$ be the extension of structure group of $E$ by the embedding $\hatgg\hookrightarrow\hatg$. Then the stabilizer of $E$ under the $\hatg$-bundle action is equal to $\hatgg$, which implies that the stabilizer of $Es\subset\hat E$ is equal to $s^{-1}\hatgg s=\hatgg$; in other words, $Es$ determines a reduction of structure group of $\hat E$ to $\hatgg$. Moreover, the map
$$E\to Es;\,e\mapsto es$$
induces an isomorphism of $G$-bundles $\theta(E)\cong Es$. Indeed,
recall that $\theta(E)$ may be regarded as the $\hatgg$-bundle which
has the same total space as $E$ and $G$-action determined by
$$E\times\hatgg\to E;\, (e,g)\mapsto e\theta^{-1}(g).$$
But we have
$$e\theta^{-1}(g)s=esgs^{-1}s=esg,$$
which shows that the induced map $\theta(E)\to Es$ is $\hatgg$-equivariant. This implies that $Es$, which is a reduction of structure group of $\hat E$ to $\hatgg$, is isomorphic to $\theta(E)$; in other words, $\hat E$ is the extension of structure group of $\theta(E)$ by the embedding $\hatgg\hookrightarrow\hatg$.

Conversely, let $E$ and $E'$ be two $\hatgg$-bundles such that the monodromies of $E/G$ and $E'/G$ are both equal to $\Gamma'$; since $G$ is connected, this implies that both $E$ and $E'$ are connected. Assume that they have the same extension of structure group $\hat E$ to $\hatg$. Note that $\hat E$ has an explicit decomposition into connected components, namely
\begin{equation*}
    E=\bigsqcup_{\gamma\Gamma\in \Gamma/\Gamma'}E(1,\gamma),
\end{equation*}
where each coset in $\Gamma/\Gamma'$ has one and only one representative component in the union. Thus $E'$ must be equal to one of these components, say $E(1,\gamma)$ for some $\gamma\in\Gamma$. Let $s:=(1,\gamma)$ and $\theta:=\Int_{s^{-1}}$. The first observation is that, since $\hat E$ is the extension of structure group of $E'$ by the prescribed embedding of $\hatgg$ in $\hatg$, the group $\hatgg$ is equal to the stabilizer of $Es$ by the $\hatg$-bundle action. But, on the other hand, the fact that the stabilizer of $E$ is $\hatgg$ implies that the stabilizer of $Es$ is equal to $s^{-1}\hatgg s=\theta(\hatgg)$; this implies that $\hatgg=\theta(\hatgg)$ or, equivalently, $\gamma\in\nog$. Finally, as in the previous paragraph, the map $E\to Es$ given by the action of $s$ induces an isomorphism of $G$-bundles $\theta(E)\cong Es=E'$, as required.
\end{proof}

%\comment{Compare the topological classification of $\hat{G}$-bundles
%  on $X$ and twisted equivariant bundles on $Y$. --- \textbf{this is now done in
%  Sec.~6(?)}}

%%%%%%%%%%%%%%%%%%%%%%%%%%%%%%%%%%%%%%%%%%%%%%%%%%%%%%
\section{Non-abelian sheaf cohomology and group extensions}
\label{non-abelian-cohomology}
%%%%%%%%%%%%%%%%%%%%%%%%%%%%%%%%%%%%%%%%%%%%%%%%%%%%%%

Let $X$ be  a topological space. We may assume  that $X$ is paracompact and Hausdorff, since  we will be mostly interested in the case in which $X$ is a differenciable real manifold or a complex manifold.
Let $G$ be a topological group, or a real Lie group (respectively a
complex Lie group) if we are working on the smooth category
(respectively holomorphic category).

In this section we briefly recall some basic facts on  non-abelian cohomology, and how this can be used to describe the set of equivalence classes of  $G$-bundles over $X$, since we will elaborate on this in the next section when we  add  twisted equivariant
structures.
A good reference for this material is the paper of
Grothendieck \cite{grothendieck} (see also Hirzebruch \cite{hirzebruch}).

%%%%%%%%%%%%%%%%%%%%%%%%%%%%%%%%%%%%%%%%%%%%%%%%%%%%
\subsection{Non-abelian cohomology and $G$-bundles}
\label{nac-sheaves}
%%%%%%%%%%%%%%%%%%%%%%%%%%%%%%%%%%%%%%%%%%%%%%%%%%%%

A {\bf sheaf} of (not necessarily abelian) groups over $X$ is defined as a triple $\SSS=(S,X,\pi)$, where $S$ is a topological space and $\pi: S\to X$ is a local homeomorphism, so that the stalk $S_x=\pi^{-1}(x)$ over every point $x\in X$ has the structure of a group, and for every  $\alpha$, $\beta$ in $S_x$, the element $\alpha\beta^{-1}$ depends continuously on $\alpha$ and $\beta$. Notice that the fact that $\pi$ is a local homeomorphism implies that the topology of $S$ induces the discrete topology on every stalk  $S_x$.

Recall that if $U$ is a open subset of $X$ we can define the group of sections $H^0(U,\SSS)$ and in particular the group $H^0(X,\SSS)$ of global sections of $\SSS$.

While higher cohomology groups cannot be defined in the case of non-abelian groups, it is nevertheless possible to define a {\bf cohomology set}  $H^1(X,\SSS)$ with a distinguished element. In the abelian case this coincides with the usual first cohomology group of $\SSS$.

To define $H^1(X,\SSS)$, let $\UUU=\{U_i\}_{i\in I}$  be an open covering of $X$. A $\UUU$-cocycle is a function $f$ which associates to each order pair $i,j$ of elements in $I$, an element $f_{ij}\in H^0(U_i\cap U_j,\SSS)$ such that
$$
f_{ij}f_{jk}=f_{ik}\;\;\mbox{in}\;\; U_i\cap U_j \cap U_k\;\;\mbox{for all}\;\;
i,j,k\in I.
$$
The set of $\UUU$-cocycles is denoted by $Z^1(\UUU,\SSS)$. Cocycles $f$ and $f'$ are said to be equivalent if for each $i\in I$ there exists an element $g_i\in H^0(U_i,\SSS)$ such that
$$
f'_{ij}=g_if_{ij}g_j^{-1}\;\;\mbox{in}\;\; U_i\cap U_j \;\;\mbox{for all}\;\; i,j\in I.
$$
The cohomology set $H^1(\UUU,\SSS)$ is the set of equivalence classes of $\UUU$-cocycles, and  the cohomology set $H^1(X,\SSS)$
is defined as a direct limit of the sets  $H^1(\UUU,\SSS)$ as $\UUU$ runs over all open covering of $X$
(see \S 3 of  \cite{hirzebruch} for details).

A  case of particular  interest to us is  that in which $G$ is a group and
$\SSS=(S,X,\pi)$ is the sheaf of germs of functions with values in $G$. This sheaf will be denoted by
$\underline{G}$. The set $S$ here   is the product $X\times G$ and
$\pi$ is the projection onto $X$. If $X$ is a topological space
(respectively  a differentiable or complex manifold)  and  $G$ is a
topological group (respectively real Lie group or complex Lie group),
then $\underline{G}$ is the sheaf for which
$H^0(U,\underline{G})$ is the group of continuous (respectively differentiable or holomorphic) functions.

With $\underline{G}$ as above, the cohomology set $H^1(X,\underline{G})$  parameterizes
equivalence classes of principal $G$-bundles over $X$, in the topological, differentiable or holomorphic category,  according to the structures considered on
$X$ and $G$. This can be easily  seen by considering  an open covering of $X$ and the trivializations and transition functions of a principal bundle for this covering. The transition functions satisfy precisely the cocycle condition. In this case the distinguished element in  $H^1(X,\underline{G})$ is of course the trivial principal $G$-bundle over $X$.

%%%%%%%%%%%%%%%%%%%%%%%%%%%%%%%%%%%%%%%%%%%%%%%%%%%%%%%%%%%%%%%%%%%%%%%%%%%
\subsection{Group extensions and induced exact non-abelian cohomology sequences}
\label{section-extensions-non-abelian-cohomology}
%%%%%%%%%%%%%%%%%%%%%%%%%%%%%%%%%%%%%%%%%%%%%%%%%%%%%%%%%%%%%%%%%%%%%%%%%%%

Let $\Gamma$ be  a finite group, and  $\theta:\Gamma\to\Aut(G)$ be a homomorphism, so that the action of $\Gamma$ on $G$ defined by $\theta$ is continuous.  Let $Z=Z(G)$ be the centre of $G$,  and  $c\in Z^2_{\theta}(\Gamma,Z)$ be a $2$-cocycle.
Consider the  extension of groups
\begin{equation}\label{extension-coh}
  1 \to G \to \hat{G} \to \Gamma \to 1
\end{equation}
where the group structure of $\hat{G}$ is given   by
(\ref{eq:cocycle-product}). Naturally, if $G$ is a real Lie group
(respectively  complex Lie group) we will require that the action of
$\Gamma$ on $G$ be differentiable (respectively  holomorphic).

The extension (\ref{extension-coh}) defines an extension of sheaves
\begin{equation}\label{sheaf-extension}
  1 \to \underline{G} \to \underline{\hat{G}} \to \Gamma \to 1,
\end{equation}
where the sheaves $\underline{G}$ and  $\underline{\hat{G}}$ are  defined as in Section \ref{nac-sheaves}, and we denote the sheaf
$\underline{\Gamma}$ simply by $\Gamma$ since $\Gamma$ is a finite group.

Associated to (\ref{sheaf-extension}) there is a long exact sequence of cohomology sets with distinguished elements
\begin{equation}\label{long-sequence}
H^0(X,\Gamma)\to H^1(X,\underline{G})\to H^1(X,\underline{\hat{G}})  \xra{\pi} H^1(X,\Gamma).
\end{equation}
In Proposition 5.6.2 and following Corollary of \cite{grothendieck}, Grothendieck gives a characterization of the inverse image by $\pi$ of an element of $ H^1(X,\Gamma)$ when this inverse image is nonempty.

Let  $E_\Gamma$ be a  principal $\Gamma$-bundle over $X$ and  $[E_\Gamma] \in H^1(X,\Gamma)$ be  its  corresponding  equivalence class. Assume that $E_\Gamma$ can be lifted to a principal $\hat{G}$-bundle $E_{\hat{G}}$, implying, in particular, that  the inverse image of $[E_\Gamma]$ under $\pi$  is nonempty.  Notice that this is indeed the case
if the cocycle $c$ is trivial and $\hat{G}$ is then the semidirect product.
The group $\hat{G}$ acts by conjugation on $G$, and we consider the bundle of groups $E_{\hat{G}}(G)$ associated to the principal bundle  $E_{\hat{G}}$ via this action. Notice that in the semidirect product case, if $E_{\hat{G}}$ is the extension of structure group of $E_{\Gamma}$ via the natural embedding $\Gamma\hookrightarrow\hat{G}$ then
$E_{\hat{G}}(G)=E_\Gamma(G)$, where $E_\Gamma(G)$ is the bundle associated to $E_\Gamma$ via the action of $\Gamma$ on $G$ given by $\theta$. We can also consider the  adjoint bundle of groups $E_\Gamma(\Gamma)$ associated to $E_\Gamma$ via the adjoint action of $\Gamma$ on itself.

With all this in place, the answer given  by Grothendieck is the following.
\begin{proposition}\label{prop-grothendieck}
With the above notation we have the identification
  $$
  \pi^{-1}([E_\Gamma])= H^1(X,E_{\hat{G}}(G))/H^0(X,E_\Gamma(\Gamma)).
  $$
\end{proposition}
\begin{proof}[Sketch of the proof]
  Throughout this proof we denote by $\pi$ the extension of structure
  group map from the set of isomorphism classes of $\hat G$-bundles to
  the set of isomorphism classes of $\Gamma$-bundles induced by
  (\ref{extension-coh}) by abuse of notation. Fix an isomorphism
  $\pi(E_{\hat G})\cong E_\Gamma$. Let $E'_{\hat G}$ be a
  $\hat G$-bundle over $X$ such that $\pi(E'_{\hat G})\cong E_\Gamma,$
  and fix such an isomorphism
  $f:E'_{\hat G}\xrightarrow{\sim} E_\Gamma$.  Let
  $\UUU=\{U_i\}_{i\in I}$ be a countable open cover of $X$
  trivializing both $E_{\hat G}$ and $E'_{\hat G}$. Let $(U_i,e_i)$
  and $(U_i,e'_i)$ be a system of trivializations of $E_{\hat G}$ and
  $E'_{\hat G}$ respectively whose images under $\pi$ (composed with
  the chosen isomorphisms) is the same trivialization
  $(U_i,\overline e_i)$ of $E_\Gamma$. Set $U_{ij}:=U_i\cap U_j$ and
  let $(U_{ij},g_{ij})$, $(U_{ij},g'_{ij})$ such that $e_j=e_ig_{ij}$
  and $e_j'=e_i'g_{ij}'$ for each $i$ and $j\in I$. Then we have that
  $g'_{ij}=h_{ij}g_{ij}$ for some set of functions
  $h_{ij}:U_{ij}\to G$. We claim that $(U_{ij},(e_i,h_{ij}))$ is a
  1-cocycle in $Z^1(\UUU,E_{\hat{G}}(G))$. Indeed, the fact that both
  $g_{ij}$ and $g'_{ij}$ are 1-cocycles implies:
    \begin{equation*}
        h_{ij}(g_{ij}h_{jk}g_{ij}^{-1})g_{ij}g_{jk}=h_{ij}g_{ij}h_{jk}g_{jk}=h_{ik}g_{ik}=h_{ik}g_{ij}g_{jk},
    \end{equation*}
    thus $h_{ij}g_{ij}h_{jk}g_{ij}^{-1}=h_{ik}$. Hence:
    \begin{align*}
        (e_i,h_{ij})(e_j,h_{jk})=(e_i,h_{ij})(e_ig_{ij},h_{jk})=(e_i,h_{ij})(e_i,g_{ij}h_{jk}g_{ij}^{-1})=(e_i,h_{ij}g_{ij}h_{jk}g_{ij}^{-1})=\\(e_i,h_{ik}),
    \end{align*}
    as required. This identifies the trivialization $(U_i,e_i')$ of $E'_{\hat G}$ with an element of $Z^1(\UUU,E_{\hat G}(G))$. 
    
    Another trivialization of $E'_{\hat G}$ whose image under $f\circ\pi$ is equal to $(U_i,\overline e_i)$ provides an isomorphic 1-cocycle.
    Similarly, given a $G$-bundle $E''_{\hat G}$, an isomorphism $\pi(E''_{\hat G})\cong E_{\Gamma}$ and an isomorphism $E'_{\hat G}\cong E''_{\hat G}$ whose induced automorphism of $E_{\Gamma}$ is the identity, $E''_{\hat G}$ yields an equivalent 1-cocycle. Thus we get a bijection between equivalence classes of pairs $(E'_{\hat G},f:\pi(E'_{\hat G})\xrightarrow{\sim}E_\Gamma)$, where the equivalence relation is an isomorphism of $\hat G$-bundles inducing the identity on $E_\Gamma$, and $H^1(X,E_{\hat{G}}(G))$. If we now consider any induced automorphism of $E_\Gamma$ and we forget the isomorphisms $f$ then we get a bijection between the set of isomorphism classes of $\hat G$-bundles $E'_{\hat G}$ and $H^1(X,E_{\hat{G}}(G))/H^0(X,E_{\Gamma}(\Gamma))$, where $H^0(X,E_{\Gamma}(\Gamma))$ is the group of automorphisms of the $\Gamma$-bundle $E_{\Gamma}$.
\end{proof}

  Notice that if   $E_\Gamma$ is the trivial $\Gamma$-bundle --- the distinguished
  element in  $H^1(X,\Gamma)$ --- then Proposition \ref{prop-grothendieck} follows simply from the exactness of  (\ref{long-sequence}) and
$$
  \pi^{-1}([E_\Gamma])= H^1(X,\underline{G})/H^0(X,\Gamma).
$$

    \begin{remark}
The characterization given by Grothendieck in \cite{grothendieck} applies  in fact  to  any extension of topological groups and actually to any extension of sheaves of groups. Obviously the answer in our situation is particularly simple, considering the special structure of the extension (\ref{extension-coh}).
 \end{remark}

The centre of $G$ is a normal subgroup of $\hat{G}$ and we can consider the quotient group $\check{G}= \hat{G}/Z$. The group $\check{G}$  fits into an extension
  $$
  1 \to G/Z \to \check{G} \to \Gamma \to 1
  $$
  where the group structure  of $\check{G}$ is the semidirect product one.
  If we consider the induced sequence of sheaves,  any
  $[E_\Gamma]\in H^1(X,\Gamma)$ can be lifted now  to an element $[E_{\check{G}}]
\in H^1(X,\underline{\check{G}})$, and the question of when this can be lifted to $H^1(X,\underline{\hat{G}})$ can be analysed in terms of the exact sequence of sheaves
$$
1 \to \underline{Z} \to \underline{\hat{G}} \to \underline{\check{G}} \to 1.
  $$
Since $Z$ is abelian one can define the cohomology group
$H^2(X,\underline{Z})$, and  this sequence induces a long exact sequence
\begin{equation}\label{eq-long-exact-seq-lifting-ghat}
     H^1(X,\underline{\hat{G}}) \to H^1(X,\underline{\check{G}})  \xra{\sigma}
 H^2(X,\underline{Z}).
\end{equation}
The element  $[E_{\check{G}}]$  can be lifted to an element $[E_{\hat{G}}]\in  H^1(X,\underline{\hat{G}})$ if and only if $\sigma([E_{\check{G}}])$ is the neutral element of the group $H^2(X,\underline{Z})$.

\begin{remark}
If $\sigma([E_{\check{G}}])$ is not the neutral element one can still lift
$[E_{\check{G}}]$ to a twisted principal $G$-bundle or equivalently, a torsor over a certain sheaf of groups (see \cite{karoubi} for example), but we will not consider these objects in this paper.
\end{remark}

A connected principal $\Gamma$-bundle $E_\Gamma$ over $X$ is of course the same as a
finite Galois covering $Y:=E_\Gamma\to X$ with Galois group $\Gamma$. In view of the results in Section \ref{coverings-twist}, and in particular Theorem
\ref{category-monodromy}, we conclude the following.

\begin{proposition}\label{grothendieck-twisted}
Let $E_{\hat{G}}$ be a connected $\hat G$-bundle over $X$.

(1) The cohomology set $H^1(X,E_{\hat{G}}(G))$ is in one-to-one correspondence with the set of equivalence classes
  of $(\theta,c)$-twisted $\Gamma$-equivariant $G$-bundles on $Y$, where equivalence means that the induced map on $Y$ is the identity.

(2) The quotient
$H^1(X,E_{\hat{G}}(G))/H^0(X,E_\Gamma(\Gamma))$ is in one-to-one correspondence with the set of equivalence classes of  $(\theta,c)$-twisted $\Gamma$-equivariant $G$-bundles on $Y$, where equivalence allows now  that the induced map on $Y$ be an element of the centre of $\Gamma$.
\end{proposition}
\begin{proof}
    Let $\pi$ be the extension of structure group map from the set of isomorphism classes of $\hat G$-bundles to the set of isomorphism classes of $\Gamma$-bundles induced by (\ref{extension-coh}) by abuse of notation. By the proof of Proposition \ref{prop-grothendieck}, $H^1(X,E_{\hat{G}}(G))$ is in bijection with the set of equivalence classes of pairs $(E'_{\hat G},f:\pi(E'_{\hat G})\xrightarrow{\sim}E_\Gamma)$, where an equivalence between pairs is defined to be an isomorphism of $\hat G$-bundles inducing the identity on $E_\Gamma$. By Proposition \ref{prop:morhisms-fixed-covering} this is in bijection with equivalence classes
    of $(\theta,c)$-twisted $\Gamma$-equivariant $G$-bundles on $Y$, where equivalence means that the induced map on $Y$ is the identity. This proves (1).
    
    Since $E_{\hat G}$ is connected, so is $E_\Gamma$. Thus the image of the monodromy of $E_\Gamma$ is surjective, and (2) follows from Theorem \ref{category-monodromy} and Proposition \ref{prop-grothendieck}.
\end{proof}

In the next section we will give a  characterization of the set of equivalence classes of  $(\theta,c)$-twisted $\Gamma$-equivariant $G$-bundles on $Y$ in terms of non-abelian cohomology from which we can also obtain Proposition \ref{grothendieck-twisted}.

%%%%%%%%%%%%%%%%%%%%%%%%%%%%%%%%%%%%%%%%%%%%%%%%%%%%%%
\subsection{Twisted equivariant bundles as torsors}
\label{torsors}
%%%%%%%%%%%%%%%%%%%%%%%%%%%%%%%%%%%%%%%%%%%%%%%%%%%%%%

Following \cite{damiolini} we offer a generalization of the analysis of Section \ref{section-extensions-non-abelian-cohomology}. We work in the topological category, even though our results are true in the smooth, holomorphic and algebraic categories (the last case is almost a direct consequence of \ref{section-extensions-non-abelian-cohomology}).

Let $X$ be a paracompact and Hausdorff topological space equipped with
a finite group of automorphisms $\Gamma$, and take the quotient space
$p:X\to Y:=X/\Gamma$. Note that we have notationally changed the roles
of $X$ and $Y$, because the basic objects of study are now twisted
equivariant bundles on the covering $X$.

Let $G$ be a topological group with centre $Z$ and consider a homomorphism $\theta:\Gamma\to\Aut(G)$ and a 2-cocycle $c\in Z^2_{\theta}(\Gamma,Z)$.

\begin{definition}
  Given two $(\theta,c)$-twisted $\Gamma$-equivariant $G$-bundles $E$ and $E'$ define the sheaf $\iso(E,E')$ whose local sections are local $G$-bundle isomorphisms from $E$ to $E'$. This inherits a $\Gamma$-equivariant structure, so we denote $\Gamma$-invariant sections over a $\Gamma$-invariant subset of $X$ with a superscript. We say that $E$ and $E'$ have the same \textbf{local type} if $\iso(E,E')\vert_{p^{-1}(y)}^{\Gamma}$ is nonempty for each $y\in Y$.
\end{definition}

\begin{remark}
In the situation where $\Gamma$ acts freely, local types are all the same since the fibre of a $(\theta,c)$-twisted $\Gamma$-equivariant $G$-bundle over $Y$ is isomorphic to $\hat G=G\times_{\theta,c}\Gamma$ by Proposition \ref{prop:principal-Ghat-fixed-cover}.
\end{remark}

See \cite{damiolini} for more properties of local types in the algebraic category.

Given a sheaf of groups $\g$ over $Y$ as defined in Section \ref{nac-sheaves}, we define a \textbf{$\g$-bundle} (sometimes also referred as a  \textbf{$\g$-torsor}) as a sheaf of sets $\mathcal S$ with nonempty stalks such that, for every open neighbourhood $V$ in $Y$, the group $\g(V)$ acts on $\mathcal S(V)$ transitively. This implies that $\mathcal S$ is locally isomorphism to $\g$. An \textbf{isomorphism} of $\g$-bundles is just a $\g$-equivariant morphism of sheaves. The set of isomorphism classes of $\g$-bundles is parametrized by $H^1(Y,\g)$ as defined in Section \ref{nac-sheaves}.

\begin{proposition}\label{prop-damiolini}
The set of isomorphism class of $(\theta,c)$-twisted $\Gamma$-equivariant $G$-bundles over $X$ with the same local type as a given twisted equivariant bundle $E$ is in bijection with the set of isomorphism classes of $\g_E:=(p_*\iso(E,E))^{\Gamma}$-bundles over $Y$. The map is given by $(p_*\iso(E,-))^{\Gamma}$, and the inverse is $p^*(-)\times_{p^*\g_E}E$
\end{proposition}
\begin{proof}
This is identical to the proof of Theorem 3.2 in \cite{damiolini},
after replacing $E,p,\g_E$ and $\iso(E,-)$ by $\mathcal P,\pi,\mathcal
H_{\mathcal P}$ and $\mathscr{I}_{\mathcal P}(-)$, respectively.
\end{proof}

Now fix a $(\theta,c)$-twisted $\Gamma$-equivariant $G$-bundle $E$.
\begin{proposition}\label{prop-ge-iso-eghat}
When $\Gamma$ acts freely on $X$, the sheaf of groups $\g_E$ is isomorphic to the sheaf of groups $E_{\hat G}(G):=E_{\hat G}\times_{\hat G}G$ defined using the conjugacy action of $\hat G$ on $G$, where $E_{\hat G}$ is the $\hat G$-bundle obtained via Proposition  \ref{prop:principal-Ghat-fixed-cover}.
\end{proposition}
\begin{remark}
This shows that Proposition \ref{prop-damiolini} is a generalization of Proposition \ref{grothendieck-twisted}.
\end{remark}
\begin{proof}[Proof of Proposition \ref{prop-ge-iso-eghat}]
Throughout the proof we regard elements of $E$ or $E_{\hat G}$ indistinctly.

Let $V$ be an open neighbourhood of $Y$. Recall that the sheaf $\iso(E,E)$ is isomorphic to $E(G):=E\times_G G$, where $G$ acts on itself by conjugation. Therefore a section in $\g_E(V)$ is just a $\Gamma$-invariant section of $E(G)(U)$, where $U:=p^{-1}(V)$ and $\Gamma$ acts simultaneously on $E$ and $G$ (see Section \ref{associated-bundles}). From such a section $(e,g):U\to E(G)$ we obtain a section of $E(G)$ which is defined for each $y\in V$ as $(e,g)\vert_x$ for any $x\in p^{-1}(y)$. This is independent of the choice of $x$: given $\gamma\in \Gamma$, by $\Gamma$-invariance we have
$$(e,g)\vert_{x\gamma}=(e\cdot\gamma,\tg^{-1}(g))=(e(1,\gamma),\tg^{-1}(g))=(e,g).$$
It is also independent of the representative $(e,g)$, since $(eh,h^{-1}g)=(e,g)$ in $E_{\hat G}(G)$ for every $h\in G$.

Thus we have a map $\g_E\to E_{\hat G}(G)$. The group multiplication is obviously preserved, so we get a morphism of sheaves of groups. The inverse takes a section $s\in E_{\hat G}(G)(V)$ and produces a section of $E(G)(U)$ which, at each $x\in U$, is equal to a representative $(e,g)\in E\times G\vert_y$ of $s$ such that $e\in E\vert_x$. This is independent of the choice of representative because all of them are related by the $G$-action, and it is $\Gamma$-equivariant because
$$(e,g)\vert_{x\gamma}=(e(1,\gamma),\tg^{-1}(g))=(e\cdot\gamma,\tg^{-1}(g))\forevery \gamma\in\Gamma.$$
It is clear that this is indeed the inverse.
\end{proof}

%%%%%%%%%%%%%%%%%%%%%%%%%%%%%%%%%%%%%%%%%%
\section{Twisted equivariant cohomology}
\label{twisted-equivariant-cohomology}
%%%%%%%%%%%%%%%%%%%%%%%%%%%%%%%%%%%%%%%%%%

Let $X$ be a paracompact and Hausdorff connected topological space,
such as a differentiable or complex manifold. Let $G$ be a group with
centre $Z$; if we are working with topological spaces $G$ is a
topological group, whereas it is a real or complex Lie group if we are
in the smooth or holomorphic category, respectively.

Let $\Gamma$ be a discrete group of automorphisms of $X$ with centre $Z(\Gamma)$. We have an associated right action of $\Gamma$ on $X$, given by
$$X\times\Gamma\to X;\, x\mapsto \gamma^{-1}(x).$$
Consider a homomorphism
$$\theta:\Gamma\to\Aut(G);\,\gamma\mapsto\tg$$
and a 2-cocycle $c\in Z^2_{\theta}(\Gamma,Z)$. We provide a (non-abelian) cohomological interpretation for the set of isomorphism classes of $(\theta,c)$-twisted $\Gamma$-equivariant $G$-bundles over $X$. In particular, this allows us to provide cohomological conditions for the existence of $(\theta,c)$-twisted $\Gamma$-equivariant $G$-bundles over $X$. Our method imitates the usual construction of \v Cech cohomology, see for example Chapter V in \cite{grothendieck}.

Recall that we have two natural notions of isomorphisms, depending on
whether they induce the identity on $X$ or an element of $Z(\Gamma)$;
hence we explain the relation between the respective cohomology sets,
which are called twisted equivariant and reduced twisted equivariant,
respectively.

We show that twisted equivariant cohomology over $X$ coincides with usual non-abelian cohomology over $Y:=X/\Gamma$ (see Section \ref{non-abelian-cohomology}) when $\Gamma$ acts freely. Thus, twisted cohomology is in some sense a generalization of non-abelian cohomology.

%%%%%%%%%%%%%%%%%%%%%%%%%%%%%%%%%%%%%%%%%%%%%%%%%%%%%%%%
\subsection{Twisted equivariant bundles and cohomology}
\label{section-definitions}
%%%%%%%%%%%%%%%%%%%%%%%%%%%%%%%%%%%%%%%%%%%%%%%%%%%%%%%%

In this section we define the (reduced) first $(\theta,c)$-twisted $\Gamma$-equivariant cohomology set $H^1_{\Gamma,\theta,c}(X,\ug)$, where $\ug$ is the sheaf of (continuous, smooth, holomorphic) functions with values in $G$. This set will parametrize isomorphism classes of twisted equivariant bundles over $X$.

\textbf{Cochains.} There is a natural right action of $\Gamma$ on the set of open covers of $X$. Let $\UUU:=\{U_i\}_{i\in I}$ be a $\Gamma$-invariant countable open cover of $X$, and set
$$U_{i_1,\dots,i_k}:=U_{i_1}\cap\dots\cap U_{i_k}.$$ The action of $\Gamma$ on $\UUU$ induces a right action on the index $I$, thus an action on $I\times\dots\times I$. First define $\cu{k}$ as the set of $k$-cochains $({\zerocochain}_{i_1,\dots,i_{k+1}})_{i_1<\dots<i_{k+1}\in I}$, where ${\zerocochain}_{i_1,\dots,i_{k+1}}\in H^0(U_{i_1,\dots,i_{k+1}},\ug)$. The group $\Gamma$ acts naturally on $\cu k$ on the left via
$\Gamma\times\cu k\to\cu k;\,(\gamma,({\zerocochain}_{i_1,\dots,i_{k+1}}))\mapsto {\zerocochain}^\gamma:=({\zerocochain}_{(i_1,\dots,i_{k+1})\cdot\gamma}\circ\gamma^{-1}).$

\textbf{From twisted equivariant bundles to cochains.} Let $\su$ be the set of isomorphism classes of $(\theta,c)$-twisted $\Gamma$-equivariant $G$-bundles over $X$ such that the restrictions of the underlying bundle to the open sets $U_i$ are trivial; the isomorphisms considered here respect fibres over $X$.
Take an isomorphism class $\class{(E,\cdot)}\in\su$, together with a representative $(E,\cdot)$ consisting of a $G$-bundle $E$ and a twisted $\Gamma$-action. Take trivializations $s_i:U_i\to E$ and  define ${\onecochain}_{ij}: U_{ij}\to G$ such that $s_i{\onecochain}_{ij}=s_j$, a 1-cocycle in $H^1(\UUU,\ug)$ as defined in Section \ref{nac-sheaves}. For each $\gamma\in\Gamma$, we define $\fgg i:U_i\to G$ to be the holomorphic function that satisfies
\begin{equation}\label{eq-definition-f}
    s_i^{\gamma}\fgg i=s_i\cdot\gamma
\end{equation}
for every $i\in I$, where $\cgamma s_i:=s_{i\cdot\gamma}\circ\gamma^{-1}$. Thus we have obtained a pair $({\onecochain},{\action})\in\cu 1\times\Fun(\Gamma,\cu 0)$, where $\Fun$ denotes the group of maps between two groups that preserve the identity.

In the following we imitate the usual yoga regarding the interplay between \v Cech cohomology and fibre bundles; see Chapter V in \cite{grothendieck}, for example.

\textbf{Dependence on the choice of trivializations.} Another choice
of trivializations is given by $s_i{\zerocochain}_i$ for a suitable
${\zerocochain}:=({\zerocochain}_i)\in\cu 0$, and the obtained pair
using these is
$({\zerocochain}_i^{-1}{\onecochain}_{ij}{\zerocochain}_j,{\zerocochain}_i^{\gamma-1}\fgg
i\tg^{-1}( {\zerocochain}_{i}))$. Similarly, if
$h:(E,\cdot)\to(E',\cdot)$ is an isomorphism of $(\theta,c)$-twisted
$\Gamma$-equivariant $G$-bundles (we are denoting both actions with a
dot by abuse of notation), choose local sections $s_i$ and $s_i'$ for
$E$ and $E'$, respectively, and set $({\onecochain},{\action})$ and $({\onecochain}',{\action}')$ to be the respective pairs as in the previous paragraph. Recall that $h$ is an equivariant isomorphism of $G$-bundles $h:E\to E'$. Define an element ${\zerocochain}:=({\zerocochain}_i)_{i\in I}\in\cu 0$ such that $h(s_i)=s_i'{\zerocochain}_i$. Then we have ${\onecochain}_{ij}={\zerocochain}_i^{-1}{\onecochain}'_{ij}{\zerocochain}_j$, and the compatibility of the actions translates into
$$h(\cgamma s_i)\fgg i=
h(\cgamma s_i\fgg i)=
h(s_i\cdot\gamma)=
h(s_i)\cdot\gamma=
(s'_{i}{\zerocochain}_{i})\cdot\gamma=
(s'_{i}\cdot\gamma)\tg^{-1}(\zerocochain_i)=$$$$
s'^{\gamma}_i{\action}'_{\gamma,i}\tg^{-1}( \zerocochain_i)=
h(\cgamma s_i){\zerocochain}_i^{\gamma-1}{\action}'_{\gamma,i}\tg^{-1}(\zerocochain_i).
$$
We may identify
$({\zerocochain}_i^{-1}{\onecochain}_{ij}{\zerocochain}_j,{\zerocochain}_i^{\gamma-1}\fgg
i\tg^{-1}({\zerocochain}_{i}))$ and
$({\zerocochain}_i^{-1}{\onecochain}'_{ij}{\zerocochain}_j,{\zerocochain}_i^{\gamma
  -1}{\action}'_{\gamma,i}\tg^{-1}( {\zerocochain}_{i}))$ with
$({\onecochain},{\action})\cdot {\zerocochain}$ and
$({\onecochain}',{\action}')\cdot {\zerocochain}$, respectively, where $\cdot$ denotes the right action
\begin{align}
   (\cu 1\times\Fun(\Gamma,\cu 0))\times  \cu 0\to (\cu 1\times\Fun(\Gamma,\cu 0));\nonumber\\
    (({\onecochain},{\action}),{\zerocochain})\mapsto
    (\onecochain\cdot{\zerocochain},\action\cdot{\zerocochain}):=
    ({\zerocochain}_{i}^{-1}  {\onecochain}_{ij} {\zerocochain}_{j},{\zerocochain}_i^{\gamma-1}{\action}_{\gamma,i}\tg^{-1}({\zerocochain}_{i})).\label{eq-action-c0-c1}
\end{align}

\textbf{`Cocycle conditions' on $\varphi$.} The $(\theta,c)$-twisted conditions translate into
$$s_ig\cdot\gamma= s_i\cdot\gamma\tg^{-1}(g)\andd
 (s_i\cdot\gamma')\cdot\gamma= s_ic(\gamma',\gamma)\cdot(\gamma'\gamma)$$
for every $g\in G$, $i\in I$ and $\gamma$ and $\gamma'\in\Gamma$. In particular, we get
$$\cgamma s_i\cgamma{\onecochain}_{ij}\fgg j=\cgamma s_j\fgg j= s_j\cdot\gamma=(s_i{\onecochain}_{ij})\cdot\gamma=( s_i\cdot\gamma)\tg^{-1}({\onecochain}_{ij})=s^{\gamma}_i\fgg i\tg^{-1}( {\onecochain}_{ij}),
$$
i.e.
\begin{equation}\label{eq-compatibility-a-f}
    ({\action}_{\gamma,i})^{-1}\cgamma{\onecochain}_{ij}{\action}_{\gamma,j}=\tg^{-1}( {\onecochain}_{ij}).
\end{equation}
Moreover,
$$s^{\gamma'\gamma}_i{\action}_{\gamma'\gamma,i}\theta_{\gamma'\gamma}^{-1}(c(\gamma',\gamma))=
(s_{i}\cdot(\gamma'\gamma))\theta_{\gamma'\gamma}^{-1}(c(\gamma',\gamma))=
(s_{i}c(\gamma',\gamma))\cdot(\gamma'\gamma)=$$$$
(s_{i}\cdot\gamma')\cdot\gamma=
(s_i^{\gamma'}{\action}_{\gamma', i})\cdot\gamma=
(s_i^{\gamma'}\cdot\gamma)\tg^{-1}({\action}_{\gamma', i})=
s_i^{\gamma'\gamma}\action_{\gamma,i}^{\gamma'}\tg^{-1}( {\action}_{\gamma', i}),
$$
where the equality $s^{\gamma'}_{i}\cdot\gamma= s^{\gamma'\gamma}_i\action_{\gamma,i}^{\gamma'}$ follows from (\ref{eq-definition-f}) after substituting $i$ by $i\cdot \gamma'$ and composing both sides on the right with $\gamma^{-1}$. In other words,
\begin{equation}\label{eq-f-1-cocycle}
    {\action}_{\gamma'\gamma,i}\theta_{\gamma'\gamma}^{-1}(c(\gamma',\gamma))=\action_{\gamma,i}^{\gamma'}\tg^{-1}( {\action}_{\gamma', i}).
\end{equation}
\begin{remark}
    Let $Z^1_{\theta,c}(\Gamma,\cu 0)$ be the set of maps $a:\Gamma\to\cu 0$ satisfying
\begin{equation*}
    a_{\gamma\gamma',i}=a_{\gamma,i}\tg( a^{\gamma}_{\gamma', i})c(\gamma,\gamma')
\end{equation*}
for each $\gamma,\gamma'\in\Gamma$ and $i\in I$. This is very similar to the definition of a 1-cocycle of $\Gamma$ with values in $\cu 0$ of Galois cohomology, where the action of $\gamma\in\Gamma$ is given by $a\mapsto \tg(a^{\gamma})$, with the difference of the 2-cocycle $c$.
Equation (\ref{eq-f-1-cocycle}) implies that $(\tg(\action_{\gamma,i})^{-1})\in Z^1_{\theta,c}(\Gamma,\cu 0)$.
\end{remark}

\textbf{Cocycle conditions and complexes of cochains.} Equations (\ref{eq-compatibility-a-f}) and (\ref{eq-f-1-cocycle}), together with the fact that ${\onecochain}_{ij}$ is a 1-cocycle may be codified using a map
$$d_1:C^1(\UUU,\ug)\times\Fun(\Gamma,\cu 0)\to \cu 2\times\Fun(\Gamma,\cu 1)\times\Fun(\Gamma\times\Gamma,\cu 0);$$$$({\onecochain}_{ij},\fgg i)\mapsto({\onecochain}_{ij}{\onecochain}_{jk}{\onecochain}_{ik}^{-1},({\action}_{\gamma,i})^{-1}\cgamma{\onecochain}_{ij}{\action}_{\gamma,j}\tg^{-1}( {\onecochain}_{ij})^{-1},\action_{\gamma,i}^{\gamma'}\tg^{-1}( {\action}_{\gamma', i}){\action}_{\gamma'\gamma,i}^{-1}),$$
where we are using $\Fun$ to denote maps between groups that preserve the identity.
Let $\theta^{-1}(c)\in \Fun(\Gamma\times\Gamma,\cu 0)$ be the map
$$\Gamma\times\Gamma\ni(\gamma,\gamma')\mapsto\theta^{-1}_{\gamma'\gamma}(c(\gamma',\gamma)):X\to Z,$$
which is a constant function (we think of this as an element of $\cu 0$ by restricting to the open neighbourhoods $U_i$). Then we define the set of \textbf{$(\theta,c)$-twisted $\Gamma$-equivariant 1-cocycles} $Z^1_{\Gamma,\theta,c}(\UUU,\ug):=d_1^{-1}(1,1,\theta^{-1}(c))\subseteq\cu{1}\times\Fun(\Gamma,\cu 0)$, equivalently defined as the set of pairs $({\onecochain},{\action})\in\cu 1\times\Fun(\Gamma,\cu 0)$ such that ${\onecochain}$ is a 1-cocycle in the sense of Section \ref{non-abelian-cohomology} and ${\action}$ satisfies (\ref{eq-compatibility-a-f}) and (\ref{eq-f-1-cocycle}).

\textbf{Action of 0-cochains and cohomology.} The action of $\cu 0$ on $\cu 1\times\Fun(\Gamma,\cu 0)$ given by (\ref{eq-action-c0-c1}) preserves $\zuc$ since, for every $({\onecochain},{\action})\in\zuc$ and ${\zerocochain}\in\cu{0}$, we have
\begin{align*}
({\action}\cdot {\zerocochain})^{-1}_{\gamma,i}({\onecochain}\cdot {\zerocochain})^{\gamma}_{ij}({\action}\cdot {\zerocochain})_{\gamma,j}=
({\zerocochain}_{i}^{\gamma-1}\fgg {i}\tg^{-1}({\zerocochain}_i))^{-1}({\zerocochain}_i^{\gamma-1} {\onecochain}^{\gamma}_{ij} {\zerocochain}^{\gamma}_j){\zerocochain}_{j}^{\gamma-1}\fgg {j}\tg^{-1}({\zerocochain}_{j})=\\
\tg^{-1}({\zerocochain}_{ i})^{-1}{\action}_{\gamma,i}^{-1} {\onecochain}_{ij}^{\gamma}\fgg { j}\tg^{-1}({\zerocochain}_j)=
\tg^{-1}({\zerocochain}_{i}^{-1}{\onecochain}_{ij}{\zerocochain}_j)=
\tg^{-1}({\onecochain}\cdot {\zerocochain}_{ij})
\end{align*}
and
\begin{align*}
({\action}\cdot {\zerocochain})_{\gamma,i}^{\gamma'}\tg^{-1}(({\action}\cdot {\zerocochain})_{\gamma',i})=
{\zerocochain}_{i}^{\gamma'\gamma-1}\fgg i^{\gamma'}\tg^{-1}({\zerocochain}^{\gamma'}_{i})
\tg^{-1}({\zerocochain}_{i}^{\gamma'-1}{\action}_{\gamma',i}\theta_{\gamma'}^{-1}({\zerocochain}_i))=\\
{\zerocochain}_{i}^{\gamma'\gamma-1}\fgg i^{\gamma'}
\tg^{-1}({\action}_{\gamma',i})\theta_{\gamma'\gamma}^{-1}({\zerocochain}_i))=
({\action}\cdot {\zerocochain})_{\gamma'\gamma,i}\theta_{\gamma'\gamma}^{-1}(c(\gamma',\gamma)).
\end{align*}
We define the \textbf{first $(\theta,c)$-twisted $\Gamma$-equivariant cohomology set over $\UUU$ with values in $G$} to be $$\huc:=\zuc/\cu 0.$$
We thus get a map $\su\to\huc$. In what follows we prove that it is a bijection.

\textbf{Surjectivity.} Conversely, given an element $({\onecochain},{\action})\in\zuc$, we may define a $G$-bundle $E$ using the 1-cocycle ${\onecochain}$, together with local trivializations $s_i:U_i\to E$ such that $s_i {\onecochain}_{ij}=s_j$. Then we define the action of $\Gamma$ on $E$ by $(s_{i}g)\cdot\gamma:=s_{i}^{\gamma}{\action}_{\gamma,i}\tg^{-1}(g)$ for every $g\in G$, $i\in I$ and $\gamma\in\Gamma$. To see that this is well defined note that, by (\ref{eq-compatibility-a-f}),
$$(s_i{\onecochain}_{ij})\cdot\gamma=s^{\gamma}_{i}{\action}_{\gamma,i}\tg^{-1}({\onecochain}_{ij})=s^{\gamma}_{i}{\onecochain}^{\gamma}_{ij}{\action}_{\gamma,j}=s^{\gamma}_{j}{\action}_{\gamma,j}=s_j\cdot\gamma.$$
The $\Gamma$-action is compatible with the action of $G$ twisted by $\theta$ by definition, and by (\ref{eq-f-1-cocycle}) we have
$$(s_{i}\cdot\gamma')\cdot\gamma=(s_i^{\gamma'}{\action}_{\gamma', i})\cdot\gamma=s_i^{\gamma'\gamma}\action^{\gamma'}_{\gamma,i}\tg^{-1}({\action}_{\gamma', i}))=s_i^{\gamma'\gamma}{\action}_{\gamma'\gamma,i}\theta_{\gamma'\gamma}^{-1}(c(\gamma',\gamma))=(s_{i}c(\gamma',\gamma))\cdot(\gamma'\gamma),$$
so that the action is $(\theta,c)$-twisted. It is easy to see that the element of $\hu$ corresponding to the class of $(E,\cdot)$ is equal to the class of $({\onecochain},{\action})$, so the constructed map is surjective.

\textbf{Injectivity.} Let $(E,\cdot)$ and $(E',\cdot)$ be two $(\theta,c)$-twisted $\Gamma$-equivariant $G$-bundles. Take local sections $s_i:U_i\to E$ and $s_i':U_i\to E'$ and define $({\onecochain},{\action})$ and $({\onecochain}',{\action}')\in\zuc$ as above. Assume that $({\onecochain},{\action})=({\onecochain}',{\action}')\cdot {\zerocochain}$ for some ${\zerocochain}\in\cu 0$. Then there is an equivariant isomorphism $h:E\to E'$, defined by $h(s_i)=s_i'\zerocochain_i$. This is well defined as usual, and we need to show that it is compatible with the $\Gamma$-actions:
$$h( s_i\cdot\gamma)=h(\cgamma s_i\fgg i)=h(\cgamma s_i)\fgg i= s_i'^{\gamma}\zerocochain^{\gamma}_i\fgg i=s'^{\gamma}_i{\action}'_{\gamma,i}\tg^{-1}(\zerocochain_{i})=( s'_i\zerocochain_{ i})\cdot\gamma=h(s_i)\cdot\gamma,$$
as required.

\textbf{Refinements.} Thus we have obtained a bijection $\su\cong\huc$. Another countable $\Gamma$-invariant open cover $\VVV:=\{V_j\}_{j\in J}$ of $X$ is called a \textbf{refinement} of $\UUU$ if there is a map $t:J\to I$ such that $V_{j}\subseteq U_{t(j)}$ for every $j\in J$. There are  "restriction maps" $\cu 0\to C^0(\VVV,\ug)$ and $\cu 1\times\Fun(\Gamma,\cu 0)\to C^1(\VVV,\ug)\times\Fun(\Gamma,C^0(\VVV,\ug))$ depending on $t$. Moreover, there is an induced map in cohomology $\huc\to H_{\Gamma,\theta,c}^{1}(\VVV,\ug)$. It can be seen (for example, from geometric arguments using the bijections) that this map is independent of $t$ and so the set $\{\huc\}_{\UUU}$ is inductive.

\begin{definition}\label{def-twisted-equivariant-cohomology}
  The \textbf{first $(\theta,c)$-twisted $\Gamma$-equivariant cohomology set over $X$ with values in $G$}, denoted $\hxc$, is the inductive limit of the sets $\huc$ over the system of $\Gamma$-invariant open covers $\UUU$ on $X$.
\end{definition}

The map $\huc\to H_{\Gamma,\theta,c}^{1}(\VVV,\ug)$ induces the natural inclusion of $\su$ in $ \sv$. Using that $\sx$ is the inductive limit of $\{\su\}_{\UUU}$ given by these inclusions gives the following.

\begin{proposition}\label{prop-bijection-hx-bundles}
  Consider the notion of isomorphism of twisted equivariant bundles
  descending to the identity on $X$. Let $\UUU$ be a countable
  $\Gamma$-invariant open cover of $X$. Then there is a natural
  bijection between the set of isomorphism classes of
  $(\theta,c)$-twisted $\Gamma$-equivariant principal $G$-bundles over
  $X$ which are trivial on the open sets of $U$ and $\huc$. This
  induces a bijection between the set of isomorphism classes of
  $(\theta,c)$-twisted $\Gamma$-equivariant principal $G$-bundles over
  $X$ and $\hxc$.
\end{proposition}

%%%%%%%%%%%%%%%%%%%%%%%%%%%%%%%%%%%%%%%%%%%%%%%%%%%%%%%%%%%%%%%%%%%%%%%%%%%%%%%%%%%%%%%%%%%%%%%
\subsection{Isomorphisms descending to $Z(\Gamma)$ and reduced twisted equivariant cohomology}\label{reduced-twisted-cohomology}
%%%%%%%%%%%%%%%%%%%%%%%%%%%%%%%%%%%%%%%%%%%%%%%%%%%%%%%%%%%%%%%%%%%%%%%%%%%%%%%%%%%%%%%%%%%%%%%

Recall from Section \ref{equivalence-of-categories} that we also have
a notion of isomorphism of twisted equivariant bundles on $X$
descending to an element of $Z(\Gamma)$ (rather than the identity). We
give a cohomological interpretation of these objects and relate it to
Definition \ref{def-twisted-equivariant-cohomology}.

First we define a $Z(\Gamma)$-action on $\huc$ as follows: the natural
left action of $\Gamma$ on $\cu 0$ that we have defined above
determines a semidirect product $\cu 0\rtimes Z(\Gamma)$. We extend
the $\cu 0$-action on $\zuc$ to a right
$\cu 0\rtimes Z(\Gamma)$-action:
\begin{align*}
    (\cu 1\times\Fun(\Gamma,\cu 0))\times  (\cu 0\rtimes Z(\Gamma))\to\\ (\cu 1\times\Fun(\Gamma,\cu 0));\\
    (({\onecochain},{\action}),({\zerocochain},\lambda))\mapsto
    (\onecochain\cdot({\zerocochain},\lambda),\action\cdot({\zerocochain},\lambda)):=
    ({\zerocochain}_{i}^{\lambda-1}\clambda  {\onecochain}_{ij} {\zerocochain}^{\lambda}_{j},{\zerocochain}_i^{\lambda\gamma-1}{\action}^{\lambda}_{\gamma,i}\tg^{-1}({\zerocochain}^{\lambda}_{i})).
\end{align*}
The inclusion of $Z(\Gamma)$ in $\cu 0\rtimes Z(\Gamma)$ as a set then induces a genuine $Z(\Gamma)$-group action on $\huc$, which we also denote with a dot.

\begin{definition}\label{def-twisted-equivariant-reduced-cohomology}
We define the \textbf{first $(\theta,c)$-twisted $\Gamma$-equivariant reduced cohomology set over $\UUU$ with values in $G$} $$\thuc:=\zuc/\cu 0\rtimes Z(\Gamma)=\huc/Z(\Gamma).$$ The inductive structure on $\{\huc\}_{\UUU}$ also makes $\{\thuc\}_{\UUU}$ inductive, and we define the \textbf{first $(\theta,c)$-twisted $\Gamma$-equivariant reduced cohomology set over $X$ with values in $G$}, denoted $\thxc$, to be the corresponding inductive limit.
\end{definition}

Let $\tsu$ be the set of isomorphism classes of $(\theta,c)$-twisted $\Gamma$-equivariant $G$-bundles which are trivial over $\UUU$, where now isomorphisms induce the action of any element of $Z(\Gamma)$ over $X$. Note that $\tsu=\su/Z(\Gamma)$, the quotient of $\su$ by the pullback action of $Z(\Gamma)$. Given a class $\class{(E,\cdot)}\in\su$, let $\class{(\onecochain,\action)}$ be the corresponding class in $\huc$. It is then clear that, for each $\gamma\in Z(\Gamma)$, the class of $\huc$ corresponding to $\gamma^*\class{(E,\cdot)}$ is
$$\class{(\onecochain^{\gamma^{-1}},\action^{\gamma^{-1}})}=\class{(\onecochain,\action)}\cdot\gamma^{-1}.$$
This shows that the natural bijection in Proposition \ref{prop-bijection-hx-bundles} induces a one-to-one correspondence $\tsu\cong\thuc$. As $\UUU$ varies,
this induces a bijection
$$
\tsx\cong \thxc.
$$
Summing up, we have the following.

\begin{proposition}\label{prop-bijection-hx-bundles-reduced}
  Let $\UUU$ be a countable $\Gamma$-invariant open cover of $X$. Then
  there is a bijection between the set of isomorphism classes of
  $(\theta,c)$-twisted $\Gamma$-equivariant principal $G$-bundles over
  $X$ which are trivial on the open sets of $\UUU$ and $\thuc$. This
  induces a bijection between the set of isomorphism classes of
  $(\theta,c)$-twisted $\Gamma$-equivariant principal $G$-bundles over
  $X$ and $\thxc$. Here we are considering all the isomorphisms
  inducing the action of elements of $Z(\Gamma)$ on $X$.
\end{proposition}

%%%%%%%%%%%%%%%%%%%%%%%%%%%%%%%%%%%%%%%%%%%%%%%%%%%%%%%%%%%%%%%%%%%%%%%%%
\subsection{Long exact sequences in twisted equivariant cohomology.}
\label{twisted-long-exact-sequences}
%%%%%%%%%%%%%%%%%%%%%%%%%%%%%%%%%%%%%%%%%%%%%%%%%%%%%%%%%%%%%%%%%%%%%%%%%

In this section we use long exact sequences in twisted equivariant cohomology to test the existence of $(\theta,c)$-twisted $\Gamma$-equivariant $G$-bundles over $X$. In the process we give an alternative interpretation of these objects as lifts of certain twisted $\Gamma$-equivariant $G/Z$-bundles over $X$.

We shall henceforth drop $c$ from the notation whenever it is trivial.

Let $\hhx$ be the group of $\Gamma$-equivariant holomomorphic functions $X\to G$, i.e. functions ${\zerocochain}:X\to G$ such that ${\zerocochain}(\gamma(x))=\tg({\zerocochain}(x))$ for each $x\in X$. %The natural pullback action of $Z(\Gamma)$ determines a semidirect product $\hhu\rtimes Z(\Gamma)$.

Let $G'$ be a second (topological, real Lie, complex Lie) group equipped with homomorphisms $\theta':\Gamma\to\Aut(G')$ and $\tau:G\to G'$ satisfying $\tau\tg=\theta'_{\gamma}\tau$ for every $\gamma\in\Gamma$. We then have induced maps $\hhx\to H_{\Gamma,\theta'}^{0}(X,\ug')$, $\hx\to H_{\Gamma,\theta'}^{1}(X,\ug')$, $\hx\to \widetilde H_{\Gamma,\theta'}^{1}(X,\ug')$ and $\thx\to \widetilde H_{\Gamma,\theta'}^{1}(X,\ug')$ sending the class of $({\onecochain},{\action})\in\zu$ to the class of $(\tau(g),\tau({\action}))$. These are morphisms of pointed sets, where the distinguished elements are the classes of pairs $(1,1)$; these correspond, under the identification of Proposition \ref{prop-bijection-hx-bundles}, to the isomorphism classes of the respective trivial bundles with the natural $\theta$-twisted $\Gamma$-actions.

In particular we may consider the extension
$$1\to Z\to G\to G/Z\to 1;$$
note that $\theta$ preserves $Z$, thus inducing an automorphism of $G/Z$ which we also call $\theta$.
This induces an exact sequence
$$1\to \hhxz\to\hhx\to\hhxgz;$$
note that these are groups and $\hhxz$ is a normal subgroup of $\hhx$ (actually it is its centre), so that exactness in the middle term is understood in the sense that the last homomorphism induces a group embedding $\hhx/\hhxz\hookrightarrow\hhxgz$.

In the following we emulate the usual  construction of long exact sequences in \v Cech cohomology, see for example \cite{grothendieck}.

\textbf{The coboundary map from sections to first cohomology.} To study the image of the last homomorphism we define a right action of $\hhxgz$ on $\hxz$ as follows: let $\overline {\zerocochain}$ be a $\Gamma$-equivariant holomorphic map $\overline {\zerocochain}:X\to G/Z$. Given a class in $\hxz$, we may find a countable invariant open cover $\UUU=\{U_i\}_{i\in I}$ such that the class has a representative $({\onecochain},{\action})\in\zuz$ and the restriction of $\overline a$ to $U_i$ has a lift to a holomorphic function ${\zerocochain}_i:U_i\to G$ for every $i\in I$. If ${\zerocochain}:=({\zerocochain}_i)_{i\in I}\in\cu 0$ then $\overline {\zerocochain}\in\hhxgz$ sends $({\onecochain},{\action})$ to
$({\onecochain},{\action})\cdot{\zerocochain}$, where we regard $({\onecochain},{\action})$ as an element of $\zu$ and `$\cdot$' denotes the action defined in \ref{eq-action-c0-c1}.

Note that this is an element of $\cuz 1\times\Fun(\Gamma,\cuz 0)$, since ${\zerocochain}_j\vert_{U_{ij}}={\zerocochain}_iz_{ij}\vert_{U_{ij}}$ for some $z_{ij}\in\cuz 1$ and the $\Gamma$-equivariance of $\overline {\zerocochain}$ implies that $\tg^{-1}(\cgamma {\zerocochain}_i)={\zerocochain}_iz_i$ for some $z_i\in\cuz 0$. Moreover the action of $\cu0$ preserves $\zu$, so that $({\onecochain},{\action})\cdot{\zerocochain}\in\zuz$. The class of $({\onecochain},{\action})\cdot{\zerocochain}$ in $\huz$ is independent of the choice of ${\zerocochain}$: another set of lifts is given by ${\zerocochain}z$ for a suitable $z=(z_i)_{i\in I}\in\cuz 0$, so that
$$({\onecochain},{\action})\cdot{\zerocochain}z=(({\onecochain},{\action})\cdot{\zerocochain})\cdot z.$$
Similarly, the class of $(({\onecochain},{\action})\cdot z)\cdot{\zerocochain}=(({\onecochain},{\action})\cdot{\zerocochain})\cdot z$ is equal to the class of $({\onecochain},{\action})\cdot{\zerocochain}$ for every $z\in \cuz 0$.

Hence we get a right action of $\hhxgz$ on $\hxz$ which we also call $\cdot$. Considering the action on the trivial element of $\hxz$ we get a \textbf{coboundary map} $\delta:\hhxgz\to\hxz$.

\textbf{Second cohomology.} Since $Z$ is abelian, we may further define a second cohomology group with values in $Z$: first we define a group homomorphism
\begin{align*}
d_2:\cuz 2\times\Fun(\Gamma,\cuz 1)\times\Fun(\Gamma\times\Gamma,\cuz 0)\to \\
\cuz 3\times\Fun(\Gamma,\cuz 2)\times\Fun(\Gamma\times\Gamma,\cuz 1)\times \Fun(\Gamma\times\Gamma\times\Gamma,\cuz 0);\\
(u,v,w)\mapsto (u_{ijk}u_{ikl}u_{ijl}^{-1}u_{jkl}^{-1},u^{\gamma-1}_{ijk}\tg^{-1}(u_{ijk})v_{ij}v_{jk}v_{ik}^{-1},\\
v^{\gamma'}_{ij}\tg^{-1}( v_{\gamma',ij})v^{-1}_{\gamma'\gamma,ij}w_{\gamma,i}w_{\gamma,j}^{-1},\tg^{-1}( w_{\gamma',\gamma'',i})w_{\gamma,\gamma''\gamma',i}(w_{\gamma'\gamma,\gamma'',i}w^{\gamma''}_{\gamma,\gamma',i})^{-1}).
\end{align*}
We set $Z^2_{\Gamma,\theta}(\UUU,\uz):=\ker d_2$.

Note that the condition that the first component is equal to 1 is equivalent to $u$ being a 2-cocycle, and the condition that the last component is equal to 1 is equivalent to $(\theta_{\gamma\gamma'}(w_{\gamma',\gamma,i}))$ being an element of $Z^2_{\theta}(\Gamma,\cuz 0)$, where $\Gamma$ acts via $\theta$ and its natural action on $I$. There is an action of $\cu{1}\times\Fun(\Gamma,\cu 0)$ on $\cuz 2\times\Fun(\Gamma,\cuz 1)\times\Fun(\Gamma\times\Gamma,\cuz 0)$, such that the result of the action of $({\onecochain},{\action})$ on $(u,v,w)$ is equal to $d_1({\onecochain},{\action})(u,v,w)$.
It can be seen that $d_2d_1$ is the trivial map, hence this action preserves $Z^2_{\Gamma,\theta}(\UUU,\uz)$.

\begin{definition}\label{def-second-twisted-equivariant-cohomology}
  We set $H^2_{\Gamma,\theta}(\UUU,\uz):=\ker d_2/\cu{1}\times\Fun(\Gamma,\cu 0)$, the \textbf{second $(\theta,c)$-twisted $\Gamma$-equivariant cohomology set with values in $Z$}. We define $H^2_{\Gamma,\theta}(X,\uz)$ to be equal to the corresponding inductive limit.
\end{definition}
%and $\widetilde H^2_{\Gamma,\theta}(\UUU,\uz):=\ker d_2/\cu{1}\times\Fun(\Gamma,\cu 0)\rgamma$. We define $\widetilde H^2_{\Gamma,\theta}(X,\uz)$ and $H^2_{\Gamma,\theta}(X,\uz)$ to be equal to the corresponding inductive limits.

\textbf{Coboundary map from first to second cohomology.} We now define a coboundary map $\delta:\hxgz\to H^2_{\Gamma,\theta}(X,\uz)$ as follows: for each class in $\hxgz$ take a representative $(\overline {\onecochain},\overline {\action})\in\zugz$, lift it to an element $({\onecochain},{\action})$ of $\cu{1}\times\Fun(\Gamma,\cu 0)$ and calculate $d_1({\onecochain},{\action})$ (we choose the countable $\Gamma$-invariant open cover $\UUU$ so that this is possible). This is an element of $Z^2_{\Gamma,\theta}(\UUU,\uz)$ (taking the quotient by $Z$ commutes with $d_1$ and $d_2d_1=1$). Moreover, its class only depends on the class of $(\overline {\onecochain},\overline {\action})$.

Indeed, let $(\overline {\onecochain}',\overline {\action}')$ be another representative with a lift $({\onecochain}',{\action}')$ in the group $\cu 0\times\Fun(\Gamma,\cu 0)$ and consider an element $\overline{\zerocochain}\in\cugz 0$ with a lift ${\zerocochain}\in\cu 0$ such that $(\overline {\onecochain}',\overline {\action}')=(\overline {\onecochain},\overline {\action})\cdot\overline {\zerocochain}$. Then there exists $({\onecochain}_0,{\action}_0)\in\cuz 0\times\Fun(\Gamma,\cuz 0)$ such that $({\onecochain}',{\action}')=(({\onecochain},{\action})\cdot {\zerocochain})({\onecochain}_0,\action_0)$, which implies that $d_1({\onecochain}',{\action}')=d_1(({\onecochain},{\action})\cdot {\zerocochain})d_1({\onecochain}_0,{\action}_0)$. But the components of $d_1(({\onecochain},{\action})\cdot {\zerocochain})$ are
$$(d_1(({\onecochain},{\action})\cdot {\zerocochain}))_1={\zerocochain}_{i}^{-1}{\onecochain}_{ij}{\zerocochain}_{j}{\zerocochain}_{j}^{-1}{\onecochain}_{jk}{\zerocochain}_{k}{\zerocochain}_{k}^{-1}{\onecochain}_{ik}^{-1}{\zerocochain}_{i}={\zerocochain}_i^{-1}{\onecochain}_{ij}{\onecochain}_{jk}{\onecochain}_{ik}^{-1}{\zerocochain}_i={\onecochain}_{ij}{\onecochain}_{jk}{\onecochain}_{ik}^{-1},$$
\begin{align*}
    (d_1(({\onecochain},{\action})\cdot {\zerocochain}))_2=({\zerocochain}_{i}^{\gamma-1}\fgg {i}\tg^{-1}( {\zerocochain}_{ i}))^{-1}{\zerocochain}_{i}^{\gamma-1}{\onecochain}^{\gamma}_{ij}{\zerocochain}^{\gamma}_{j}{\zerocochain}_{j}^{\gamma-1}\fgg {j}\tg^{-1}({\zerocochain}_{ j})\tg^{-1}({\zerocochain}_{i}^{-1} {\onecochain}_{ij} {\zerocochain}_{ j})^{-1}=\\
    \tg^{-1}({\zerocochain}_i)^{-1}{\action}^{\gamma-1}_{\gamma,i}{\onecochain}^{\gamma}_{ij}\fgg {j}\tg^{-1}({\onecochain}_{ij})^{-1}\tg^{-1}( {\zerocochain}_{i})={\action}^{\gamma-1}_{i}{\onecochain}_{ij}\fgg {j}\tg^{-1}({\onecochain}_{ij}^{\gamma})^{-1}
\end{align*}
and
\begin{align*}
(d_1(({\onecochain},{\action})\cdot {\zerocochain}))_3= {\zerocochain}^{\gamma'\gamma-1}_{i}\fgg {i}^{\gamma'}\tg^{-1}({\zerocochain}_i^{\gamma'})\tg^{-1}({\zerocochain}_{i}^{\gamma'-1}{\action}_{\gamma',i}\theta_{\gamma'}^{-1}({\zerocochain}_{ i}))({\zerocochain}_{i}^{\gamma'\gamma-1}{\action}_{\gamma'\gamma,i}\theta_{\gamma'\gamma}^{-1}({\zerocochain}_{i}))^{-1}=\\
{\zerocochain}_{i}^{\gamma'\gamma-1}\fgg i^{\gamma'}\tg^{-1}({\action}_{\gamma', i}){\action}_{\gamma'\gamma,i}^{-1}{\zerocochain}_{i}^{\gamma'\gamma}=
\fgg i^{\gamma'}\tg^{-1}({\action}_{\gamma', i}){\action}_{\gamma'\gamma,i}^{-1},
\end{align*}
where the last equalities follow from $d_1(\zerocochain,\action)\in Z^2_{\Gamma,\theta}(\UUU,\uz)$.
Thus $d_1(({\onecochain},{\action})\cdot {\zerocochain})=d_1({\onecochain},{\action})$ and so $d_1({\onecochain},{\action})$ and $d_1({\onecochain}',{\action}')$ are in the same class of $H^2_{\Gamma,\theta}(X,\uz)$.

\begin{proposition}\label{prop-long-exact-seq-Z}
The sequence
\begin{align}\label{eq-long-exact-seq-G-Z}
    1\to \hhxz\to\hhx\to\hhxgz\xrightarrow{\delta}\\
    \notag\xrightarrow{\delta}\hxz\to\hx\to\hxgz\xrightarrow{\delta} H^2_{\Gamma,\theta}(X,\uz)
\end{align}
is exact. Moreover, given a 2-cocycle $c\in Z^2_{\theta}(\Gamma,Z)$, the image of the map $\hxc\to\hxgz$ induced by the quotient is equal to the preimage of the class of $(1,1,\theta^{-1}(c))\in Z^2_{\Gamma,\theta}(X,\uz)$ in $H^2_{\Gamma,\theta}(X,\uz)$ under the coboundary map $\delta$. Here $\theta^{-1}(c)(\gamma,\gamma'):=\theta^{-1}_{\gamma'\gamma}(c(\gamma',\gamma))$ is regarded as a constant function from $X$ to $Z$ for each $\gamma$ and $\gamma'\in\Gamma$.
\end{proposition}

\begin{remark}
Proposition \ref{prop-long-exact-seq-Z} should be understood as follows: call $A_k$ to the $k$-th term in the sequence ($A_0=1$, etc.). Then the preimage of the distinguished element of $A_{k+2}$ is equal to the image of $A_k$. A stronger statement is proven whenever $A_k$ is a group since in this case, by the discussion above, there is an action of $A_k$ on $A_{k+1}$: two elements in $A_{k+1}$ have the same image if and only if they are in the same orbit. Here the action of $\hhx$ on $\hhxgz$ is induced by multiplication on the left.
\end{remark}

The long exact sequence (\ref{eq-long-exact-seq-G-Z}) is a twisted equivariant version of (\ref{eq-long-exact-seq-lifting-ghat}) when $\Gamma$ acts freely. In this case $p:X\to Y:=X/\Gamma$ is an \'etale cover and we know by Proposition \ref{prop:principal-Ghat-fixed-cover} that $\hat G$-bundles over $Y$ are in correspondence with $(\theta,c)$-twisted $\Gamma$-equivariant $G$-bundles over $X$, whose isomorphism classes are parametrized by twisted equivariant cohomology by Propositions \ref{prop-bijection-hx-bundles} and \ref{prop-bijection-hx-bundles-reduced}; the group $\hat G:=G\times_{\theta,c}\Gamma$ is defined in Section \ref{sec:extensions-lie}. Thus, in particular, (\ref{eq-long-exact-seq-G-Z}) answers the question of existence of $\hat G$-bundles over $Y$ with associated $\Gamma$-bundle equal to $X$ by considering them as lifts of $\theta$-twisted $\Gamma$-equivariant $G/Z$-bundles over $X$, which are in correspondence with $\check{G}:=\hat G/Z$-bundles over $Y$. We state this more precisely.

\begin{corollary}\label{cor-comparison-twisted-equivariant-long-exact-seq}
  The set of $\hat G$-bundles over $Y$ with associated $\Gamma$-bundle $X$ is nonempty if and only if $(1,1,\theta^{-1}(c))$ is in the image of $\hxgz\xrightarrow{\delta} H^2_{\Gamma,\theta}(X,\uz)$. Here,   for $\gamma$ and $\gamma'\in\Gamma$, we define
  $\theta^{-1}(c)(\gamma,\gamma'):=\theta^{-1}_{\gamma'\gamma}(c(\gamma',\gamma))$.
\end{corollary}

\begin{proof}[Proof of Proposition \ref{prop-long-exact-seq-Z}]
\textbf{Exactness at $\hhxz$} is just injectivity, which is obvious. \textbf{Exactness at $\hhx$} follows the usual arguments: given a $\Gamma$-equivariant map $X\to Z$, its composition with the inclusion in $G$ and the quotient $G\to G/Z$ is the trivial map. Conversely, a $\Gamma$-equivariant map $X\to G$ with trivial image in $\hhxgz$ must come from a $\Gamma$-equivariant map $X\to Z$.

\textbf{Exactness at $\hhxgz$.} Given an element ${\zerocochain}\in\hhx$ with image $\overline {\zerocochain}\in\hhxgz$, we know that $(1,1)\cdot\overline {\zerocochain}$ is the class of $(1,1)\cdot {\zerocochain}=({\zerocochain}_{i}^{-1}  {\zerocochain}_{j},{\zerocochain}_i^{\gamma-1}\tg^{-1}({\zerocochain}_{i}))=(1,1)$, where ${\zerocochain}_{i}$ is the restriction of ${\zerocochain}$ to $U_i$ for each $i\in I$. Therefore, for each $\overline {\zerocochain}'\in\hhxgz$,
$$\delta(\overline{\zerocochain}\overline {\zerocochain}')=((1,1)\cdot {\zerocochain})\cdot\overline {\zerocochain}'=\delta(\overline {\zerocochain}').$$
For the converse, it is enough to prove that the kernel of $\delta$ is the image of $\hhx$ since, if this is true, $\delta(\overline{{\zerocochain}})=\delta(\overline{{\zerocochain}}')$ implies that $\overline{{\zerocochain}}\overline{{\zerocochain}}'^{-1}$ is in the image of $\hhx$ for every $\overline{{\zerocochain}}$ and $\overline{{\zerocochain}}'\in\hhxgz$. Consider a $\Gamma$-equivariant function $\overline {\zerocochain}:X\to G/Z$ such that $\delta(\overline {\zerocochain})$ is the trivial class. Then there exists a $\Gamma$-invariant open cover $\UUU$, a lift ${\zerocochain}=({\zerocochain}_i)\in\cu 0$ of $\overline {\zerocochain}$ and an element $z\in\cuz 0$ such that $((1,1)\cdot {\zerocochain})\cdot z=(1,1)$. Hence we conclude that
 ${\zerocochain}z$ is an element of $\hhu$ whose image in $\hhxgz$ coincides with $\overline {\zerocochain}$.

\textbf{Exactness at $\hxz$.} Let $\UUU$ be a countable $\Gamma$-invariant open cover and consider two elements $({\onecochain},{\action})$ and $({\onecochain}',{\action}')\in\zuz$ such that there exists ${\zerocochain}\in\cu 0$ satisfying $({\onecochain},{\action})=({\onecochain}',{\action}')\cdot {\zerocochain}$. Had we proven that the image $\overline {\zerocochain}$ of ${\zerocochain}$ in $\cugz 0$ is a $\Gamma$-equivariant map $X\to G/Z$ then the class of $({\onecochain},{\action})$ would be equal to the class of $({\onecochain}',{\action}')\cdot\overline {\zerocochain}$, as required. But the equation ${\zerocochain}_i^{-1}{\onecochain}_{ij}'{\zerocochain}_j={\onecochain}_{ij}$ on $U_{ij}$ implies that the compositions of ${\zerocochain}_i\vert_{U_{ij}}$ and ${\zerocochain}_j\vert_{U_{ij}}$ with $G\to G/Z$ are the same. Similarly, the equation ${\zerocochain}_{i}^{\gamma-1}{\action}'_{\gamma,i}\tg^{-1}({\zerocochain}_{i})=\fgg i$ implies that the compositions of ${\zerocochain}_{i}$ and $\tg^{-1}(\cgamma {\zerocochain}_{i})$ with $G\to G/Z$ are the same. Conversely, given two elements $({\onecochain},{\action})$ and $({\onecochain}',{\action}')\in\zuz$, $\overline {\zerocochain}\in\hhxgz$ and a lift ${\zerocochain}\in \cu 0$ of $\overline {\zerocochain}$ such that $({\onecochain}',{\action}')\cdot{\zerocochain}=({\onecochain},{\action})$, it follows by definition that the classes of $({\onecochain},{\action})$ and $({\onecochain}',{\action}')$ in $\hx$ are equal.

\textbf{Exactness at $\hx$.} First note that the action of $\hxz$ on the set $\hx$ preserves the image in $\hxgz$. Conversely, given two elements $({\onecochain},{\action})$ and $({\onecochain}',{\action}')\in\zu$ and $\overline{\zerocochain}\in\cugz 0$ such that $(\overline {\onecochain},\overline {\action})=(\overline {\onecochain}',\overline {\action}')\cdot \overline {\zerocochain}$, where overlining means taking the composition with $G\to G/Z$, we may take a lift $\zerocochain\in\cu 0$ of $\overline{\zerocochain}$ and we have $({\onecochain},{\action})=(({\onecochain}',{\action}')\cdot {\zerocochain})({\onecochain}_0,{\action}_0)$, where $({\onecochain}_0,{\action}_0)\in \cuz 0\times\Fun(\Gamma,\cuz 0)$. Applying $d_1$ on both sides we see that the fact that both $({\onecochain},{\action})$ and $({\onecochain}',{\action}')\cdot {\zerocochain}$ are in $\zu$ implies that $({\onecochain}_0,{\action}_0)\in\zuz$, as required.

Finally we prove that \textbf{the preimage of the class $[(1,1,\theta^{-1}(c))]\in H^2_{\Gamma,\theta}(X,\uz)$ is equal to the image of $\hxc$}; in particular, setting $c=1$, this implies exactness at $\hxgz$. If $({\onecochain},{\action})\in\zuc$ and $(\overline {\onecochain},\overline {\action})$ is its image in $\zugz$ then the image of the class of $(\overline {\onecochain},\overline {\action})$ under $\delta$ is represented by $d_1({\onecochain},{\action})=(1,1,\theta^{-1}(c))$. Conversely, if $(\overline {\onecochain},\overline {\action})\in\zugz$ has a lift $({\onecochain},{\action})\in C^1(U,\ug)\times\Fun(\Gamma,\cuz 0)$ and $d_1({\onecochain},{\action})$ is in the class of $(1,1,\theta^{-1}(c))$ then $d_1({\onecochain},{\action})=d_1({\onecochain}',{\action}')(1,1,\theta^{-1}(c))$ for some $({\onecochain}',{\action}')\in \cuz 1\times\Fun(\Gamma,\cuz 0)$. Hence $d_1({\onecochain}{\onecochain}'^{-1},{\action}{\action}'^{-1})=(1,1,\theta^{-1}(c))$ and $({\onecochain}{\onecochain}'^{-1},{\action}{\action}'^{-1})$ is a lift of $(\overline {\onecochain},\overline {\action})$, so that the image of the class of $({\onecochain}{\onecochain}'^{-1},{\action}{\action}'^{-1})$ in $\hx$ is equal to the class of $(\overline {\onecochain},\overline {\action})$ in $\hxgz$.
\end{proof}

\begin{remark}
A similar long exact sequence may be constructed using reduced twisted equivariant cohomology sets, namely
\begin{align*}
    1\to \hhxz\to\hhx\rgamma\to\hhxgz\rgamma\xrightarrow{\delta}\\
    \xrightarrow{\delta}\hxz\to\thx\to\thxgz\xrightarrow{\delta}\widetilde H^2_{\Gamma,\theta}(X,\uz)
\end{align*}
The group $\widetilde H^2_{\Gamma,\theta}(X,\uz)$ is the quotient of $H^2_{\Gamma,\theta}(X,\uz)$ by a suitable $Z(\Gamma)$-action, and the preimage of the class of $(1,1,\theta^{-1}(c))$ under the coboundary morphism is equal to the image of $\thxc$ in $\thxgz$.
\end{remark}

\begin{remark}
There is a generalization of the long exact sequence (\ref{eq-long-exact-seq-G-Z}) to the situation when $Z$ is replaced by a normal subgroup $H$ of $G$. When $H$ is not abelian there is no second twisted equivariant cohomology group with values in $H$, but we still have an exact sequence
\begin{align}\label{eq-long-exact-seq}
    1\to \hhxh\to\hhx\to\hhxgh\xrightarrow{\delta}\\
    \notag\xrightarrow{\delta}\hxh\to\hx\to\hxh.
\end{align}
Using reduced cohomology we also find the exact sequence
\begin{align*}
    1\to \hhxh\to\hhx\rtimes Z(\Gamma)\to\hhxgh\rtimes Z(\Gamma)\xrightarrow{\delta}\\
    \xrightarrow{\delta}\hxh\to\thx\to\thxgh.
\end{align*}
\end{remark}

%%%%%%%%%%%%%%%%%%%%%%%%%%%%%%%%%%%%%%%%%%%%%%%%%%%
\subsection{The case when $\Gamma$ acts freely.}
\label{section-nonempty-hx}
%%%%%%%%%%%%%%%%%%%%%%%%%%%%%%%%%%%%%%%%%%%%%%%%%%%%

In this section we assume that $\Gamma$ acts freely on $X$, inducing
an \'etale morphism $p:X\to Y:=X/ \Gamma$; alternatively, $p:X\to Y$ is
a principal $\Gamma$-bundle. According to Section
\ref{equivalence-of-categories}, there should be a relation between
the twisted equivariant cohomology over $X$ introduced in Section
\ref{section-definitions} and the non-abelian cohomology over $Y$
defined in Section \ref{non-abelian-cohomology}. We establish this
explicitly. We continue to drop $c$ from the notation whenever it is trivial.

Consider the bundles of groups $X(G):=X\times_{\Gamma}G$ and $X(Z):=X\times_{\Gamma} Z$ over $Y$. Their sections, which are defined using the continuous/smooth/complex structure on $G$, determine sheaves of groups as in Section \ref{nac-sheaves}. A 2-cocycle $c\in Z^2_{\theta}(\Gamma,Z)$ determines a \v{C}ech 2-cocycle in $Z^2(Y,X(Z))$ as follows: take a good open cover $\VVV=\{V_i\}$ of $Y$ trivializing $X$, and let $s_i$ be a local trivialization on $V_i$. Then we define $(c_{ijk}):=(s_i,c(\gamma_{ij},\gamma_{jk}))\in Z^2(\VVV,X(Z))$, where $\gamma_{ij}\in \Gamma$ is the unique element in $\Gamma$ such that $s_i\gamma_{ij}$ intersects $s_j$. We check that this is indeed a 2-cocycle:
\begin{align*}
    c_{ijk}c_{ikl}=(s_i,c(\gamma_{ij},\gamma_{jk}))(s_i,c(\gamma_{ik},\gamma_{kl}))=
(s_i,c(\gamma_{ij},\gamma_{jk})c(\gamma_{lk},\gamma_{ki}))\stackrel{(\ref{cocycle-condition})}{=}\\
(s_i,\theta_{\gamma_{ij}}(c(\gamma_{jk},\gamma_{kl})c(\gamma_{ij},\gamma_{jl}))=
(s_j,c(\gamma_{jk},\gamma_{kl}))(s_i,c(\gamma_{ij},\gamma_{jl}))=c_{jkl}c_{ijl}.
\end{align*}

The class of $c_{ijk}$ in $H^2(Y,X(Z))$ is independent of the class of $c$ in the second Galois cohomology group $H^2_{\theta}(\Gamma,Z)$ as defined at the end of Section \ref{sec:extensions-lie}: given a map $a:\Gamma\to Z$, the 2-cocycle $\delta a\in Z^2_{\theta}(\Gamma,Z)$ as defined in Section \ref{sec:extensions-lie} gives
$$(\delta a)_{ijk}=(s_i,\theta_{\gamma_{ij}}(a(\gamma_{jk}))a(\gamma_{ik})^{-1}a(\gamma_{ij}))=(s_j,a(\gamma_{jk}))(s_i,a(\gamma_{ij}))(s_i,a(\gamma_{ik}))^{-1},$$
which is the derivative of the \v{C}ech 1-cochain $(s_i\vert_{V_{ij}},a(\gamma_{ij}))_{ij}$. Thus we get a morphism of pointed sets $H^2_{\theta}(\Gamma,Z)\to H^2(\VVV,X(Z))$, and in particular a morphism $H^2_{\theta}(\Gamma,Z)\to H^2(Y,X(Z))$. This depends on the choice of trivializations for $X$. As long as there is no ambiguity we denote the image of $c$ in $Z^2(\VVV,X(Z))$ with the same letter.

\begin{remark}
The existence of the morphism above is not surprising: let $\eg\to\bg$ be the universal $\Gamma$-bundle, and let $Y\to\bg$ be the map inducing the $\Gamma$-bundle $X$. There is an isomorphism $H^2_{\theta}(\Gamma,Z)\cong H^2(\bg,\eg\times_{\Gamma}Z_0)$, where $Z_0$ is the group $Z$ equipped with the discrete topology; this is seen using the bar resolution (see \cite{weibel}, Section 6.7). Composing with the inclusion $\eg\times_{\Gamma} Z_0\hookrightarrow\eg\times_{\Gamma}Z$ and pulling back by $Y\to \bg$ provides a map $H^2_{\theta}(\Gamma,Z)\to H^2(Y,X(Z))$. The matter of whether this map is equal to the one above will possibly be addressed in future work.
\end{remark}

Following \cite{karoubi} we define the \textbf{$c$-twisted first cohomology set} $H^1_{c}(Y,X(G))$:\footnote{Our $c$-twisted 1-cocycles are actually $c^{-1}$-twisted 1-cocycles in \cite{karoubi}. We adopt this convention to simplify notation.} given an open cover $\VVV=(V_i)_{i\in J}$ of $Y$ trivializing $X\to Y$ (with trivializations $s_i$) and an element $c_{ijk}\in Z^2(Y,X(Z))$, we may define the subset $Z^1_c(\VVV,X(G))$ of elements $\onecochain_{ij}\in C^1(\VVV,X(G))$ satisfying
\begin{equation}\label{eq-1-cocycle-Y}
    \onecochain_{ij}\onecochain_{jk}c_{ijk}=\onecochain_{ik}
\end{equation}
on $U_{ijk}$ for every $i,j$ and $k\in J$. If $c$ is trivial this is the set of 1-cocycles defined in Section \ref{nac-sheaves}. The usual action of $C^0(\VVV,X(G))$ on $C^1(\VVV,X(G))$ preserves $Z^1_c(\VVV,X(G))$ and so we may take the quotient set $H^1_c(\VVV,X(G)):=Z^1_c(\VVV,X(G))/C^0(\VVV,X(G))$. Finally, varying $\VVV$, we may define an inductive limit denoted by $H^1_c(Y,X(G))$.

\begin{theorem}\label{prop-long-exact-seq-Y}
  There are canonical isomorphisms
  $$
  \hhx\cong H^0(Y,X(G))\;\;\; \mbox{and}\;\;\; \hx\cong H^1(Y,X(G)).
  $$
  More generally, there are (in general non-canonical) isomorphisms
  $$
  \hxc\cong H_c^1(Y,X(G)).
  $$
Moreover, the long exact sequence (\ref{eq-long-exact-seq-G-Z}) induces the long exact sequence
\begin{align}\label{eq-long-exact-seq-cohomology}
    1\to H^0(Y,X(Z))\to H^0(Y,X(G))\to H^0(Y,X(G/Z))\to H^1(Y,X(Z))\\
    \notag\to H^1(Y,X(G))\to H^1(Y,X(G/Z))\to H^2(Y,X(Z)).
\end{align}
The preimage of $[c]\in H^2(Y,X(Z))$ under the coboundary map is precisely the image of $\hxc$ under the map $\hxc\cong H^1_c(Y,X(G))\to H^1(Y,X(G/Z))$ induced by the quotient $G\to G/Z$.
In particular, $\hxc$ (or $\thxc$) is nonempty if and only if $c$ is in the image of $\hxgz$.
\end{theorem}

\begin{remark}\label{remark-comparison-long-exact-seqs}
Proposition \ref{prop-long-exact-seq-Y} and Remark \ref{cor-comparison-twisted-equivariant-long-exact-seq} imply that (\ref{eq-long-exact-seq-cohomology}) is an alternative to (\ref{eq-long-exact-seq-lifting-ghat}) when the associated $\Gamma$-bundle $X$ is fixed. It answers the question of the existence of $\hat G$-bundles $E$ over $Y$ such that $E(\Gamma)\cong X$ by considering them as lifts of $X(G/Z)$-bundles, where $E(\Gamma)$ is the extension of structure group of $E$ via the quotient $\hat G\to\Gamma$.
\end{remark}

\begin{proof}[Proof of Theorem \ref{prop-long-exact-seq-Y}]
We define the isomorphisms, leaving the rest to the reader. The first one is clear, since $\Gamma$-equivariant sections of $\ug$ over $X$ are precisely $X(G)$-sections over $Y$.

\textbf{Definition of the isomorphisms.} Consider a good cover $\VVV=(V_i)$ of $Y$ trivializing $X\to Y$, choose a system of local trivializations $(s_i)$. Let $\UUU$ be the preimage of $\VVV$ over $X$; the sections $s_i$ determine elements of $\UUU$ which we call $U_i$, assuming for simplicity $J\subset I$. Define $\zerocochain\in\cu0$ so that ${\zerocochain}_i\vert_{s_i}=1$ and ${\zerocochain}_i\vert_{s_i\cdot\gamma}=\action_{\gamma,i}$ for each $\gamma\in \Gamma$. Set $(\onecochain',\action'):=(\onecochain,\action)\cdot \zerocochain$, which is in the same class as $(\onecochain,\action)$ and satisfies that $\action'_{\gamma,i}\vert_{s_i\cdot\gamma}:=1$; indeed,
$$(\action\cdot\zerocochain)_{\gamma,i} = (\zerocochain^{\gamma-1}_i\fgg i\tg^{-1}(\zerocochain_i)) = (\action_{\gamma,i}^{-1}\fgg i)=1.$$
Moreover, since $(\onecochain',\action')\in\zu$ by construction, the set $(\action'_{\gamma,i})_{\gamma\in \Gamma,i\in J}$ determines $\action'$ by (\ref{eq-f-1-cocycle}): since $\Gamma$ acts transitively on the fibres of $p$, every element of $I$ is of the form $i\cdot\gamma$ for some $\gamma\in \Gamma$ and $i\in J$. But, for each $\gamma',\gamma\in\Gamma$ and $i\in J$, we have
\begin{equation}\label{eq-trivial-1-cocycle}
    \action'_{\gamma, i\cdot\gamma'}=\action'^{\gamma'}_{\gamma,i}=\action'_{\gamma'\gamma,i}\theta_{\gamma'\gamma}^{-1}(c(\gamma',\gamma))\tg^{-1}(\action'_{\gamma',i})^{-1}=\theta_{\gamma'\gamma}^{-1}(c(\gamma',\gamma)).
\end{equation}

In this setting the image of the class of $(\onecochain,\action)$ in $Z^1(\VVV,X(G))$ is defined to be $(s_j\vert_{V_{ij}},h_{ij}:=\onecochain'_{i\cdot\gamma_{ij},j})$, where $\gamma_{ij}$ is the unique element of $\Gamma$ such that $U_{i\cdot\gamma_{ij},j}$ is nonempty and we have also called $\onecochain'_{i\cdot\gamma_{ij},j}$ to the composition
$$V_{ij}\cong U_{i\cdot\gamma_{ij}}\cap U_j\xrightarrow{\onecochain'_{i\cdot\gamma_{ij},j}}G.$$
More explicitly, $h_{ij}=\action_{\gamma_{ij},i}^{-1}\onecochain_{i\cdot\gamma_{ij},j}$. Let us check that this is in $Z^1_c(\VVV,X(G))$: for every three elements $i,j$ and $k\in J$ we have
\begin{align*}
    (s_j,h_{ij})(s_k,h_{jk})=(s_k,\theta_{\gamma_{kj}}(h_{ij})h_{jk})=
(s_k,\theta_{\gamma_{kj}}(\onecochain'_{i\cdot\gamma_{ij},j})\onecochain'_{j\cdot\gamma_{jk},k})\stackrel{(\ref{eq-compatibility-a-f})}{=}\\
(s_k,\action_{\gamma_{jk},i\cdot\gamma_{ij}}'^{-1}\onecochain'^{\gamma_{jk}}_{i\cdot\gamma_{ij},j}\action'_{\gamma_{jk},j}\onecochain'_{j\cdot\gamma_{jk},k})\stackrel{(\ref{eq-trivial-1-cocycle})}{=}(s_k,\theta_{ik}^{-1}(c(\gamma_{ij},\gamma_{jk}))^{-1}\onecochain'^{\gamma_{jk}}_{i\cdot\gamma_{ij},j}\onecochain'_{j\cdot\gamma_{jk},k})=\\
(s_k,\onecochain'_{i\cdot\gamma_{ik},k})(s_i,c(\gamma_{ij},\gamma_{jk}))^{-1}=(s_k,h_{ik})c_{ijk}^{-1},
\end{align*}
as required.

\textbf{Well-definedness.} Now we show that the class of $(s_i,h_{ij})$ is independent of the class of $(\onecochain,\action)$. Let $\zerocochain\in\cu 0$ and consider $(\onecochain',\action')=(\onecochain,\action)\cdot\zerocochain$, with corresponding cocycle $(s_i,h'_{ij})\in C^1(\VVV,X(G))$. Without loss of generality we may assume that $\action_{\gamma,i}=\action'_{\gamma,i}=1$ for every $\gamma\in\Gamma$ and $i\in J$, $h_{ij}=\onecochain_{i\cdot\gamma_{ij},j}$ and $h'_{ij}=\onecochain'_{i\cdot\gamma_{ij},j}$. Then we have
$$1=\fgg i'={\zerocochain}_i^{\gamma-1}{\action}_{\gamma,i}\tg^{-1}({\zerocochain}_{i})={\zerocochain}_i^{\gamma-1}\tg^{-1}({\zerocochain}_{i}),$$
which implies that ${\zerocochain}_i^{\gamma}=\tg^{-1}({\zerocochain}_{i}).$ Therefore,
$$(s_j,h'_{ij})=(s_j,{\zerocochain}_{i\cdot\gamma_{ij}}^{-1} h_{ij} {\zerocochain}_{j})=(s_j,\theta_{\gamma_{ij}}^{-1}({\zerocochain}_{i})^{-1} h_{ij} {\zerocochain}_{j})=(s_i,{\zerocochain}_{i})^{-1}
(s_j,h_{ij})
(s_j, {\zerocochain}_{j}),$$
which is isomorphic to $(s_i,h_{ij})$.

The class of $(s_i,h_{ij})$ in $H^1(\VVV,X(G))$ does depend on the choice of local sections for $X$ unless $c$ is trivial: consider another system of trivializations $s_i\cdot\gamma_i$. By the previous paragraph we may assume, after changing the representative of the equivalence class in $\hxc$, that $h_{ij}=f_{i\cdot\gamma_{ij},j}$ and (\ref{eq-trivial-1-cocycle}) holds. From (\ref{eq-compatibility-a-f},\ref{eq-trivial-1-cocycle}) we have

\begin{align*}
  (s_j\cdot\gamma_j,\action_{\gamma_{i}^{-1}\gamma_{ij}\gamma_j,i\cdot\gamma_i}^{-1}\onecochain^{\gamma_j}_{i\cdot\gamma_{ij},j})
            =\\
    (s_j\cdot\gamma_j,\theta_{\gamma_{ij}\gamma_j}^{-1}(c(\gamma_i,\gamma_{i}^{-1}\gamma_{ij}\gamma_j))^{-1}\theta_{\gamma_{ij}\gamma_j}^{-1}(c(\gamma_{ij},\gamma_j))^{-1}\theta_{\gamma_j}^{-1}(\onecochain_{i\cdot\gamma_{ij},j}))=\\
    (s_i,c(\gamma_i,\gamma_{i}^{-1}\gamma_{ij}\gamma_j)c(\gamma_{ij},\gamma_j))^{-1}
    (s_j,\onecochain_{i\cdot\gamma_{ij},j}).
\end{align*}

\textbf{Surjectivity.} Let $(s_j,h_{ij})\in Z^1(\VVV,X(G))$. We want an element $(\onecochain,\action)\in\zu$ such that $\onecochain_{i\cdot\gamma_{ij},j}=h_{ij}$ and $\action_{\gamma,i}=1$ for every $\gamma\in\Gamma$ and $i$ and $j$ in $J$. Equation $(\ref{eq-f-1-cocycle})$ determines $\action$, whereas (\ref{eq-compatibility-a-f}) determines $\onecochain$.

\textbf{Injectivity.} Let $(\onecochain,\action)$ and $(\onecochain',\action')$ in $\zu$ such that the corresponding elements $(s_j,h_{ij})$ and $(s_j,h_{ij}')\in Z^1(\VVV,X(G))$ are isomorphic, say
$$(s_j,h_{ij}')=(s_i,\zerocochain_i)^{-1}(s_j,h_{ij})(s_j,\zerocochain_j).$$
We may assume that $h_{ij}'=\onecochain_{i\cdot\gamma_{ij},j}'$. Then
$$\onecochain_{i\cdot\gamma_{ij},j}'=\theta_{\gamma_{ij}}^{-1}(\zerocochain_i)^{-1}\onecochain_{i\cdot\gamma_{ij},j}\zerocochain_j.$$
Thus $\onecochain'=\onecochain\cdot\zerocochain'$, where $\zerocochain'$ is the extension of $\zerocochain$ determined by $\zerocochain'_{i\cdot\gamma}=\tg^{-1}(\zerocochain_i)$ for each $i\in J$ and $\gamma\in\Gamma$. The fact that $\action'=\action\cdot\zerocochain'$ is checked easily using (\ref{eq-trivial-1-cocycle}).

\textbf{Refinements.} Finally, it can be shown that, if we replace $\VVV$ by a good subcover, the bijections are compatible with the respective restriction maps and so they induce an isomorphism $\hx\cong H^1_c(Y,X(G))$.
\end{proof}

\begin{remark}
The $Z(\Gamma)$-action on $\hxc$ induces a $Z(\Gamma)$-action on $H^1_c(Y,X(G))$. When $c$ is trivial this is the one given by the $Z(\Gamma)$-action on the left or the right factor of $X(G)$. According to the proof of Proposition \ref{prop:morhisms-fixed-covering} the group $ Z(\Gamma)$ is equal to the group of covering transformations $H^0(Y,X(\Gamma))$. With this interpretation, the induced action is the one involved in the statement of Proposition \ref{prop-grothendieck} when $c$ is trivial and $E_{\hat G}$ is obtained from the trivial $\theta$-twisted $\Gamma$-equivariant $G$-bundle over $X$ (which implies $E_{\hat G}(G)\cong X(G)$).
\end{remark}

\begin{corollary}
The set $H^1_{\Gamma,\theta,c}(X,\uz)$ is empty unless the class of $c$ in $H^2(Y,X(Z))$ is trivial. In other words, there are no $(\theta,c)$-twisted $\Gamma$-equivariant $Z$-bundles over $X$ unless $[c]\in H^2(Y,X(Z))$ is trivial.
\end{corollary}

\begin{theorem}\label{prop-iso-karoubi-ghat-bundles}
Let $\hat G:=G\times_{\theta,c}\Gamma$ be the twisted product defined in Section \ref{sec:extensions-lie}. There is a (non-canonical unless $c$ is trivial) surjection
\begin{equation}\label{eq-iso-twisted-cohomology-Y-ghat}
  H^1_c(Y,X(G))\to \pi^{-1}(X),
\end{equation}
where $\pi:H^1(Y,\underline{\hat G})\to H^1(Y,\Gamma)$ is induced by the quotient as in Section \ref{non-abelian-cohomology}. The composition with the isomorphism of Proposition \ref{prop-long-exact-seq-Y} induces the isomorphism
\begin{equation}\label{eq-iso-twisted-equivariant-cohomology-ghat}
    \thxc=\hxc/Z(\Gamma)\cong \pi^{-1}(X)
\end{equation}
given by Propositions \ref{prop:principal-Ghat-fixed-cover} and \ref{prop-bijection-hx-bundles-reduced}.
\end{theorem}
\begin{proof}
Let $\VVV=(V_i)_{i\in J}$ be a good cover trivializing $X$ over $Y$. Consider the morphism which sends the class of $(s_i,h_{ij})\in Z_c^1(\VVV,X(G))$ to the class of $(h_{ij},\gamma_{ij})\in Z^1(Y,\hat G)$. This is well defined: we have the cocycle condition
$$(h_{ij},\gamma_{ij})(h_{jk},\gamma_{jk})=(h_{ij}\theta_{\gamma_{ij}}(h_{jk})c(\gamma_{ij},\gamma_{jk}),\gamma_{ik})=(h_{ik},\gamma_{ik}),$$
where the last equation follows from
$$(s_i,h_{ij}\theta_{\gamma_{ij}}(h_{jk})c(\gamma_{ij},\gamma_{jk}))=(s_i,h_{ij})(s_j,h_{jk})(s_i,c(\gamma_{ij},\gamma_{jk}))=(s_i,h_{ik}).$$
Moreover, given $(s_i,\onecochain_i)\in C^0(\VVV,X(G))$ we have
$$(s_i,\onecochain_i)^{-1}(s_i,h_{ij})(s_j,\onecochain_j)=(s_i,\onecochain_i^{-1}h_{ij}\theta_{ij}(\onecochain_j)),$$
whose image is equal to $\onecochain_i^{-1}(h_{ij},\gamma_{ij})\onecochain_j$.

%The image is clearly contained in $\pi^{-1}(X)$. Conversely, note that a class in $\pi^{-1}(X)$ is always represented by an element of the form $(h_{ij},\gamma_{ij})$, which is the image of $(s_i,h_{ij})$.

To finish the proof it is left to show that the composition
$$\hxc\to H^1_c(Y,X(G))\to H^1(Y,\hat G)$$
induces the isomorphism given by Proposition \ref{prop:principal-Ghat-fixed-cover}, since this implies in particular that the image of the second map is equal to $\pi^{-1}(X)$. Let $\UUU=(U_i)_{i\in I}:=p^{-1}(\VVV)$ and consider an element $(\onecochain,\action)\in \zuc$. Let $E$ be the $(\theta,c)$-twisted $\Gamma$-equivariant $G$-bundle given by Proposition \ref{prop-bijection-hx-bundles}, which is equipped with trivializations $(e_i)_{i\in I}$ such that $e_j=e_i\onecochain_{ij}$ for each $i$ and $j$ in $I$. By the proof of Proposition \ref{prop-long-exact-seq-Y} we may assume that its image in $H^1_c(Y,X(G))$ is equal to $(s_j,\onecochain_{i\cdot\gamma_{ij},j})_{i,j\in J}=(s_i,\theta_{\gamma_{ij}}(\onecochain_{i\cdot\gamma_{ij},j}))$, whose image in $H^1(Y,\hat G)$ is $(\theta_{\gamma_{ij}}(\onecochain_{i\cdot\gamma_{ij},j}),\gamma_{ij})$. Recall that Proposition \ref{prop:principal-Ghat-fixed-cover} lets us regard $E$ as a $\hat G$-bundle such that, for every $\gamma\in\Gamma$, the action of an element $(1,\gamma)\in\hat G$ coincides with the given action of $\gamma$. With this notation,
$$e_i(\theta_{\gamma_{ij}}(\onecochain_{i\cdot\gamma_{ij},j}),\gamma_{ij})=e_i(1,\gamma_{ij})\onecochain_{i\cdot\gamma_{ij},j}=(e_i\cdot\gamma_{ij})\onecochain_{i\cdot\gamma_{ij},j}=e_j.$$
This shows that the 1-cocycle $(\theta_{\gamma_{ij}}(\onecochain_{i\cdot\gamma_{ij},j}),\gamma_{ij})\in Z^1(\VVV,\hat G)$ represents the $\hat G$-bundle corresponding to $E$, as required.
\end{proof}

%%%%%%%%%%%%%%

\begin{thebibliography}{99}

\bibitem{atiyah}
M.~F.~Atiyah, {\em K-theory and reality}, Quart. Jour. Math. \textbf{17} (1966), 367--386.

\bibitem{balaji-seshadri}
  V. Balaji and C. S. Seshadri, {\em Moduli of parahoric $G$-torsors on a compact Riemann surface}, J. Algebraic
Geom., {\bf 24} (2015), 1--49.

\bibitem{barajas-garcia-prada1}
G. Barajas and O. Garc\'ia-Prada,
\emph{A Prym--Narasimhan--Ramanan construction of principal bundle fixed points}, preprint 2022 (arXiv:2211.12812).

\bibitem{barajas-basu-garcia-prada}
G. Barajas, S. Basu and O. Garc\'ia-Prada,
\emph{Finite group actions on Higgs bundle moduli spaces}, in preparation.

\bibitem{basu-garcia-prada}
S. Basu and O. Garc\'{\i}a-Prada, \emph{Finite group actions on Higgs
  bundle moduli spaces and twisted equivariant structures}, preprint 2020 (arXiv:2011.04017).

\bibitem{BCGT}
 O. Biquard, B. Collier, O.  Garc\'ia-Prada and D. Toledo, Arakelov–Milnor inequalities
and maximal variations of Hodge structure, preprint 2021 (arXiv:2101.02759).

\bibitem{BCFG} I.~Biswas, L.~Calvo, E.~Franco and O.~Garc\'{\i}a-Prada, {\it Involutions of the moduli spaces of G-Higgs
bundles over elliptic curves}, Jour. Geom. Phys. {\bf 142} (2019), 47--65.

\bibitem{BCG} I.~Biswas, L.~Calvo and O.~Garc\'{\i}a-Prada,
  {\it Real Higgs pairs and  non-abelian Hodge correspondence on a Klein surface}, Communications in Analysis and Geometry,
  to appear.

\bibitem{BG} I.~Biswas and O.~Garc\'{\i}a-Prada, {\it Anti-holomorphic involutions of the moduli spaces of Higgs bundles},
Jour. \'Ecole Poly. -- Math. \textbf{2} (2014), 35--54.

\bibitem{BGH1} I.~Biswas, O.~Garc\'{\i}a-Prada and J.~Hurtubise, {\it Pseudo-real principal Higgs bundles on compact K\"ahler
manifolds}, Ann. Inst. Fourier \textbf{64} (2014), 2527--2562.

\bibitem{BGH2}
  I. Biswas,  O. Garc\'{\i}a-Prada and J. Hurtubise,
  \emph{Pseudo-real principal $G$-bundles over a real
curve}, J. London Math. Soc. {\bf 93} (2015), 47--64.

  \bibitem{biswas-gomez}
I.~Biswas and T. G\'omez,
{\it Semistability of principal bundles on a K\"ahler
manifold with a non-connected structure group}, SIGMA {\bf 10} (2014), 013, 7 pages.


\bibitem{biswas-ramanan:1994}
I.~Biswas and S.~Ramanan, \emph{An infinitesimal study of the moduli of
 {H}itchin pairs}, J. London Math. Soc. (2) \textbf{49} (1994), 219--231.


\bibitem{borel:1991}  A. Borel, \emph{Linear algebraic groups}, Second
  edition, Graduate Texts in Mathematics, vol.~126, Springer--Verlag,
  New York, 1991.

\bibitem{BCGGO}
S.B. Bradlow, B. Collier, O. Garc\'ia-Prada, P.B. Gothen and A. Oliveira,
\emph{A general Cayley
correspondence and higher Teichmüller spaces}, preprint 2021 (arXiv:2101.09377).


\bibitem{damiolini}
C. Damiolini,
  \emph{On equivariant bundles and their moduli spaces}, preprint 2021 (arXiv:2109.08698).


\bibitem{donagi-gaitsgory}
  R.Y. Donagi and D. Gaitsgory, {\em The gerbe of Higgs bundles}, Transform. Groups {\bf 7} (2001), 109--153.


 \bibitem{garcia-prada-gothen-mundet:2009}
O.~Garc{\'\i}a-Prada, P.~B. Gothen, and I.~Mundet i~Riera,
  \emph{The Hitchin--Kobayashi correspondence, Higgs pairs
 and surface group representations}, preprint 2009 (arXiv:0909.4487).

\bibitem{GGM1}
O.~Garc{\'\i}a-Prada, P.~B. Gothen, and I.~Mundet i~Riera,
  \emph{Higgs pairs, twisted equivariant structures, and non-connected groups}, in preparation.


  \bibitem{GGM2}
O.~Garc{\'\i}a-Prada, P.~B. Gothen, and I.~Mundet i~Riera,
  \emph{Non-abelian Hodge correspondence on a compact Riemann surface for non-connected groups}, in preparation.


\bibitem{garcia-prada-ramanan} O.~Garc{\'\i}a-Prada, and S. Ramanan,
  \emph{Involutions and higher order automorphisms of Higgs bundle
    moduli spaces}, Proc.  London Math. Soc. (3) \textbf{119} (2019),
  681--732.

\bibitem{garcia-prada-wilkin}
O. Garc\'{\i}a-Prada and G. Wilkin,
  \emph{Action of the mapping class group on character varieties and Higgs
bundles}, Documenta Mathematica, {\bf 25} (2020), 841--868.

\bibitem{grothendieck}
  A. Grothendieck, \emph{A general theory of fibre spaces with structure sheaf},
  Second Edition, University of Kansas, 1958.

\bibitem{hilgert-neeb:2012}
J.~Hilgert and K.-H.~Neeb, \emph{Structure and Geometry of Lie
  Groups}, Springer Monographs in Mathematics, 2012.

\bibitem{hirzebruch} F. Hirzebruch, \emph{Topological Methods in Algebraic Geometry}, Third Edition, Springer-Verlag, 1978.

\bibitem{karoubi}
  M. Karoubi, \emph{Twisted bundles and twisted K-theory},
  Clay Mathematics Proceedings {\bf 16} (2012), 223--257.

\bibitem{kozlowski}
  A. Kozlowski, \emph{Equivariant bundles and cohomology}, Trans. Amer. Math. Soc.
  {\bf 296} (1986), 181--190.

\bibitem{serre-galois} J.-P. Serre, \emph{Galois cohomology}, First Edition, Springer-Verlag, 1997.

\bibitem{siebenthal:1956}
J.~de Siebenthal, \emph{Sur les groupes de Lie compacts non connexes}, Comment. Math. Helv. \textbf{31} (1956), 41--89.


\bibitem{tom-dieck} T. tom Dieck, {\em Faserb\"undel mit Gruppenoperation},
  Arch. Math. {\bf 20} (1969), 136--143.

\bibitem{weibel} C. A. Weibel,
\textsl{An introduction to homological algebra}, Cambridge studies in advanced mathematics 38 (Cambridge University Press, Cambridge, 1994).

\end{thebibliography}
\end{document}